\documentclass[11pt,a4paper]{amsart}
\usepackage{amssymb, amstext, amscd, amsmath}
\usepackage{url}
\usepackage{amsfonts}
\usepackage{lscape}
\usepackage{amsthm}
\usepackage{amsfonts}
\usepackage{array}
\usepackage{eucal}
\usepackage{latexsym}
\usepackage{mathrsfs}
\usepackage{textcomp}
\usepackage{verbatim}
\usepackage{tikz}

\newtheorem{thm}{Theorem}[section]

\newtheorem{claim}[thm]{Claim}

\newtheorem{cor}[thm]{Corollary}

\newtheorem{lem}[thm]{Lemma}
\newtheorem{prop}[thm]{Proposition}

\numberwithin{equation}{section}

\newcommand{\bQ}{{\mathbb{Q}}}
\newcommand{\R}{{\mathbb{R}}}

\newcommand{\bZ}{{\mathbb{Z}}}

\newcommand{\bK}{{\mathbb{K}}}

  \newcommand{\M}{{\mathcal{M}}}

\newcommand{\rank}{\operatorname{ rank}}

\newcommand{\coker}{\operatorname{coker}}

\newcommand{\sm}{\setminus}

\newcommand{\dl}{{[d-1]}}

\usetikzlibrary{patterns}
\tikzset{roundnode/.style={circle,draw=black!50,fill=black!20,inner sep=1.2pt}}
\tikzset{every loop/.style={}}
\begin{document}
%%%%%%%%%%%%%%%%%%%%%%%%%%%%%%%%%%%%%%%%%%%%%%%%%%%%%%%%%%%%%%%%%%%%%%%%%%%%%%%%%%%%%%%%%%%%%%%%%%%%%%%%%%%%%%%%%%%%%%%%%%%%%%%%%%%%%%%%%%%%%%%%%%%%%%

\title{Global rigidity of 2-dimensional linearly constrained frameworks}

\author[Hakan Guler]{Hakan Guler}
\address{Department of Mathematics, Faculty of Arts \& Sciences, Kastamonu University, Kastamonu, Turkey}
\email{hakanguler19@gmail.com}
\author[Bill Jackson]{Bill Jackson}
\address{School of Mathematical Sciences, Queen Mary
University of London, Mile End Road, London E1 4NS, UK. }
\email{b.jackson@qmul.ac.uk}
\author[Anthony Nixon]{Anthony Nixon}
\address{Department of Mathematics and Statistics\\ Lancaster University\\
LA1 4YF \\ U.K. }
\email{a.nixon@lancaster.ac.uk}
\date{\today}

\begin{abstract}
A linearly constrained framework in $\R^d$ is a point configuration together with a system of constraints which fixes the distances between some pairs of points and additionally  restricts some of the points to lie in given affine subspaces. It is globally rigid if the configuration is uniquely defined by the constraint system, and is rigid if it is uniquely defined within some small open neighbourhood.  Streinu and Theran characterised generic rigidity of
linearly constrained frameworks in $\R^2$ in 2010.
We obtain an analagous characterisation for generic global rigidity in $\R^2$. More
precisely we show that a generic linearly constrained  framework in
$\R^2$ is globally rigid if and only if it is redundantly rigid and
`balanced'. For generic frameworks which are not balanced, we determine the precise number of solutions to the constraint system whenever the underlying  rigidity matroid of the given framework is connected.
We also 
obtain a stress matrix sufficient condition 
and a Hendrickson type necessary condition for a generic linearly constrained framework to be globally rigid in $\R^d$.
\end{abstract}


\maketitle

\section{Introduction}\label{introduction}

A (bar-joint) framework $(G,p)$ in $\mathbb{R}^d$ is the combination of a
finite, simple graph $G=(V,E)$ and a  realisation $p:V\rightarrow
\mathbb{R}^d$. The framework $(G,p)$ is rigid if every edge-length
preserving continuous motion of the vertices arises as a congruence
of $\mathbb{R}^d$. Moreover $(G,p)$ is globally rigid if every framework $(G,q)$ with the same edge lengths as $(G,p)$ arises from a congruence of $\mathbb{R}^d$.

In general it is an NP-hard problem to determine the global rigidity of a given framework \cite{Sax}. The problem becomes more tractable, however, if we consider generic frameworks i.e.\ frameworks in which the set of coordinates of the points is algebraically independent over $\bQ$. Hendrickson \cite{Hen} obtained two necessary conditions for a generic framework $(G,p)$  in $\mathbb{R}^d$ to be globally rigid: the graph $G$ should be $(d+1)$-connected, and the framework $(G,p)$ should be redundantly rigid i.e.\ it remains rigid after deleting any edge. While Hendrickson's conditions are insufficient to imply generic global rigidity 
%in $\mathbb{R}^d$ 
when $d\geq 3$  \cite{C91,JKT}, they are sufficient when $d=1,2$. In particular, we have the following theorem of Jackson and Jord\'{a}n \cite{J&J} when $d=2$.

\begin{thm}\label{thm:bar-joint}
A generic framework $(G,p)$ in $\mathbb{R}^2$ is globally rigid if and only if $G$ is either a complete graph on at most three vertices or $G$ is 3-connected and redundantly rigid.
\end{thm}

A linearly constrained framework is a bar-joint framework 
 in which certain vertices are constrained to lie in given affine subspaces, in addition to the usual distance constraints between pairs of vertices. Linearly constrained frameworks are motivated by numerous 
%potential 
practical applications, notably in mechanical engineering and biophysics, see for example \cite{EJNSTW, Tetal}. Streinu and Theran \cite{ST} give a characterisation for generic rigidity of
linearly constrained frameworks in $\R^2$. Together with Cruickshank \cite{CGJN}, we recently obtained an analogous  characterisation for generic rigidity of
linearly constrained frameworks in $\R^d$ as long as the dimensions of the  affine subspaces at each vertex are sufficiently small  (compared to $d$).
In this article we consider global rigidity for linearly constrained frameworks. Global rigidity of bar-joint frameworks has its own suite of practical applications, for example in sensor network localisation \cite{JJsn}, and we expect our extension to have similar uses.

Throughout this paper we will consider graphs whose only possible
multiple edges are multiple loops. We call such a graph $G=(V,E,L)$
a \emph{looped simple graph} where $E$ denotes the set of (non-loop)
edges and $L$ the set of loops.
A {\em $d$-dimensional linearly constrained framework}  is a triple
$(G, p, q)$ where $G=(V,E,L)$ is a looped simple graph, $p:V\to
\R^d$  and $q:L\to \R^d$. For $v_i\in V$ and $e_j\in L$ we put
$p(v_i)=p_i$ and $q(e_j)=q_j$. The framework $(G,p,q)$ is {\em generic} if 
the set of coordinates of $\{p,q\}$ is algebraically independent over $\mathbb Q$ i.e.\
%In particular we have ?? 
the transcendence degree of $\mathbb{Q}(p,q)$ over $\mathbb{Q}$ is $d(|V|+|L|)$.

%We assume throughout that $d\geq 2$.

Two $d$-dimensional linearly constrained frameworks $(G,p,q)$ and
$(G,\tilde p,q)$ are {\em equivalent} if
\begin{eqnarray*}
\|p_i-p_j\|^2&=&\|\tilde p_i-\tilde p_j\|^2 \mbox{ for all $v_iv_j \in E$, and}\\
p_i\cdot q_j&=&\tilde p_i\cdot q_j \mbox{ for all incident pairs
$v_i\in V$ and $e_j \in L$.}
\end{eqnarray*}
We say that $(G,p,q)$ is {\em globally rigid} if its only equivalent
framework is itself.

We give an illustration of rigidity and global rigidity in $\R^2$ in Figure \ref{fig:globally_rigid}.
First note that a loop at a vertex constrains that vertex to lie on a specific line. Every realisation
of the graph $H$ as a generic linearly constrained framework will be globally rigid as having two different
line constraints at each vertex fixes the position of the vertices in $\R^2$. Every generic realisation
$(G,p)$ of the graph $G$ is rigid by Theorem \ref{thm:ST} below, but is not globally rigid, since we can obtain
an equivalent realisation by reflecting the vertex $v_2$ in the line through $p(v_1)$ which is perpendicular
to the line constraint at $v_2$.
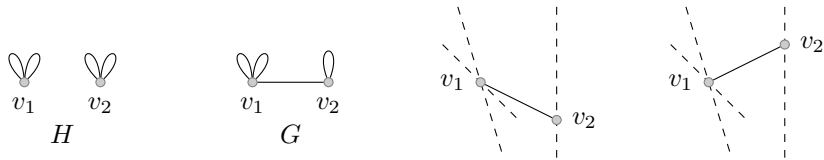
\begin{figure}[h]
\begin{center}
\begin{tikzpicture}[font=\small]
\node[roundnode] at (0,0) (v1) [label=below:$v_1$]{}
	edge[in=135,out=95,loop]()
	edge[in=85,out=45,loop] ();
\node[roundnode] at (1,0) (v2) [label=below:$v_2$]{}
	edge[in=135,out=95,loop]()
	edge[in=85,out=45,loop] ();
\node[] at (0.5,-0.3) (H) [label=below:$H$]{};
\begin{scope}[xshift=3cm]
\node[roundnode] at (0,0) (v1) [label=below:$v_1$]{}
	edge[in=135,out=95,loop]()
	edge[in=85,out=45,loop] ();
\node[roundnode] at (1,0) (v2) [label=below:$v_2$]{}
	edge[] (v1)
	edge[in=110,out=70,loop]();
\node[] at (0.5,-0.3) (G) [label=below:$G$]{};
\end{scope}
\begin{scope}[xshift=6cm]
\draw[dashed] (-0.5,0.5)--(0.5,-0.5);
\draw[dashed] (-0.3,1)--(0.3,-1);
\node[roundnode] at (0,0) (v1) [label=left:$v_1$]{};

\draw[dashed] (1,1)--(1,-1);
\node[roundnode] at (1,-0.5) (v2) [label=right:$v_2$]{}
	edge[] (v1);
\end{scope}
\begin{scope}[xshift=9cm]
\draw[dashed] (-0.5,0.5)--(0.5,-0.5);
\draw[dashed] (-0.3,1)--(0.3,-1);
\node[roundnode] at (0,0) (v1) [label=left:$v_1$]{};

\draw[dashed] (1,1)--(1,-1);
\node[roundnode] at (1,0.5) (v2) [label=right:$v_2$]{}
	edge[] (v1);
\end{scope}
\end{tikzpicture}
\end{center}
\caption{Every realisation of the graph $H$ as a generic linearly constrained framework in
$\R^2$ is globally rigid. Every generic realisation of the graph $G$ is rigid but not globally rigid.
Two distinct equivalent realisations of $G$ are given on the right of the figure.}
\label{fig:globally_rigid}
\end{figure}

Our main results characterise   global rigidity for generic linearly
constrained frameworks in $\R^2$ and  determine the precise number of frameworks which are equivalent to a given generic framework whenever the underlying  rigidity matroid of the given framework is connected. We also obtain a stress matrix sufficient condition 
%for a linearly constrained framework to be globally rigid 
and a Hendrickson type necessary condition for generic global rigidity in $\R^d$. A more detailed description of the content of the paper is as follows.

%, %The vector $q$ is \emph{generic} if any subset of $q$ with at most 3 elements is linearly independent and 
%where a framework $(G,p,q)$ is {\em generic} if 
%%$\{p,q\}$ is algebraically independent over $\mathbb Q$. 
%%In particular we have ?? 
%$td[\mathbb{Q}(p,q):\mathbb{Q}]=d(|V|+|L|)$.

In Section \ref{sec:inf} we provide a brief background on infinitesimal rigidity for linearly constrained frameworks. In Section \ref{sec:Hendrickson} we give necessary conditions for generic global rigidity analogous to Hendrickson's conditions. 
We  obtain an algebraic sufficient condition for the generic global rigidity of a linearly constrained framework  in Section \ref{sec:stress}. This sufficient condition in terms of the rank of an appropriate stress matrix extends a key result of Connelly \cite{C05} for bar-joint frameworks. 
%We combinatorially characterise the generic global rigidity for linearly constrained frameworks in $\mathbb{R}^d$ subject to the constraint that each vertex moves on a line in Section \ref{sec:line}. 
We focus on  characterising generic global rigidity  in $\R^2$ in the remainder of the paper. We obtain structural results on 
%In Section \ref{sec:circuits} we give a recursive construction of 
rigid circuits in the generic 2-dimensional linearly constrained rigidity matroid in Section \ref{sec:circuits}. These are used in Sections \ref{sec:admissible} and \ref{sec:feasible} to obtain a recursive construction for the family of looped simple graphs which are  balanced and generically redundantly rigid.
The recursive construction is then used in Section \ref{sec:globalthm} 
% These key combinatorial results are sufficient to allow us to prove our 
to characterise generic global rigidity. We determine the precise number of frameworks which are equivalent to a given generic framework whenever the underlying rigidity matroid of the given framework is connected in Section \ref{sec:count}.
% by induction 

\section{Infinitesimal rigidity}
\label{sec:inf}

An {\em infinitesimal motion} of a linearly constrained framework $(G, p, q)$ is a map $\dot
p:V\to \R^d$ satisfying the system of linear equations:
\begin{eqnarray*}
(p_i-p_j)\cdot (\dot p_i-\dot p_j)&=&0 \mbox{ for all $v_iv_j \in E$}\\
q_j\cdot \dot p_i&=&0 \mbox{ for all incident pairs $v_i\in V$ and
$e_j \in L$.}
\end{eqnarray*}
The second constraint implies that, for each vertex $v_i$,  its {\em infinitesimal velocity} $\dot p(v_i)$  
%of each $v_i\in V$ 
is constrained to lie in the intersection of the hyperplanes with normals $q_j$
for every loop $e_j$ incident to  $v_i$.

%When there are $k$ loops $e_{v_i}^1,\dots, e_{v_i}^k$ at $v$, we use
%$q_i^1,\dots,q_i^k$ to denote the corresponding vectors
%$q(e_{v_i}^1),\dots, q(e_{v_i}^k)$.
The {\em rigidity matrix $R (G, p, q)$} of the framework  is  the
matrix of coefficients of this system of equations for the unknowns
$\dot p$. Thus $R (G, p, q)$ is a $(|E|+|L|)\times d|V|$ matrix, in
which: the row indexed  by an edge $v_iv_j\in E$ has $p(u)-p(v)$ and
$p(v)-p(u)$ in the $d$ columns indexed by $v_i$ and $v_j$,
respectively and zeros elsewhere; the row indexed  by a loop
$e_j=v_iv_i\in L$ has $q_j$  in the $d$ columns indexed by $v_i$ and
zeros elsewhere.

The framework $(G, p, q)$ is {\em infinitesimally rigid} if its only
infinitesimal motion is $\dot p =0$, or equivalently if $\rank R(G,
p, q) = d|V|$. We say that the graph $G$ is  {\em rigid} in $\R^d$
if $\rank R(G, p, q) = d|V|$ for some realisation $(G,p, q)$ in
$\R^d$, or equivalently if $\rank R(G, p, q) = d|V|$ for all {\em
generic} realisations $(G,p,q)$.
% i.e. all realisations for which
%$(p,q)$ is algebraically independent over $\bQ$.

Streinu and Theran \cite{ST} characterised the looped simple graphs
$G$ which are rigid in $\mathbb{R}^2$.
%in which some joints were constrained to move on fixed lines. In particular they proved the following %theorem.
Given a looped simple graph  $G=(V,E,L)$ and 
% For $X\subseteq V$, let $i_{E\cup L}(X)$ denote the number of edges in the subgraph induced by $X$ in $G$ and put $i_{E\cup L}(X)=i_E(X)+i_L(X)$.
%The graph $G$ is \emph{$(k,\ell)$-sparse} if $i(X)\leq k|X|-\ell$
%holds for all $X\subseteq V$ with $|X|\geq k$. A $(k,\ell)$-sparse
%graph $G=(V,E)$ is \emph{$(k,\ell)$-tight} if $|E|=k|V|-\ell$.
$F\subseteq E\cup L$, let $V_F$ denote the set of vertices incident to $F$.

\begin{thm}\label{thm:ST}
Let $H$ be a looped simple graph. Then $H$ is rigid in $\R^2$ if and
only if $H$ has a spanning subgraph $G=(V,E,L)$ such that
$|E|+|L|=2|V|$, $|F|\leq 2|V_F|$ for all $F\subseteq E\cup L$ and $|F|\leq 2|V_F|-3$ for all $\emptyset \neq F\subseteq E$.
\end{thm}

\section{Necessary conditions for global rigidity}
\label{sec:Hendrickson}

We say that a looped graph $G=(V,E,L)$ is \emph{redundantly rigid} if $G-e$ is rigid for any $e\in E\cup L$,
and that  $G$ is {\em $d$-balanced}
if, 
%every connected component has at least $d$ loops and 
for all $X\subset V$ with $|X|= d$, 
each connected component of $G-X$ has at least 
%$d+1-|X|$ loops. 
one loop. We will show 
%in this section 
that the properties of being redundantly rigid and $d$-balanced are necessary conditions for a connected generic linearly constrained framework with at least two vertices  to be globally rigid in $\mathbb{R}^d$. (Note that a framework with one vertex is globally rigid if and only if it is rigid, and that a disconnected framework is globally rigid if and only if each of its connected components is globally rigid.)

Given a linearly constrained framework $(G,p,q)$ in $\mathbb{R}^d$ we define its \emph{configuration space} $C(G,p,q)$ to be the set
$$C(G,p,q)=\{\hat p \in \mathbb{R}^{d|V|}\,:\,\mbox{$(G,\hat p,q)$ is equivalent to $(G,p,q)$}\}.$$

In order to establish that globally rigid linearly constrained frameworks are also redundantly rigid an important step is to prove that the configuration space is compact. Since it is easy to see that the configuration space is closed this will follow from the following lemma.

\begin{lem}\label{lem:bounded}
Let $(G,p,q)$ be a generic linearly constrained framework in $\mathbb{R}^d$. Then $C(G,p,q)$ is bounded if and only if each connected component of $G$ contains at least $d$ loops.
\end{lem}

\begin{proof}
Let $H$ be a connected component of $G$.

Suppose $H$ does not contain $d$ loops. Let $W$ be subspace of $\mathbb{R}^d$ spanned by the vectors $q(f)$ for $f\in L(H)$. Choose $0\neq t\in W^\perp$ and define $(G,p',q)$ by putting $p'(v)=p(v)+t$ for all $v\in V(H)$, and $p'(v)=p(v)$ for all $v\in V(G)\setminus V(H)$. Then $(G,p',q)$ is equivalent to $(G,p,q)$ and since we can choose $t$ to be arbitrarily large, $C(G,p,q)$ is not bounded.

Suppose $H$ contains $d$ loops $f_1,f_2,\dots,f_d$. Let $B:=\sum_{xy\in E(H)}{|p(x)-p(y)|}$ and choose $v\in V(H)$. Then the fact that $H$ is connected implies that $|p(v)\cdot q(f_i)|\leq B$ for all $1\leq i \leq d$.
Since $q$ is generic we have ${\bf e_1}=(1,0,\dots,0)=\sum_{i=1}^d \alpha_i q(f_i)$ for some scalars $\alpha_1,\alpha_2,\dots,\alpha_d$. Hence 
$$|p(v)\cdot {\bf e_1}|=\left|\sum_{i=1}^d \alpha_i p(v)\cdot q(f_i)\right|\leq B\sum_{i=1}^d |\alpha_i| .$$
A similar argument shows that $|p(v)\cdot  {\bf e_j}|$ is bounded for all vectors ${\bf e_j}$ in the standard basis for $\mathbb{R}^d$. Since $v$ is arbitrary, $C(H,p|_H,q|_H)$ is bounded. 

We may apply the same argument to each connected component of $G$ to deduce that $C(G,p,q)$ is bounded.
\end{proof}

\begin{thm}\label{thm:glob_nec}
Suppose $(G,p,q)$ is a generic globally rigid linearly constrained
framework in $\R^d$. Then each connected
component of $G$ is either a single vertex with at least $d$ loops or is $d$-balanced  and redundantly
rigid in $\R^d$.
\end{thm}

\begin{proof}
Since $(G,p,q)$ is globally rigid if and only if each of its connected components are globally rigid, we may assume that $G$ is connected and is rigid in $\R^d$. It is easy to see that the theorem holds when $G$ has one vertex so we may assume that $|V|\geq 2$.

We first prove that $G$ is $d$-balanced. Let $X\subseteq V$ with $|X|=d$. Suppose some connected component  $G_1$ of $G-X$,
%but not $G$ 
is incident with no loops.
Then we can obtain an equivalent but noncongruent realisation 
from $(G,p,q)$ by reflecting $G_1$ in the hyperplane spanned by the points $p(v)$, $v\in X$.
% where $S_{L_1}$ is the $|L_1|$ dimensional (affine) subspace spanned by the 1-dimensional (affine) subspaces that are complements of the hyperplanes corresponding to the loops in
%$L_1$. 
This contradicts the global rigidity of $(G,p,q)$. Hence 
%$d\leq |X|+|Y|-1\leq |X|+|L_1|-1$.
$G$ is $d$-balanced.

We next show that $G$ is redundantly
rigid in $\R^d$. Suppose not. Then there exists $e\in E\cup L$ such that $(G-e,p,q)$ is flexible. 

We will use Lemma \ref{lem:bounded} to show that the configuration space  $C(G-e,p,q)$ is bounded. The fact that $G$ is rigid in $\R^d$ implies that each connected component $H$ of $G-e$ with $n\geq 2$ vertices  
contains at least $d$ loops. (This follows since the dimension of the kernel of $R(G-e,p,q)$ is one, so the dimension of the kernel of $R(H,p|_H,q|_H)$ is at most one. 
This implies that  the rank of 
$R(H,p|_H,q|_H)$ is at least $dn-1$. 
On the other hand, the rank of the submatrix of $R(H,p|_H,q|_H)$ consisting of the rows indexed by the (non-loop) edges of $H$ is at most $dn-{{d+1}\choose{2}}$ when $n\geq d$ and ${{n}\choose{2}}$ when $n<d$. This implies that the number of loops in $H$ is at least $(dn-1)-dn+{{d+1}\choose{2}}$ when $n\geq d$, and at least  $(dn-1)-{{n}\choose{2}}$ when $n<d$.) 

It remains to show that each connected component $H$ of $G-e$ with exactly one vertex $v$,
has at least $d$ loops. Suppose not. The facts that $G$ is rigid and connected imply that $H$ has exactly $(d-1)$ loops and that $e=uv$ for some $u\neq v$. 
Let $\ell$ be the line through $p (v)$ which is perpendicular to $q(f)$ for all loops incident to $v$ and let $P$ be the point on $\ell$ which is closest to $p (u)$. Let $p'(x)=p (x)$ for all $x \in V-v$ and $p'(v)=2P-p(v)$. Then $(G,p',q)$ is equivalent 
%but noncongruent 
to $(G,p,q)$ and $p'\neq p$. This contradicts the fact that $(G,p,q)$ is globally rigid. Hence $C(G-e,p,q)$ is bounded. 

We can now use a similar argument to that given in \cite{JMN} to deduce that 
%$C(G-e,p,q)$ is a 1-dimensional manifold. 
%Since $C(G-e,p,q)$ is bounded,
%Let $\mathcal{C}$ be 
the component $\mathcal{C}$ of $C(G-e,p,q)$ that contains $p$ 
%
%Since $G$ is $d$-balanced, every connected component of $G-e$ contains at least $d$ loops. Lemma \ref{lem:bounded} now implies that $\mathcal{C}$ is bounded and hence 
%
is diffeomorphic to a circle, and that
%. We can now use a similar argument to that given in \cite{JMN} to find a 
there exists a $p'\in \mathcal{C}-p$ with $(G,p',q)$ equivalent %but not congruent 
to $(G,p,q)$.
\end{proof}

We will show in Section \ref{sec:globalthm} that the necessary conditions for generic global rigidity given in Theorem \ref{thm:glob_nec} are also sufficient when $d=2$.

\section{Equilibrium stresses}
\label{sec:stress}
We will obtain an algebraic sufficient condition for a generic linearly constrained framework in $\R^d$ to be globally rigid, and show that the property that this condition holds is preserved by the graph `1-extension operation'. These results are  key  to our characterisation of global rigidity for 2-dimensional generic frameworks.

An \emph{equilibrium stress} for a linearly constrained  framework
$(G,p,q)$ in $\mathbb{R}^d$ is a pair $(\omega,\lambda)$, where
$\omega:E\to \R$, $\lambda:L\to \R$ and $(\omega,\lambda)$ belongs
to the cokernel of $R(G,p,q)$. Thus $(\omega,\lambda)$ is an
equilibrium stress for $(G,p)$ in $\R^d$ if and only if, for all
$v_i\in V$,
\begin{equation}\label{eq:stressdefn}
\sum_{v_j\in V} \omega_{ij}(p(v_i)-p(v_j)) + \sum_{e_j\in
L}\lambda_{i,j} q(e_j)=0.
\end{equation}
where $\omega_{ij}$ is taken to be equal to $\omega(e)$ if
$e=v_iv_j\in E$ and to be equal to $0$ if $v_iv_j\not\in E$, and $\lambda_{ij}$ is  equal to $\lambda(e_j)$ if
$e_j$ is a loop at $v_i$ and is equal to $0$ otherwise. An example is presented in Figure \ref{fig:stress}.

\begin{figure}[h]
\begin{center}
\begin{tikzpicture}[scale=1.5,font=\small]
\node[roundnode] at (1,0) (a) [label=below:$v_1$]{}
	edge[in=0,out=320,loop] node[pos=.5,below]{$l_1$} ()
	edge[in=40,out=0,loop] node[pos=.5,above]{$l_2$} ();
\node[roundnode] at (-1,0) (b) [label=below:$v_2$]{}
	edge[] node[pos=.5,above]{$e_1$} (a) 
	edge[in=180,out=220,loop] node[pos=.7,above]{$l_3$}();
\node[roundnode] at (0,1) (c) [label=left:$v_3$]{}
	edge[] node[pos=.5,right]{$e_2$}(a)
	edge[] node[pos=.5,left]{$e_3$}(b)
	edge[in=70,out=110,loop] node[pos=.7,right]{$l_4$}();
\node[] at (0,0) (G) [label=below:$G$] {};
\node[] at (2,1.3) () [label=right:$p(v_1)\text{=}(1\text{,}0)$]{};
\node[] at (2,1.05) () [label=right:$p(v_2)\text{=}(-1\text{,}0)$]{};
\node[] at (2,0.8) () [label=right:$p(v_3)\text{=}(0\text{,}1)$]{};
\node[] at (2,0.55) () [label=right:$q(l_1)\text{=}(1\text{,}-1)$]{};
\node[] at (2,0.30) () [label=right:$q(l_2)\text{=}(1\text{,}0)$]{};
\node[] at (2,0.05) () [label=right:$q(l_3)\text{=}(1\text{,}1)$]{};
\node[] at (2,-0.2) () [label=right:$q(l_4)\text{=}(0\text{,}1)$]{};

\begin{scope}[xshift=6cm,yshift=0.5cm]
\draw[dashed] (0.3,-0.7) -- (1.7,0.7) node[pos=.9,below right]{$l_1$};
\draw[dashed] (1,-0.7) -- (1,0.7) node[pos=0,right]{$l_2$};
\draw[dashed] (-1.7,0.7) -- (-0.3,-0.7) node[pos=-.1,right]{$l_3$};
\draw[dashed] (-0.7,1) -- (0.7,1) node[pos=1,above]{$l_4$};
\node[roundnode] at (1,0) (a) [label=right:$v_1$]{};
\node[roundnode] at (-1,0) (b) [label=below:$v_2$]{}
	edge[] node[pos=.5,above]{$e_1$} (a);
\node[roundnode] at (0,1) (c) [label=above:$v_3$]{}
	edge[] node[pos=.5,right]{$e_2$}(a)
	edge[] node[pos=.5,left]{$e_3$}(b);
\node[] at (0,0) (Gpq) [label=below:$(G\text{,}p\text{,}q)$] {};
\end{scope}
\end{tikzpicture}
\end{center}
\caption{A looped simple graph $G$ and the corresponding realisation as a linearly constrained framework in $\mathbb{R}^2$. The vectors $\omega=(0,1,1)$ and $\lambda=(-1,0,1,-2)$ give an equilibrium stress for this framework.
%For the framework $(G,p,q)$ drawn above the vectors $\omega=(0,1,1)$ and $\lambda=(-1,0,1,-2)$ give
%an equilibrium stress.
}
\label{fig:stress}
\end{figure}
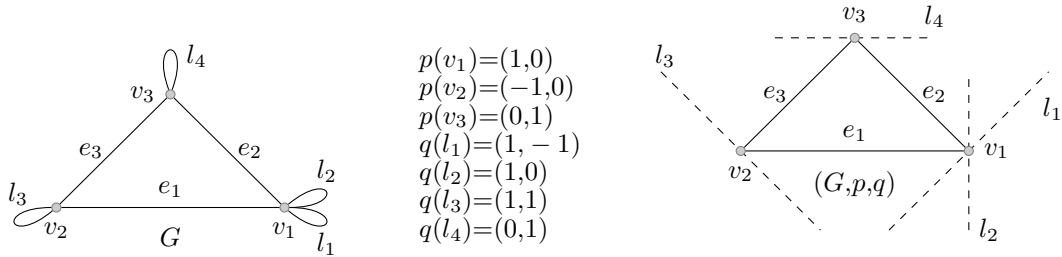
%and there are $t_i$ loops on the vertex $i$.

%For each $e\in L$ let $q_e=(q_e^1,q_e^2,\ldots,q_e^d)$. Then 
We can write (\ref{eq:stressdefn}) in matrix form
\begin{equation}\label{eq:stressmatrix}
\Pi(p)\Omega(\omega) + \Pi(q)\Lambda(\lambda)=0.
\end{equation}
 where: 
 \begin{itemize}
\item the
{\em stress matrix} $\Omega(\omega)$ is a $|V|\times |V|$-matrix in which the off
diagonal entry in row $v_i$ and column $v_j$ is $-\omega_{ij}$, and the
diagonal entry in row $v_i$  is $\sum_{v_j\in V} \omega_{ij}
$;
\item the 
{\em linear constraint matrix} $\Lambda(\lambda)$ is a $|V|\times |L|$-matrix in which the entry in row $v_i$ and column $e_j$ is $\lambda_{ij}$;
\item for any set $S=\{s_1,s_2,\ldots,s_t\}$ and map $f:S\to \R^d$, the {\em coordinate matrix} $\Pi(f)$ is the $d\times t$ matrix in which the $i$'th column is $f(s_i)$.
\end{itemize}

\begin{lem}\label{lem:stress}
Suppose $(\omega,\lambda)$ is an equilibrium stress for a $d$-dimensional linearly constrained framework $(G,p,q)$. Then 
%$\Pi (p)\,\Omega(\omega,\lambda)=0$ and 
$\rank \Omega(\omega)\leq |V|-1$. In addition, if $(G,p)$ is generic and $G$ has at least two vertices and at most $d-1$ loops, then $\rank \Omega(\omega)\leq |V|-2$.
\end{lem}

\begin{proof} 
The first assertion follows from the fact that the vector $(1,1,\ldots,1)$ belongs to the cokernel of $\Omega(\omega)$. To prove the second assertion, we choose a nonzero vector $q_0\in \R^d$ such that $q_0$ is orthogonal to $q(f)$ for all loops $f$ of $G$. Then $(p(v_1)\cdot q_0,p(v_2)\cdot q_0,\ldots,p(v_n)\cdot q_0)$ belongs to the cokernel of $\Omega(\omega)$ by Equation  (\ref{eq:stressdefn}), and is linearly independent from $(1,1,\ldots,1)$ since $(G,p,q)$ is generic.
%Let $V(G):=\{v_1,v_2,...,v_n\}$ and $p(v_i):=(x_i^1,x_i^2,...,x_i^d)$
%for all $1\leq i \leq n$. 
%The fact that $\Pi (p)\,\Omega(\omega,\lambda)=0$ follows immediately from the hypothesis that $(\omega,\lambda)$ is an
%equilibrium stress for $(G,p,q)$.
%This implies that each row of $\Pi(p)$ belongs to the cokernel of $\Omega(\omega,\lambda)$.
%and hence that $\ddim \coker \Omega(\omega,\lambda)\geq 
%\rank \Pi(p)$. Thus
%\begin{equation}\label{eq:Pi}
%\begin{split}
%\rank \Omega(\omega,\lambda)&=d|V|-\ddim \coker \Omega(\omega,\lambda)\\
%							&\leq d|V|- \rank \Pi(p)\\
%							&\leq d|V|- d.
%\end{split}
%\end{equation}
\end{proof}

We say that $(\omega,\lambda)$ is a {\em full rank equilibrium stress} for
$(G,p,q)$ if $\rank \Omega(\omega)= |V|-1$.
For example, the framework $(G,p,q)$ drawn in Figure \ref{fig:stress} has a stress matrix
\[\Omega(\omega)=\left[\begin{array}{rrr}
1&0&-1\\
0&1&-1\\
-1&-1&2
\end{array}\right]\]
with respect to the given equilibrium stress $(\omega,\lambda)=(0,1,1,-1,0,1,-2)$ and we have $\rank \Omega(\omega)=2$.
Thus $(\omega,\lambda)$ is a full rank equilibrium stress for the framework $(G,p,q)$ drawn in Figure \ref{fig:stress}. We will show in Theorem \ref{thm:stress} below that the property of having a full rank equilibrium stress is a sufficient condition for a generic linearly constrained framework 
%with at least $d$ hyperplane constraints 
to be globally rigid in $\R^d$.

\begin{lem}\label{lem:stress_for_both}
Let $(G,p,q)$ and $(G,p',q)$ be frameworks and $(\omega,\lambda)$ be a full rank
equilibrium stress for both $(G,p,q)$ and $(G,p',q)$. 
%Suppose that
%$\rank \Omega(\omega)=|V(G)|-1$.
%\rank \Pi(G,p,q)$. 
Then, for some fixed
$t\in \R^d$, we have $p(v_i)=t+p'(v_i)$ for all $v_i\in V(G)$.
\end{lem}

\begin{proof}
Since $\rank \Omega(\omega)=|V(G)|-1$, the cokernel of $\Omega(\omega)$ is spanned by $(1,1,\ldots,1)$.
Since $(\omega,\lambda)$ is an
equilibrium stress for both $(G,p,q)$ and $(G,p',q)$, we may use (\ref{eq:stressmatrix}) to deduce that $(\Pi(p)-\Pi(p'))\Omega(\omega) =0$. This implies that each row of $\Pi(p-p')$ belongs to $\coker  \Omega(\omega)$ and hence is a scalar multiple of $(1,1,\ldots,1)$. This gives $p(v_i)=t+p'(v_i)$ for all $v_i\in V$, where $t=(t_1,t_2,\ldots,t_d)$ and $t_i(1,1,\ldots,1)$ is the $i$'th row of $\Pi(p-p')$.
\end{proof}

%We will say that $p'$ is an $(a_1,a_2,\ldots,a_n)$-{\em dilation} of $p$ if
%it satisfies the conclusion of Lemma \ref{lem:stress_for_both}.

\begin{lem}\label{lem:affine_image}
Let $(G,p,q)$ be a generic linearly constrained framework in $\R^d$ and $(G,p',q)$  be an equivalent framework. Suppose that $G$ has at least $d$ loops
%at least $d$ distinct vertices of $G$ are incident to loops 
and that $p(v)=p'(v)+t$ for some fixed $t\in \R^d$, for all $v\in V$. 
Then $p=p'$.
\end{lem}

\begin{proof}
Choose distinct loops $e_i\in L$ for $1\leq i \leq d$ and let $v_i$ be the vertex incident to $e_i$. Let $P$, $P'$ and $Q$ be the $d\times d$ matrices whose $i$'th columns are $p(v_i)$, $p'(v_i)$ and $q(e_i)$, respectively and let $A$ be the $d\times d$ diagonal matrix with the coordinates of $t$ on the diagonal. Then $(P-P')Q^T=0$ since $(G,p,q)$ and $(G,p',q)$ are equivalent. Since $P-P'=AJ$ where $J$ is the $d\times d$ matrix all of whose entries are 1, we have $AJQ^T=0$. Since $(G,p,q)$ is generic, $Q$ is nonsingular and hence $AJ=0$. This implies that $A=0$ and hence $t=0$.
\end{proof}

\begin{prop}[Connelly \cite{C05}]\label{prop:C}
Suppose that $f:\mathbb{R}^a\rightarrow \mathbb{R}^b$ is a function, where each coordinate is a polynomial with 
%integer 
coefficients in some finite extension $\mathbb{K}$ of $\mathbb{Q}$, $p\in \mathbb{R}^a$ is generic over $\mathbb{K}$ and $f(p)=f(\hat p)$, for some $\hat p\in \mathbb{R}^a$. Then there are (open) neighbourhoods $N_p$ of $p$ and $N_{\hat p}$ of $\hat p$ in $\mathbb{R}^a$ and a diffeomorphism $g:N_{\hat p}\rightarrow N_p$ such that for all $x\in N_q$, $f(g(x))=f(x)$, and $g({\hat p})=p$.
\end{prop}

\begin{thm}\label{thm:stress}
Suppose $(G,p,q)$ is a generic linearly constrained framework in
$\R^d$ with at least two vertices, and $(\omega,\lambda)$ is
%at least $d$ 
%%loops
%hyperplane constraints, and suppose $(\omega,\lambda)$ is 
a full rank equilibrium stress for  $(G,p,q)$.
%for $(G,p,q)$. 
Then $(G,p,q)$ is globally rigid.
\end{thm}

\begin{proof}
Let $(G,\hat p, q)$ be equivalent to $(G,p,q)$. Let $F:\mathbb{R}^{d|V|}\rightarrow \mathbb{R}^{|E|+|L|}$ be the rigidity map defined by $F(p)=(F_E(p),F_L(p))$ where $F_E:\mathbb{R}^{d|V|}\rightarrow \mathbb{R}^{|E|}$ is the usual rigidity map defined by 
$F_E(p)=(\dots,\|p(u)-p(v)\|, \dots)_{e=uv\in E}$,
and $F_L:\mathbb{R}^{d|V|}\rightarrow \mathbb{R}^{|L|}$ is defined by $F_L(p)=(\dots,p(v)\cdot q(f), \dots)_{f=vv\in L}$.

Proposition \ref{prop:C} implies that there exist neighbourhoods $N_p$ of $(p,q)$ and $N_{\hat p}$ of $(\hat p,q)$ and a diffeomorphism $g:N_{\hat p}\rightarrow N_p$ such that $g(\hat p)=p$ and for all $\bar p \in N_{\hat p}$ we have $F(g(\hat p))=F(\hat p)$. By differentiating at $\hat p$, observing that the differential of $F$ is (up to scaling) the rigidity matrix $R(G,p,q)$ and using the fact that $(\omega,\lambda)$ is an equilibrium stress for $(G,p,q)$, we have $(\omega,\lambda)R(G,\hat p, q)=(\omega,\lambda)R(G,p,q)D=0\cdot D=0$ where $D$ is the Jacobean of $g$ at $\hat p$. Hence $(\omega,\lambda)$ is an equilibrium stress for $(G,\hat p ,q)$.
We can now use Lemmas Lemma \ref{lem:stress}, \ref{lem:stress_for_both}
% to deduce that $(G,p,q)$ is a q-affine image of $(G,\hat p,q)$ and then use Lemma 
 and \ref{lem:affine_image} to deduce that $p=\hat p$. 
 %Therefore $(G,p,q)$ is globally rigid.
\end{proof}

We will prove a partial converse to Theorem \ref{thm:stress} in Section \ref{sec:globalthm} by showing that every {\em connected, 2-dimensional,} globally rigid generic framework with at least two vertices has a full rank equilibrium stress. Note that the converse to Theorem \ref{thm:stress} is false for disconnected frameworks since a framework is globally rigid if and only if each of its connected components is globally rigid, whereas no disconnected framework can have a full rank equilibrium stress (since we may apply Lemma \ref{lem:stress} to each connected component).

Our next result shows that we can apply Theorem \ref{thm:stress} whenever we can find a full rank equilibrium stress in an arbitrary infinitesimally rigid framework. 

\begin{lem}\label{lem:stress_gen} 
Suppose $G=(V,E,L)$ can be realised in $\R^d$ as an infinitesimally rigid linearly constrained framework
with a full rank equilibrium stress. Then every generic realisation of $G$ in $\R^d$ is infinitesimally rigid and has a full rank equilibrium stress which is nonzero on all elements of $E\cup L$. 
\end{lem}
\begin{proof}
Let $|V|=n$, $|E\cup L|=m$, and let $(G,p,q)$ be a realisation of $G$ in $\R^d$. Since the  entries in the rigidity matrix $R(G,p,q)$ are polynomials in $p$ and $q$, the rank of  $R(G,p,q)$ will be maximised whenever $(G,p',q')$ is generic. Hence  $(G,p,q)$ will be infinitesimally rigid whenever $(G,p,q)$ is generic.

We adapt the proof technique of Connelly and Whiteley \cite[Theorem 5]{C&W} to prove the second part of the theorem.
Since the  entries in $R(G,p,q)$ are polynomials in $p$ and $q$, and the space of equilibrium stresses of $(G,p,q)$ is the cokernel of $R(G,p,q)$, each equilibrium stress of $(G,p,q)$ can be expressed as a pair of rational functions $(\omega(p,q,t),\lambda(p,q,t))$ of $p$, $q$ and $t$, where $t$ is a vector of $m-dn$ indeterminates. This implies that the entries in the corresponding stress matrix $\Omega(\omega(p,q,t))$ will also be rational functions of $p,q$ and $t$. Hence the rank of
$\Omega(\omega(p,q,t))$ will be maximised whenever $p,q,t$ is algebraically independent over $\bQ$. In particular, for any generic $p,q\in \R^{dn}$,  we can choose $t\in \R^{m-dn}$ such that
 $\rank \Omega(\omega(p,q,t))=dn-1$ and hence $(\omega(p,q,t),\lambda(p,q,t))$ will be a  full rank equilibrium stress for $(G,p,q)$.

Now suppose that $(G,p,q)$ is generic and that  $(\omega,\lambda)$ is a full rank stress for $(G,p,q)$ chosen such that the total number of edges $e\in E\cup L$ with $\omega_e=0$ and loops $\ell\in L$ with $\lambda_\ell=0$ is as small as possible. We may assume  that, for some $f\in E\cup L$, we have
$\omega_f=0$ if $f\in E$ and $\lambda_f=0$ if $f\in L$. Then $(\omega|_{E-f},\lambda|_{L-f})$ is a full rank equilibrium stress for $(G-f,p,q)$.  By Theorem \ref{thm:stress}, $(G-f,p,q|_{L-f})$ is globally rigid. In particular $(G-f,p,q|_{L-f})$ is rigid,  and hence, since $(G-f,p,q|_{L-f})$ is generic, it is infinitesimally rigid. This implies that the row of $R(G,p,q)$ indexed by $f$ is contained in a minimal linearly dependent set of rows, and gives us an equilibrium stress $(\hat \omega, \hat \lambda)$ for $(G,p,q)$ with $\hat \omega_f\neq 0$ when $f\in E$, and $\hat \lambda_f\neq 0$ when $f\in L$. Then
$(\omega',\lambda')=(\omega,\lambda)+c(\hat\omega,\hat\lambda)$ is an equilibrium stress for $(G,p)$, for any $c\in \R$. We can now choose
a sufficiently small $c>0$ so that  $\rank  \Omega(\omega')= n-1$, $\omega'_e\neq 0$ for all $e\in E$ for which $\omega_e\neq 0$, and $\lambda'_\ell\neq 0$ for all $\ell\in L$ for which $\lambda_\ell\neq 0$.
This contradicts the choice of $(\omega,\lambda)$.
\end{proof}
Theorem \ref{thm:stress}, Lemma \ref{lem:stress_gen} and the fact that the framework $(G,p,q)$
drawn in Figure \ref{fig:stress} is infinitesimally rigid and has a full rank equilibrium stress 
%$(\omega,\lambda)=(0,1,1,-1,0,1,-2)$ we obtain 
imply that every generic realisation of $G$
as a linearly constrained framework in $\R^2$ is globally rigid. A similar argument will allow us to show that the `1-extension operation' preserves the property of having a full rank equilibrium stress.

Let $G=(V,E,L)$ be a looped simple graph. The
\emph{$d$-dimensional $1$-extension operation} forms a new looped simple graph
from $G$ by deleting an edge or loop $e\in E\cup L$ and adding a new vertex
$v$ and $d+1$ new edges or loops incident to $v$, 
%which may or may not be loops, 
with the provisos that each end vertex of $e$ is incident to
exactly one new edge, and, if $e\in L$, then there is at least one new loop incident to $v$.
See Figure \ref{fig:1-ext} for an illustration of the types of 1-extension we will use.
%(We may want to weaken the latter proviso if we consider, e.g., plane-constrained frameworks...)

\begin{figure}[h]
\begin{center}
\begin{tikzpicture}[scale=0.8, font=\small]
\node[] at (0,2) (G) [label=above:$G$]{};
\draw[] (0,2) circle (1cm);
\node[roundnode] at (0,1.5) (x) [label=left:$x$]{}
 edge[in=113,out=67,loop]();

\draw[] (-3,-1) circle (1cm);
\node[roundnode] at (-2.5,-1.5) (y2) []{};
\node[roundnode] at (-3,-1.5) (x2) [label=left:$x$]{};
\node[roundnode] at (-3,-2.5) (v2) [label=left:$v$]{}
 edge[] (x2)
 edge[bend right] (y2)
 edge[in=247,out=293,loop]();
\draw[->] (-1,1)--(-2,0);

\draw[] (3,-1) circle (1cm);
\node[roundnode] at (3,-1.5) (x3) [label=right:$x$]{};
\node[roundnode] at (3,-2.5) (v3) [label=left:$v$]{}
 edge[in=220,out=265,loop]()
 edge[in=275,out=320,loop]()
 edge[](x3);
\draw[->] (1,1)--(2,0);

%\draw[->] (0,0.7)--(0,0.2);

%\draw[] (3,-1) circle (1cm);
%\node[roundnode] at (3.5,-1.2) (a2) []{};
%\node[roundnode] at (2.5,-1.2) (b2) []{};
%\node[roundnode] at (3,-1.5) (x2) []{};
%\node[roundnode] at (3,-2.5) (v2) [label=left:$v$]{}
% edge[] (x2)
% edge[bend right](a2)
% edge[bend left](b2);

\begin{scope}[xshift=9cm]
\node[] at (0,2) (G) [label=above:$G$]{};
\draw[] (0,2) circle (1cm);
\node[roundnode] at (-0.3,1.5) (x) [label=above:$x$]{};
\node[roundnode] at (0.3,1.5) (u) [label=above:$u$]{}
 edge[](x);

\draw[] (-3,-1) circle (1cm);
\node[roundnode] at (-2.7,-1.5) (u2) [label=above:$u$]{};
\node[roundnode] at (-3.3,-1.5) (x2) [label=above:$x$]{};
\node[roundnode] at (-3,-2.5) (v2) [label=left:$v$]{}
 edge[] (x2)
 edge[] (u2)
 edge[in=247,out=293,loop]();
\draw[->] (-1,1)--(-2,0);

\draw[] (3,-1) circle (1cm);
\node[roundnode] at (3.8,-1.2) (z3) []{};
\node[roundnode] at (2.7,-1.5) (x3) [label=above:$x$]{};
\node[roundnode] at (3.3,-1.5) (u3) [label=above:$u$]{};
\node[roundnode] at (3,-2.5) (v3) [label=left:$v$]{}
 edge[bend right] (z3)
 edge[] (u3)
 edge[](x3);
\draw[->] (1,1)--(2,0);
\end{scope}
\end{tikzpicture}
\end{center}
\caption{Possible 2-dimensional 1-extensions on a loop (on the left) and
on an edge (on the right) of a graph $G$.}
\label{fig:1-ext}
\end{figure}
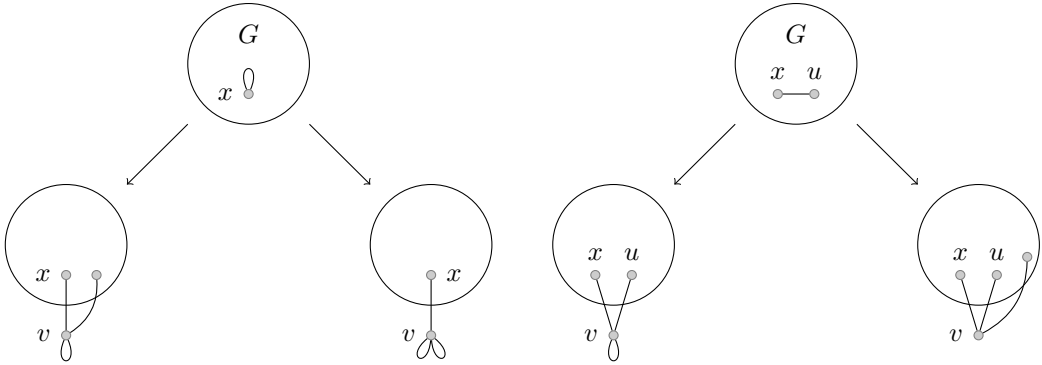

%\begin{lem}\label{lem:stress_ext} 
%Suppose $G$ is obtained from $H=(V,E,L)$ by a
%$d$-dimensional $1$-extension. 
%%which deletes an edge $e\in E\cup L$. 
%Suppose
%further that $H$ has a realisation as a linearly constrained framework
%in $\R^d$ which is infinitesimally rigid and has a full rank
%equilibrium stress.
%% $(\omega,\lambda)$. 
%%such that $w(e)\neq 0$ when $e\in E$ and $\lambda(e)\neq 0$ when $e\in L$. 
%Then $G$ has a
%realisation as a linearly constrained framework in $\R^d$ which is
%both infinitesimally rigid and has a full rank equilibrium stress.
%\end{lem}

\begin{thm}\label{thm:stress_ext} 
Let $H$ be a looped simple graph and $G$ be obtained from $H$ by either adding an edge or a loop, or by applying the $d$-dimensional $1$-extension operation. Suppose that $H$ can be realised in $\R^d$ as an infinitesimally rigid linearly constrained framework with a full rank equilibrium stress.  Then $G$ can be realised in $\R^d$ as an infinitesimally rigid linearly constrained framework
with a full rank equilibrium stress. 
\end{thm}

\begin{proof} 
Let $(H,p,q)$ be a generic realisation of $H$ as a $d$-dimensional linearly constrained framework. By Lemma \ref{lem:stress_gen}, $(H,p,q)$ has a full rank equilibrium stress $(\omega,\lambda)$ which is nonzero on all elements of $E\cup L$. 

Suppose $G$ is obtained from $H$ by adding an edge or loop $f$. Let
 $(G,p,q')$ be the realisation of $G$ obtained from $(H,p,q)$ by assigning an arbitrary value to $q'(f)$ if $f$ is a loop.
Then $(G,p,q')$ is infinitesimally rigid and the equilibrium stress $(\omega',\lambda')$ for  $(G,p,q')$ obtained from $(\omega,\lambda)$ by setting $(\omega',\lambda')$ to be zero on $f$ has full rank. 
Hence we may suppose that $G$ is a 1-extension of $H$.

There are two cases to consider depending on whether the $1$-extension deletes %$e\in E$ or $e\in L$. 
an edge or a loop.
In the former case we can proceed by the standard collinear triangle technique given in \cite{C05} (if a new edge is a loop then the loop is assigned stress 0). Hence we present the latter case. 

Let $V=\{v_1,v_2,\dots,v_n\}$.
Suppose that the $1$-extension deletes $f_1\in L$, where $f_1$ is a loop at $v_1$ and adds a new vertex $v_0$ with neighbours $v_1,v_2,\dots,v_{k_1}$ and loops $f_0^1,\dots,f_0^{k_2}$ at $v_0$ (where $k_1+k_2=d+1$). Let 
%$(H,p,q)$ be generic and define 
$(G,p',q')$ be defined by putting $p'(v)=p(v)$ for all $v\in V$, $q'(f)=q(f)$ for all $f\in L-f_1$, $p'(v_0)=p(v_1)+q(f_1)$ and $q'(f_0^1)=q(f_1)$ and choosing $\{q(f_0^1),\dots,q(f_0^{k_2})\}$ to be algebraically independent. 

We first show that the framework $(G+f_1-v_0v_1,p',q')$ is infinitesimally rigid. Its rigidity matrix 
$R$ 
can be constructed from $R(H,p,q)$ by adding $d$ new columns indexed by $v_0$, and $d$ new rows indexed by $v_0v_2, \dots v_0v_{k_1}, f_0^1,\dots,f_0^{k_2}$, respectively. Since $(p,q)$ is generic the $d\times d$ submatrix $M$ of $R$ with rows indexed by $v_0v_2, \dots v_0v_{k_1}, f_0^1,\dots,f_0^{k_2}$ and columns indexed by $v_0$ is nonsingular.
The fact that the new columns contain zeros everywhere except in the new rows
now gives $\rank R=\rank R(H,p,q)+d$.
By the choice of $p',q'$, the rows in $R$ corresponding to
$v_0v_1,f_1,f_0^1$ are a minimal linearly dependent set. Thus
$$\rank R(G,p',q')=
%\rank R(G+f_1,p',q')=
\rank R=\rank R(H,p,q)+d.$$
The fact that $(H,p,q)$ is infinitesimally rigid now implies that $(G,p',q')$ is infinitesimally rigid.

Let $({\omega'},{\lambda'})$ be the stress for $(G,p',q')$ defined by putting $\omega'_e=\omega_e$ for all $e \in E$, $\omega'_{v_0v_i}=0$ for all $i>1$, $\omega'_{v_0v_1}=-\lambda(f_1)$, ${\lambda'}(f)=\lambda(f)$ for all
$f \in L$, ${\lambda'}(f_0^1)=\lambda(f_1)$ and ${\lambda'}(f_0^i)=0$ for all $i>1$. 
It is straightforward to verify
that $({\omega'},{\lambda'})$ is an equilibrium stress for
$(G,p',q')$. 
%If $e=v_iv_j\in E$ then 
We let $\omega_{ij}$ denote the stress $\omega'_e=\omega_e$ on $e=v_iv_j\in E$ and 
%for convenience in the expressions below we list the stresses on $f\in L$ as $\lambda_1, \lambda_2, \dots$ (where 
put $\lambda_1=\lambda(f_1)$.
% is the stress on the loop deleted in the extension operation).
We have

\[ \Omega(\omega')=\begin{bmatrix}
-\lambda_1  & \lambda_1 & 0 & 0 & \dots & 0\\
\lambda_1 & \sum_{j\geq 1} \omega_{1j} -\lambda_1 & -\omega_{12} & -\omega_{13} & \dots & -\omega_{1n}\\
0 & -\omega_{12} & \sum_{j\geq 1} \omega_{2j}   &
-\omega_{23} & \dots & -\omega_{2n} \\ 0 & -\omega_{13} &
-\omega_{23} &  \sum_{j\geq 1} \omega_{3j} & \dots & -\omega_{3n} \\  \vdots & \vdots &
\vdots &  &  & \vdots
\end{bmatrix}. \] By  adding row 1 to row 2 and then column 1 to column 2 this reduces to
\[\begin{bmatrix}
-\lambda_1 & 0 & 0 & 0 & \dots & 0\\0 & \sum_j
\omega_{1j} & -\omega_{12} & -\omega_{13} & \dots &
-\omega_{1n}\\ 0 & -\omega_{12} & \sum_j \omega_{2j} &
-\omega_{23} & \dots & -\omega_{2n} \\ \vdots & \vdots & \vdots &
\vdots & & \vdots \\ \end{bmatrix} =
\begin{bmatrix} -\lambda_1 & 0 & \dots & 0\\ 0 & \\ \vdots & & \Omega(\omega) \\ 0 \end{bmatrix}. \]
Since $-\lambda_1\neq 0$, we have
 $\rank \Omega(\omega')=\rank \Omega(\omega)+1$. Hence $(\omega',\lambda')$ is a full rank equilibrium stress for $(G,p',q')$.
\end{proof}

\section{Circuits in the 2-dimensional linearly constrained rigidity matroid}
\label{sec:circuits}

We will focus on 2-dimensional linearly constrained frameworks in the following sections, and will suppress specific reference to the dimension.  In particular we will refer to the $2$-dimensional $1$-extension operation as a {\em $1$-extension}, to the property of being 2-balanced as {\em balanced}, and say that a graph is {\em rigid} to mean it is rigid in $\R^2$. 

This section will contain a combinatorial analysis of the simplest redundantly rigid graphs - these correspond to rigid circuits in the 
%2-dimensional 
linearly constrained generic rigidity matroid. Our results are analogous to those obtained in \cite{BJ} for bar-joint frameworks and \cite{JJdl} for direction-length frameworks.

Let $G=(V,E,L)$ be a looped graph. For $X\subset V$, we define a matroid $\M_{lc}(G)$ on $E\cup L$ by the  conditions of Theorem \ref{thm:ST}: a set $F\subset E\cup L$ is independent if $|F'|\leq 2|V_{F'}|$ for all $F'\subseteq F$, and  $|F'|\leq 2|V_{F'}|-3$ for all $\emptyset\neq F'\subseteq F\cap E$.
We will refer to subgraphs of $G$ whose edge-set is a circuit in
$\M_{lc}(G)$ as {\em $\M_{lc}$-circuits}, see Figure \ref{fig:Mlc-circuits}.
\begin{figure}[h]
\begin{center}
\begin{tikzpicture}[font=\small]
\node[roundnode] at (0,0) (v1) []{}
	edge[in=205, out=245,loop]();
\node[roundnode] at (1,0) (v2) []{}
	edge[] (v1)
	edge[in=295, out=335,loop]();
\node[roundnode] at (1,1) (v3) []{}
	edge[] (v1)
	edge[] (v2)
	edge[in=25, out=65,loop]();
\node[roundnode] at (0,1) (v4) []{}
	edge[] (v1)
	edge[] (v2)
	edge[] (v3)
	edge[in=115, out=155,loop]();
	
\begin{scope}[xshift=3cm]
\node[roundnode] at (0,0) (v1) []{};
\node[roundnode] at (1,0) (v2) []{}
	edge[] (v1);
\node[roundnode] at (1,1) (v3) []{}
	edge[] (v1)
	edge[] (v2);
\node[roundnode] at (0,1) (v4) []{}
	edge[] (v1)
	edge[] (v2)
	edge[] (v3);
\end{scope}
\begin{scope}[xshift=6cm]
\node[roundnode] at (0,0) (v1) []{}
	edge[in=205, out=245,loop]();
\node[roundnode] at (1,0) (v2) []{}
	edge[] (v1)
	edge[in=295, out=335,loop]();
\node[roundnode] at (1,1) (v3) []{}
	edge[] (v1)
	edge[] (v2)
	edge[in=25, out=65,loop]();
\node[roundnode] at (0,1) (v4) []{}
	edge[] (v1)
	edge[] (v3)
	edge[in=115, out=155,loop]();
\end{scope}
\end{tikzpicture}
\end{center}
\caption{A graph $G$ on the left and two $\M_{lc}$-circuits of $G$ in the middle and on the right.
It can be verified by Lemma \ref{lem:circuit_char} that the subgraph in the middle is a flexible $\M_{lc}$-circuit and the subgraph on the right is a rigid $\M_{lc}$-circuit.}
\label{fig:Mlc-circuits}
\end{figure}
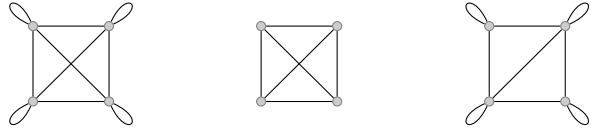
The definition of independence in $\M_{lc}(G)$ gives rise to  the following characterisation of its circuits. 
For $X\subseteq V$, let $i_{E}(X)$ and $i_{L}(X)$ denote the number of elements of $E$ and $L$ respectively in the subgraph induced by $X$ in $G$ and put $i_{E\cup L}(X)=i_E(X)+i_L(X)$.
%Theorem \ref{thm:ST} gives us the following convenient combinatorial description of $\M_{lc}(G)$-circuits.

\begin{lem}\label{lem:circuit_char}
Let $G=(V,E,L)$ be a looped simple graph. Then $G$ is an $\M_{lc}$-circuit if and only if, either\\
(a) $|E|+|L|=2|V|+1$,  $i_E(X)\leq 2|X|-3$ for all $X\subset V$ with $|X|\geq 2$ and $i_{E\cup L}(X)\leq 2|X|$ for all $X\subsetneq V$, or \\
(b) $L=\emptyset$, $|E|=2|V|-2$  and $i_E(X)\leq 2|X|-3$ for all $X\subsetneq V$ with $|X|\geq 2$.
\end{lem}

We will refer to the circuits described in (a) as {\em rigid $\M_{lc}$-circuits} and to those in (b) as {\em flexible $\M_{lc}$-circuits}. 
The smallest rigid $\M_{lc}$-circuits are the looped graphs $K_1^{[3]}$ and $K_2^{[2]}$.
The flexible $\M_{lc}$-circuits are precisely the circuits of the 2-dimensional generic  bar-joint rigidity matroid. A recursive construction for these circuits is given in \cite{BJ}. %In the remainder of this section 
We will use the results of this section to obtain  a recursive construction for rigid $\M_{lc}$-circuits in Section \ref{sec:admissible}.

Let $d^\dag(v)$ denote the number of edges or loops that are incident with a vertex $v$ in  a looped simple graph $G$, $d(v)=d^\dag(v)+i_L(v)$ denote the {\em degree} of $v$ in $G$, and $V_3=\{v\in V\,:\,d^\dag(v)=3\}$ denote the set of {\em nodes} of $G$. For disjoint $X,Y\subseteq V$, let $G[X]$ denote the subgraph of $G$ induced by $X$ and $d(X,Y)$ denote the number of edges of $G$ between $X$ and $Y$.

\begin{lem}\label{lem:forest}
Let $G=(V,E,L)$ be a rigid $\M_{lc}$-circuit. Then $V_3\neq \emptyset$  and the subgraph $G^\dag[V_3]$ of $G$ obtained by deleting all loops in  $G[V_3]$ is a forest.
\end{lem}

\begin{proof}
Since $G$ is a rigid $\M_{lc}$-circuit, $d^\dag(v)\geq 3$ for all $v\in V$,
$$4|V|+2=2|E|+2|L|=\sum_{v\in V}d(v)=\sum_{v\in V}d^\dag(v)+|L|,$$ and $|L|\geq 3$. This gives   $\sum_{v\in V}d^\dag(v)\leq 4|V|-1$ so $G$ has at least one node.

Suppose $C=(S,F)$ is  an induced 
cycle in $G^\dag[V_3]$.
Let $T=V\setminus S$. Then we have
\begin{align*}
i_{E\cup L}(T)&=2|V|+1-i_{E\cup L}(S)-d(S,T)\\
             % &=2|V|+1-i_E(C)-i_L(C)-d(C,T)\\
			  &=2|V|+1-|S|-i_L(S)-d(S,T)\\
			  &=2|T|+1,
\end{align*}
where the second equality follows from the fact that $C$ is a cycle,
and the last equality follows from the fact that a vertex in $S$
contributes $1$ to exactly one of $i_L(S)$ and $d(S,T)$ so $i_L(S)+d(S,T)=|S|$.
This contradicts the fact that $G$ is a rigid $\M_{lc}$-circuit and $T$ is
a proper subset of $V$.
\end{proof}

Let $G=(V,E,L)$ be a rigid $\M_{lc}$-circuit.  We say that $X\subseteq V$ is \emph{mixed critical} if $i_{E\cup L}(X)=2|X|$ and is \emph{pure critical} if 
%it is not mixed critical and 
$i_E(X)=2|X|-3$. We say that $X$ is {\em critical} if it is either mixed or pure critical.

\begin{lem}\label{lem:minus}
Let $X$ be a mixed critical set in  a rigid $\M_{lc}$-circuit $G=(V,E,L)$ and $Y=V\setminus X$.\\ 
(a)  Then $|V_3\cap Y|\geq 1$, with strict inequality when $|Y|\geq 2$.\\
(b) If $d(X,Y)\geq 3$ then $Y$ contains a vertex of degree three in $G$.
\end{lem}

\begin{proof}
(a) If $|Y|=1$ then the unique vertex in $Y$ must be a node of $G$. Hence we may assume  that $|Y|\geq 2$. 

Suppose $|Y\cap V_3|\leq 1$. Note that 
for a vertex $v\in Y$, we have $d(v)= d^\dag(v)+i_L(v)$ so 
$$\sum_{v\in Y}{d(v)}=\sum_{v\in Y}{d^\dag(v)}+i_L(Y)\geq 4|Y|-1+i_L(Y).$$
We also have
\begin{align*}
2i_{E\cup L}(Y)=\sum_{v\in Y}{d_{G[Y]}(v)}=\sum_{v\in Y}{d(v)}-d(Y,X).%\geq4|Y|-1+i_L(Y)-d(Y,X),
\end{align*}
We can combine these two (in)equalities to obtain
\begin{equation}\label{eq:degre_count1}
2i_E(Y)+i_L(Y)+d(Y,X)\geq 4|Y|-1.
\end{equation}
Since $|Y|\geq 2$ we have $i_E(Y)\leq 2|Y|-3$ and Equation
(\ref{eq:degre_count1}) now gives 
$$i_{E\cup L}(Y)+d(Y,X)=i_E(Y)+i_L(Y)+d(Y,X)\geq 4|Y|-1-i_E(Y)\geq  2|Y|+2.$$
The fact that $X$ is mixed critical now gives
\[
|E|+|L|=i_{E\cup L}(X)+i_{E\cup L}(Y)+d(Y,X)\geq 2|X|+2|Y|+2=2|V|+2,
\]
a contradiction.
\\[2mm]
(b) We have
$\sum_{v\in Y} d(v)=2i_{E\cup L}(Y)+d(Y,X)$ and  
$$i_{E\cup L}(Y)=|E|-i_{E\cup L}(X)-d(Y,X)=2|V\setminus X|-d(Y,X)+1.$$
This implies that
$$\sum_{v\in Y}d(v)=4|V\setminus X|-2d(Y,X)+d(Y,X)+2=4|Y|-d(Y,X)+2.$$
If $d(Y,X)\geq 3$ then $\sum_{v\in Y}d(v)<4|Y|$ and hence $Y$ contains a vertex of degree less than four.
\end{proof}

%\subsection{Admissible nodes}

We next consider two versions of the {$1$-extension operation} for a looped simple graph $G=(V,E,L)$. 
The {\em $1$-extension operation at an edge $uw\in E$}  deletes  $uw$ and adds a new vertex $v$, new edges  $uv$ and $wv$ and either a new edge $vx$ for some $x\in V\setminus \{u,w\}$ or a new loop at $v$. The {\em $1$-extension operation at a loop $uu\in L$} deletes  $uu$, adds a new vertex $v$, a new edge  $uv$ and a new  loop at $v$, and either a new edge $wv$, for some $w\in V-u$ or a second new  loop at $v$.
It is straightforward to show that both versions of the $1$-extension operation transform a rigid $\M_{lc}$-circuit into another rigid $\M_{lc}$-circuit using Lemma \ref{lem:circuit_char}(a).

We refer to the inverse operation to each of the above 1-extension operations as a \emph{1-reduction to an edge} or \emph{loop}, respectively. When  $G$ is a rigid
$\M_{lc}$-circuit and $v\in V$, we say that these reduction operations are \emph{admissible} if  they result in a smaller rigid $\M_{lc}$-circuit, and that
$v$ is \emph{admissible} if there is an admissible  1-reduction
at $v$.

%We first collect together some 
The remainder of this section will be devoted to obtaining a 
structural characterisation of `non-admissibility'. This will be used in the next section to show that every balanced rigid $\M_{lc}$-circuit on at least two vertices  contains an admissible vertex. 

\begin{lem}\label{lem:admcases}
Let $G=(V,E,L)$ be a rigid $\M_{lc}$-circuit and $v$ be a node in $G$.\\
(a) Suppose $N(v)=\{u,w,z\}$. Then the $1$-reduction at $v$ which adds $uw$ is non-admissible if and only if there exists a pure critical set $X$ with $u,w\in X$ and $v,z\notin X$ or a mixed critical set $Y$ with $u,w\in Y$ and $v,z\notin Y$.\\
%(b) Let $v$ be chosen such that $d(v)=3$ and $N(v)=\{u,w,z\}$. Then a 1-reduction at $v$ which adds $uu$ is non-admissible if and only if there exists a mixed critical set $X$ with $u\in X$, $v\notin X$ and $\{w,z\}\not\subseteq Y$.\\
(b) Suppose  $N(v)=\{u,w\}$. Then: \\
(i) the $1$-reduction at $v$ which adds $uw$  is non-admissible if and only if there exists a pure critical set $X$ with $u,w\in X$ and $v\notin X$;\\
(ii) the 1-reduction at $v$ which adds $uu$ is non-admissible if and only if there exists a mixed critical set $X$ with $u\in X$ and $w,v\notin X$.\\
(c) 
If $N(v)=\{u\}$, then the $1$-reduction at $v$ which adds $uu$ is admissible.
\end{lem}

\begin{proof}
It is straightforward to show that the existence of each of the critical sets described in (a) and (b) implies non-admissibility. 

For the converse, we first suppose that the 1-reduction described in case (a) is non-admissible. Then the graph resulting from this 1-reduction is not a rigid $\M_{lc}$-circuit. This implies that there exists  either an $X\subset V-v$ with $i_E(X)\geq 2|X|-2$ or a $Y\subsetneq V-v$ with $i_{E\cup L}(Y)\geq 2|Y|+1$. These subsets are pure critical and mixed critical, respectively, in $G$.  If the first alternative holds then $z\not \in X$, since otherwise $i_E(X+v)=i_E(X)+3=2|X|-3+3=2|X+v|-2$ and we would contradict the fact that $G$ is a rigid $\M_{lc}$-circuit. Similarly,
if the second alternative holds and $z\in Y$ then $i_{E\cup L}(Y+v)=i_{E\cup L}(Y)+3=2|Y|+3=2|Y+v|+1$, again contradicting the fact that $G$ is a rigid $\M_{lc}$-circuit.

The arguments in cases (b) and (c) are similar.
%For (e) suppose $v$ is non-admissible. Then there is a mixed critical set $X\subsetneq V-v$ containing $u$. However $i_{E\cup L}(X\cup v)=2|X\cup v|+1$, contradicting the fact that $G$ is a $\M_{lc}$-circuit.
\end{proof}

%Let $G=(V,E,L)$ be a rigid $\M_{lc}$-circuit.
%We say a node $v$ of $G$ 
%%is a 
%%{\em node} if $v\in V_3$, that 
%%a node $v$
%%, with $d(v)=3$,
% is a {\em leaf node} if $d_{G^\dag[V_3]}(v)\in \{0,1\}$, a {\em series node} if $d_{G^\dag[V_3]}(v)=2$ and a {\em branching node} if $d_{G^\dag[V_3]}(v)=3$.
%It is clear that a 1-reduction
%at a branching node will be non-admissible since the resulting graph will have  a vertex $v$ with $d^\dag(v)=2$.

\begin{lem}\label{lem:union}
Let $G$ be a rigid $\M_{lc}$-circuit. \\
(a) If $X,Y\subset V$ are pure critical with $|X\cap Y|\geq 2$, then $X\cup Y$ and $X\cap Y$ are pure critical and $d(X\sm Y,Y\sm X)=0$.\\
(b) If $X,Y\subset V$ are mixed critical and $|X\cup Y|\leq |V|-1$, then $X\cup Y$ and $X\cap Y$ are mixed critical and $d(X\sm Y,Y\sm X)=0$.\\
(c) If $X\subset V$ is mixed critical, $Y\subset V$ is pure critical, $|X\cap Y|\geq 2$ and $|X\cup Y|\leq |V|-1$, then $X\cup Y$ is mixed critical, $X\cap Y$ is pure critical and $i_L(Y\sm X)=0=d(X\sm Y,Y\sm X)$.
\end{lem}

\begin{proof}
We prove (c), parts (a) and (b) can be proved similarly.  We have
%We can write $i_E(Y)=i_E(X\cap Y)+i_E(Y\setminus X)+d(X\cap Y,Y\setminus X)$. Using the fact that $i_E(Y)=2|Y|-3$ we obtain 
\begin{eqnarray*}
2|X|+2|Y|-3 &=& i_{E\cup L}(X)+i_{E}(Y)\\
&=& i_{E\cup L}(X\cup Y)+i_{E}(X\cap Y)-i_L(Y\sm X)-d(X\sm Y,Y\sm X)\\
&\leq & 2|X\cup Y|+2|X\cap Y|-3- i_L(Y\sm X)-d(X\sm Y,Y\sm X)\\
&=& 2|X|+2|Y|-3- i_L(Y\sm X)-d(X\sm Y,Y\sm X)
%
%2|X\cap Y|+2|Y\setminus X|-3&=& 2|Y|-3=i_E(Y) \\ &=& i_E(X\cap Y)+i_E(Y\setminus X)+d(X\cap Y,Y\setminus X) \\ &\leq& 2|X\cap Y|-3+i_E(Y\setminus X)+d(X\cap Y,Y\setminus X).\end{eqnarray*}
% Hence $2|Y\setminus X|\leq i_E(Y\setminus X)+d(X\cap Y,Y\setminus X)$.  Then 
%\begin{eqnarray*}
%2|X\cup Y|\geq i_{E\cup L}(X\cup Y)&\geq& i_{E\cup L}(X)+i_E(Y\setminus X)+d(X\cap Y,Y\setminus X)+d(X\setminus Y,Y\setminus X)\\ &\geq& 2|X|+2|Y\setminus X|+d(X\setminus Y,Y\setminus X)\geq 2|X\cup Y|.
\end{eqnarray*}
 Hence $i_L(Y\sm X)=d(X-Y,Y-X)=0$ and equality holds throughout the above displayed calculation. In particular, $i_{E\cup L}(X\cup Y)=2|X\cup Y|$, $i_E(X\cap Y)=2|X\cap Y|-3$.
% and $d(X\setminus Y,Y\setminus X)=0$.
\end{proof}

\begin{lem}\label{lem:flower1}
Let $G$ be a 
%bridgeless 
rigid $\M_{lc}$-circuit and $v$ be a node of $G$ with three distinct neighbours $u,w,t$. Suppose that $X,Y$ are mixed critical sets in $G$ satisfying $\{u,w\}\subseteq X\subseteq V\setminus \{v,t\}$ and $\{w,t\}\subseteq Y\subseteq V\setminus \{v,u\}$. Suppose further that
$Z$ is a (mixed or pure) critical set with $\{u,t\}\subseteq Z\subseteq V\setminus \{v,w\}$ 
%or\\
%(ii) there exists a pure critical set $Z$ with $\{u,t\}\subseteq Z\subseteq V\setminus \{v,w\}$. \\
Let $W^\dag=(V-v)\setminus W$ for each $W\in \{X,Y,Z\}$. Then:\\
(a) $X\cup Y=X\cup Z=Y\cup Z=V-v$;\\
%(b) $X\cap Y\cap Z\neq \emptyset$;\\
(b) $d(X^\dag,Y^\dag)=d(Y^\dag,Z^\dag)=d(X^\dag,Z^\dag)=0$;\\
(c) either $\{X^\dag,Y^\dag,Z^\dag,X\cap Y\cap Z\}$ is a partition of $V-v$, or $X\cap Y\cap Z=\emptyset$ and $\{X^\dag,Y^\dag,Z^\dag\}$ is a partition of $V-v$;\\
(d) if $Z$ is pure critical then $i_L(X^\dag)=0=i_L(Y^\dag)$.
\end{lem}

\begin{proof}
Since $X,Y$ are mixed critical, $X\cup Y$ is mixed critical and $d(X\sm Y,Y\sm X)=0$ by Lemma \ref{lem:union}(b). The first assertion gives  $i_{E\cup L}(X\cup Y\cup\{v\})=2|X\cup Y\cup\{v\}|+1$. Since $G$ is an $\M_{lc}$-circuit, this imples that $X\cup Y=V-v$. We now have  $X^\dag=Y\sm X$ and $Y^\dag=X\sm Y$. When  $Z$ is mixed critical, a similar argument for $X,Z$ and $Y,Z$ tells us that (a) and (b) hold. Part (c) follows immediately from (a). Hence we may assume that  $Z$ is pure critical.

Since  $Z$ is pure critical, $G[Z]$ is connected and hence there is a path $P$ in $G[Z]$ from $u$ to $t$. If $X\cap Z=\{u\}$ then $P$ would contain no vertices of $X-u$. The existence of such a path $P$ would  contradict the fact that $u\in X\sm Y$, $t\in Y\sm X$ and $d(X\sm Y,Y\sm X)=0$. Hence $|X\cap Z|\geq 2$ and we can use   
Lemma \ref{lem:union}(c) to deduce that $X\cup Z$ is mixed critical and  $d_L^\dag(X^\dag)=0=d(X\sm Z,Z\sm X)$. A similar argument as in the previous paragraph now gives  $X\cup Z=V-v$. We can now use symmetry to deduce that $Y\cup Z=V-v$ and $d_L^\dag(Y^\dag)=0=d(Y\sm Z,Z\sm Y)$. This gives (a), (b), (c) and (d) in the case when $Z$ is pure critical.
%
%
%
%
%(a) By Lemma \ref{lem:union}(b) the statement holds when (i) occurs. Suppose (ii) occurs. We know $G[Z]$ is connected and $X\cup Y=V-v$ by Lemma \ref{lem:union}(b). If $|X\cap Z|=1$ (or similarly $|Y\cap Z|=1$), then the fact that $Z\subseteq X\cup Y$ forces $d(X-Y,Y-X)\geq 1$ and this contradicts Lemma \ref{lem:union}(b). Therefore $|X\cap Z|\geq 2$ and $|Y\cap Z|\geq 2$ so Lemma \ref{lem:union}(c) gives $X\cup Z=Y\cup Z=V-v$.
%
%(b) By (a) we have $Y\setminus (X\cup Z)=X\setminus (Y\cup Z)=Z\setminus (X\cup Y)=\emptyset$.  If $X\cap Y\cap Z=\emptyset$, then it follows that $G$ has a bridge, a contradiction.
%
%(c) If (i) holds, then Lemma \ref{lem:union}(b) gives the result. If (ii) holds, then Lemma \ref{lem:union}(b) gives $d(X,Y)=0$ and the connectivity of $G[Z]$ implies $|X\cap Z|\geq 2$ and $|Y\cap Z|\geq 2$. Then Lemma \ref{lem:union}(c) gives $d(X,Z)=0=d(Y,Z)$.
%
%(d) By (a) we have $Y\setminus (X\cup Z)=X\setminus (Y\cup Z)=Z\setminus (X\cup Y)=\emptyset$. The stated partition follows.
\end{proof}

%Following \cite{JJdl} 
We call a triple $(X,Y,Z)$ of three sets satisfying the hypotheses of Lemma \ref{lem:flower1} 
%for $G$   a {\em flower} on $v$. We say that the flower is 
a \emph{strong flower on $v$} when $Z$ is mixed critical and a \emph{weak flower on $v$} when $Z$ is pure critical. Note that it is 
possible for $(X,Y,Z)$ to be both a strong and weak flower on $v$, see Figure \ref{fig:strong_weak_flower}.

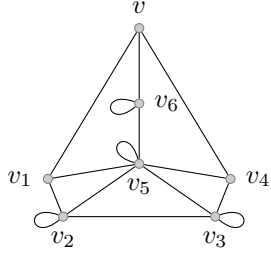
\begin{figure}[h]
\begin{center}
\begin{tikzpicture}[font=\small]
\node[roundnode] at (-1.2,0.5) (v1) [label=left:$v_1$]{};
\node[roundnode] at (-1,0) (v2) [label=below:$v_2$]{}
	edge[] (v1)
	edge[in=160,out=220,loop] ();
\node[roundnode] at (1,0) (v3) [label=below:$v_3$]{}
	edge[] (v2)
	edge[in=20,out=320,loop]();
\node[roundnode] at (1.2,0.5) (v4) [label=right:$v_4$]{}
	edge[] (v3);
\node[roundnode] at (0,0.7) (v5) [label=below:$v_5$]{}
	edge[] (v1)
	edge[] (v2)
	edge[] (v3)
	edge[] (v4)
	edge[in=105,out=165,loop] ();
\node[roundnode] at (0,1.5) (v6) [label=right:$v_6$]{}
	edge[] (v5)
	edge[in=160,out=220,loop] ();
\node[roundnode] at (0,2.5) (v) [label=above:$v$] {}
	edge[] (v1)
	edge[] (v4)
	edge[] (v6);
\end{tikzpicture}
\end{center}
\caption{Let $X=\{v_1,v_2,v_3,v_5,v_6\}$, $Y=\{v_2,v_3,v_4,v_5,v_6\}$, and  $Z=\{v_1,v_2,v_3,v_4,v_5\}$.
Then $Z$ is pure critical and mixed critical so $(X,Y,Z)$ is both a strong and a weak flower on $v$.}
\label{fig:strong_weak_flower}
\end{figure}

\begin{lem}\label{lem:flower2}
Let $G$ be a 
%bridgeless 
rigid $\M_{lc}$-circuit and $v$ be a node of $G$ with three distinct neighbours $r,s,t$. 
Suppose that $Y,Z$ are pure critical sets satisfying $\{s,t\}\subseteq Y\subseteq V\setminus \{v,r\}$ and $\{r,s\}\subseteq Z\subseteq V\setminus \{v,t\}$, and $X$ is a critical set with $\{r,t\}\subseteq X\subseteq V\setminus \{v,s\}$. Then $X$ is mixed critical, $X\cap Y=\{t\}$, $X\cap Z=\{r\}$, $Y\cap Z=\{s\}$, $G-v=G[X]\cup G[Y]\cup G[Z]$, and the component of $G-\{r,t\}$ which contains $\{v,s\}$ has no loops.
%and $G$ is not balanced.
%(a) Suppose that $X,Y$ are pure critical sets satisfying $\{u,w\}\subseteq X\subseteq V\setminus \{v,t\}$ and $\{w,t\}\subseteq Y\subseteq V\setminus \{v,u\}$. Then there is no pure critical set $Z$ with $\{u,t\}\subseteq Z\subseteq V\setminus \{v,w\}$.\\
%(b) Suppose that $X$ is a mixed critical set satisfying $\{u,w\}\subseteq X\subseteq V\setminus \{v,t\}$ and $Y,Z$ are pure critical sets satisfying $\{w,t\}\subseteq Y\subseteq V\setminus \{v,u\}$ and $\{u,t\}\subseteq Z\subseteq V\setminus \{v,w\}$. Then 
\end{lem}

\begin{proof}
%(a) is standard.
If $|Y\cap Z|\geq 2$ then $Y\cup Z$ would be pure critical by Lemma \ref{lem:union}(a) and we would have $i_E(Y\cup Z\cup \{v\})=2|Y\cup Z\cup \{v\}|-2$. This would contradict the fact that $G$ is a rigid $\M_{lc}$-circuit. Hence $Y\cap Z=\{s\}$. Since $Y$ is pure critical, $G[Y]$ is connected and hence there exists a path $P$ in $G$ from $s$ to $t$ which avoids $Z-s$.
 
Suppose $X$ is pure critical. Then a similar argument to the above gives $X\cap Z=\{r\}$ and $Y\cap X=\{t\}$,
$X\cup Y\cup Z$ is pure critical and $i_E(X\cup Y\cup Z\cup \{v\})=2|X\cup Y\cup Z\cup \{v\}|-2$. This would again contradict the fact that $G$ is a  rigid $\M_{lc}$-circuit. Hence $X$ is mixed critical.

Suppose $|X\cap Z|>1$. Then $X\cup Z$ is mixed critical and $d(X\sm Z,Z\sm X)=0$ by Lemma \ref{lem:union}(c). This gives $i_{E\cup L}(X\cup Z\cup \{v\})=2|X\cup Z\cup \{v\}|+1$ so  $X\cup Z=V-v$.  The path $P$ now implies that $d(X\sm Z,Z\sm X)>0$, a contradiction. 

Hence we have $X\cap Z=\{r\}$ and, by symmetry, $X\cap Y=\{t\}$. 
This gives
\begin{eqnarray*}
i_{E\cup L}(X\cup Y\cup Z\cup\{v\})&\geq& i_{E\cup L}(X)+i_E(Y)+i_E(Z)+3\\
&=&2|X|+(2|Y|-3)+(2|Z|-3)+3\\
&=&2|X\cup Y\cup Z\cup \{v\}|+1.
\end{eqnarray*}
Since $G$ is a rigid $\M_{lc}$-circuit, we must have $X\cup Y\cup Z\cup\{v\}=V$ and
$i_{E\cup L}(X)+i_E(Y)+i_E(Z)+3=|E|+|L|$. This implies that all loops in $G$ are contained in $G[X]$ and that the component of $G-\{r,t\}$ which contains $\{v,s\}$ has no loops.
Hence $G$ is not balanced.
\end{proof}

We call a triple $(X,Y,Z)$ of three sets satisfying the hypotheses of Lemma \ref{lem:flower2} for $G$  an {\em unbalanced flower on $v$}. Note that $(X,Y,Z)$ cannot be both an unbalanced flower and a strong or weak flower since, in the the former, $X\cap Y\cap Z=\emptyset$ and $G[Y],G[Z]$ are connected, while, for every strong or weak flower with $X\cap Y\cap Z=\emptyset$, each of $G[X],G[Y],G[Z]$ are disconnected. 

\begin{lem}\label{lem:flower3}
Let $G$ be a 
%bridgeless 
rigid $\M_{lc}$-circuit and $v$ be a non-admissible node of $G$ with three distinct neighbours. Then at least one of the following holds:\\
%(a) $v$ is admissible;\\
(a) there exists a strong flower on $v$ in $G$;\\
(b) there exsits a weak flower on $v$ in $G$; \\
(c) there exists an unbalanced flower on $v$ in $G$.
\end{lem}

\begin{proof}
This follows immediately from Lemmas \ref{lem:admcases}, \ref{lem:flower1} and \ref{lem:flower2}.
\end{proof}

Our final result of this section is a decomposition lemma for {\em unbalanced} rigid $\M_{lc}$-circuits i.e.\ rigid $\M_{lc}$-circuits which are not balanced. It uses the following graph operation. 

Given three looped simple graphs $G=(V,E,L)$, $G_1=(V_1,E_1,L_1)$ and $G_2=(V_2,E_2,L_2)$, we say that {\em $G$ is the $2$-sum of $G_1$ and $G_2$ along an edge $uv$} if $V_1\cup V_2=V$, $V_1\cap V_2=\{u,v\}$, $E=(E_1\cup E_2)-uv$, $E_1\cap E_2=\{uv\}$,  
$L=L_1\cup L_2$ and $L_1\cap L_2=\emptyset$.
Figure \ref{fig:unbalanced_circuit} gives an example of an unbalanced rigid $\M_{lc}$-circuit
which is the 2-sum of a rigid and a flexible $\M_{lc}$-circuit.
\begin{figure}[h]
\begin{center}
\begin{tikzpicture}[font=\small]

\node[roundnode] at (0,0) (v1) [label=below:$v$]{}
	edge[in=205,out=245,loop]();
\node[roundnode] at (1,0) (v2) []{}
	edge[] (v1);
\node[roundnode] at (1,1) (v3) []{}
	edge[] (v1)
	edge[] (v2);
\node[roundnode] at (0,1) (v4) [label=above:$u$]{}
	edge[] (v2)
	edge[] (v3)
	edge[in=125,out=165,loop]();
\node[roundnode] at (-1,0) (x1) []{}
	edge[] (v1)
	edge[in=210,out=250,loop]();
\node[roundnode] at (-1,1) (x2) []{}
	edge[] (v1)
	edge[] (v4)
	edge[in=110,out=150,loop]();
\node[roundnode] at (-2,0.5) (x3) []{}
	edge[] (x1)
	edge[] (x2)
	edge[in=160,out=200,loop]();
\begin{scope}[xshift=-5cm]
\node[roundnode] at (0,0) (v1) [label=below:$v$]{};
\node[roundnode] at (1,0) (v2) []{}
	edge[] (v1);
\node[roundnode] at (1,1) (v3) []{}
	edge[] (v1)
	edge[] (v2);
\node[roundnode] at (0,1) (v4) [label=above:$u$]{}
	edge[] (v1)
	edge[] (v2)
	edge[] (v3);
\end{scope}
\begin{scope}[xshift=-6cm]
\node[roundnode] at (0,0) (v1) [label=below:$v$]{}
	edge[in=205,out=245,loop]();
\node[roundnode] at (0,1) (v4) [label=above:$u$]{}
	edge[] (v1)
	edge[in=125,out=165,loop]();
\node[roundnode] at (-1,0) (x1) []{}
	edge[] (v1)
	edge[in=210,out=250,loop]();
\node[roundnode] at (-1,1) (x2) []{}
	edge[] (v1)
	edge[] (v4)
	edge[in=110,out=150,loop]();
\node[roundnode] at (-2,0.5) (x3) []{}
	edge[] (x1)
	edge[] (x2)
	edge[in=160,out=200,loop]();
\end{scope}
\end{tikzpicture}
\end{center}
\caption{An unbalanced rigid $\M_{lc}$-circuit on the right (removing $u$ and $v$ results in a loopless component) obtained from a rigid
$\M_{lc}$-circuit and a flexible $\M_{lc}$-circuit by a 2-sum operation along the edge $uv$ on the left.}
\label{fig:unbalanced_circuit}
\end{figure}
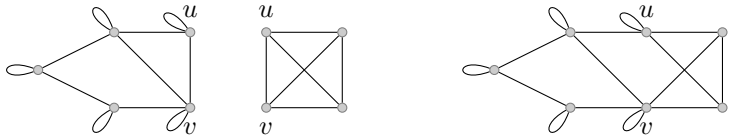
Lemma \ref{lem:unbalanced} below shows that every unbalanced rigid $\M_{lc}$-circuit can be obtained in this way.
\begin{lem}\label{lem:unbalanced}
Let $G=(V,E,L)$ be a looped simple graph. Then $G$ is an unbalanced rigid $\M_{lc}$-circuit if and only if $G$ is a $2$-sum of a rigid $\M_{lc}$-circuit and a flexible $\M_{lc}$-circuit. 
\end{lem}

\begin{proof}
First suppose that $G$ is a 2-sum of a rigid $\M_{lc}$-circuit $G_1=(V_1,E_1,L_1)$ and a  flexible $\M_{lc}$-circuit $G_2=(V_2,E_2)$ along an edge
$e=uv\in E_1\cap E_2$. It is straightforward to check that $G$ satisfies the conditions given in Lemma \ref{lem:circuit_char}(a). Hence $G$ is  a rigid $\M_{lc}$-circuit. It is unbalanced since $G_2-\{u,v\}$ is a  connected component of $G-\{u,v\}$ which contains no loops.

Now suppose $G$ is an unbalanced, rigid $\M_{lc}$-circuit. Then there is a set $\{u,v\}\subset V$ such that $G-\{u,v\}$  has a component $H$ with
no loops. Let $H_1=(V_1,E_1,L_1)$  be the subgraph of $G$ induced by   $V\setminus V(H)$, and $H_2=(V_2,E_2)$ be the simple subgraph of $G$ obtained from $G-(V_1\sm \{u,v\})$ by deleting any loops at $u$ or $v$. 
We have
$$2|V|+1=|E|+|L|\leq i_{E_1\cup L_1}(V_1)+i_{E_2}(V_2)\leq 2|V_1|+2|V_2|-3=2|V|+1$$
Thus equality must occur throughout. This implies that $uv\not\in E$ (by equality in the first inequality),  $|E_1|+|L_1|=2|V_1|$ and $|E_2|=2|V_2|-3$ (by equality in the second inequality). Let $G_1$ and $G_2$ be the graphs obtained from $H_1$ and $H_2$, respectively, by adding the edge $uv$. It is straightforward to check that $G_1$ and $G_2$ satisify the conditions of Lemma \ref{lem:circuit_char} (a) and (b), respectively. Hence $G_1$ is a rigid $\M_{lc}$-circuit, $G_2$ is a flexible $\M_{lc}$-circuit, and $G$ is the 2-sum of $G_1$ and $G_2$. 
\end{proof}

\section{$\mathcal{M}_{lc}$-connected graphs}
\label{sec:admissible}

Our long term aim is to obtain a recursive construction for balanced redundantly rigid graphs. To accomplish this we first consider the closely related family of `{$\mathcal{M}_{lc}$-connected graphs}'.
It is easy to see that a looped simple graph  $G=(V,E,L)$ is redundantly rigid if and only if it is rigid and every element of $E\cup L$ belongs to an  $\mathcal{M}_{lc}$-circuit. 
%We define 
The  graph  $G$ is {\em $\mathcal{M}_{lc}$-connected} if every pair of elements of $E\cup L$ belong to a common $\mathcal{M}_{lc}$-circuit in $G$.

%We will obtain a recursive construction for $\mathcal{M}_{lc}$-connected looped simple graphs. More precisely we show that every $\mathcal{M}_{lc}$-connected looped simple graph can be obtained from $K_1^{[3]}$ by edge additions, 1-extensions and a new operation `$K_4$-extension' which will be defined later. 

We will show that any balanced $\mathcal{M}_{lc}$-connected graph other than $K_1^{[3]}$ can be reduced to a smaller $\mathcal{M}_{lc}$-connected graph using the operation of edge/loop deletion or 1-reduction.  Our proof uses the concept of an `ear decomposition' of a matroid and follows a similar strategy to that used in \cite{Carxiv,J&J}. 

%Let $\M=(D,r)$ denote a matroid with ground set $D$ and rank function $r$. 
Recall that a matroid $M=(E,r)$ is \emph{connected} if every pair of elements of $M$ is contained in a common circuit. 
Given a non-empty sequence of circuits $C_1,C_2,\dots,C_m$ in $M$, let  $D_i=C_1\cup C_2\cup \dots C_i$ for all $1\leq i \leq m$, and put $\tilde C_i=C_i-D_{i-1}$. The sequence $C_1,C_2,\dots,C_m$ is a \emph{partial ear decomposition} of $M$ if, for all $2\leq i \leq m$,
\begin{enumerate}
\item[(E1)] $C_i\cap D_{i-1}\neq \emptyset$,
\item[(E2)] $C_i-D_{i-1}\neq \emptyset$, and
\item[(E3)] no circuit $C_i'$ satisfying (E1) and (E2) has $C_i'-D_{i-1}\subset C_i-D_{i-1}$.
\end{enumerate}
A partial ear decomposition $C_1,C_2,\dots,C_m$ is an \emph{ear decomposition} of $M$ if $D_m=E$.

\begin{lem}[\cite{C&H}]\label{lem:ear}
Let $M=(E,r)$ be a matroid with $|E|\geq 2$. Then:
\begin{enumerate}
\item[(i)] $M$ is connected if and only if $M$ has an ear decompostion.
\item[(ii)] If $M$ is connected then every partial ear decomposition is extendable to an ear decomposition of $M$.
\item[(iii)] If $C_1,C_2,\dots,C_m$ is an ear decomposition of $M$ then $r(D_i)-r(D_{i-1})=|\tilde C_i|-1$ for all $2\leq i \leq m$.
\end{enumerate}
\end{lem}

%We will say that a looped simple graph $G=(V,E,L)$ is \emph{simple} if $L=\emptyset$ and \emph{looped} if $L\neq \emptyset$. 
Given a looped simple graph $G$, it will be convenient to refer to an ear decomposition $C_1,C_2,\dots,C_m$ of $\M_{lc}(G)$
as an ear decomposition $H_1,H_2,\dots,H_m$ of $G$ where $H_i$ is the $\M_{lc}$-circuit of $G$ induced by $C_i$ for $1\leq i \leq m$.
See Figure \ref{fig:ear_decomp} for an example giving two distinct ear decompositions of the graph $G$ drawn on the far left. The ear decomposition drawn in the middle has a flexible
$\M_{lc}$-circuit $K_4$ whereas the circuits of the ear decomposition drawn on the right are all rigid $\M_{lc}$-circuits. The following lemma
tells that a rigid $\M_{lc}$-connected graph always has an ear decompositon into rigid $\M_{lc}$-circuits.
%Our next 4 results are proved using ideas from REF - Katie's thesis.
\begin{figure}[h]
\begin{center}
\begin{tikzpicture}[font=\small]
\node[roundnode] at (0,0) (v1) []{}
	edge[in=180,out=220,loop]()
	edge[in=230, out=270,loop]();
\node[roundnode] at (1,0) (v2) []{}
	edge[] (v1)
	edge[in=270, out=310,loop]()
	edge[in=320, out=360,loop]();
\node[roundnode] at (1,1) (v3) []{}
	edge[] (v1)
	edge[] (v2)
	edge[in=0, out=40,loop]()
	edge[in=50, out=90,loop]();
\node[roundnode] at (0,1) (v4) []{}
	edge[] (v1)
	edge[] (v2)
	edge[] (v3)
	edge[in=90, out=130,loop]()
	edge[in=140, out=180,loop]();
	
\begin{scope}[xshift=4cm]
\node[roundnode] at (0,0) (v1) []{}
	edge[in=180,out=220,loop]()
	edge[in=230, out=270,loop]();
\node[roundnode] at (0,1) (v4) []{}
	edge[] (v1)
	edge[in=90, out=130,loop]()
	edge[in=140, out=180,loop]();
	
	\begin{scope}[xshift=0.5cm]
	\node[roundnode] at (0,0) (v1) []{};
	\node[roundnode] at (1,0) (v2) []{}
		edge[] (v1);
	\node[roundnode] at (1,1) (v3) []{}
		edge[] (v1)
		edge[] (v2);
	\node[roundnode] at (0,1) (v4) []{}
		edge[] (v1)
		edge[] (v2)
		edge[] (v3);
	\end{scope}
	\begin{scope}[xshift=1cm]
	\node[roundnode] at (1,0) (v2) []{}
		edge[in=270, out=310,loop]()
		edge[in=320, out=360,loop]();
	\node[roundnode] at (1,1) (v3) []{}
		edge[] (v2)
		edge[in=0, out=40,loop]()
		edge[in=50, out=90,loop]();
	\end{scope}
\end{scope}
\begin{scope}[xshift=8cm]
\node[roundnode] at (0,0) (v1) []{}
	edge[in=180,out=220,loop]()
	edge[in=230, out=270,loop]();
\node[roundnode] at (1,0) (v2) []{}
	edge[] (v1);
\node[roundnode] at (1,1) (v3) []{}
	edge[] (v1)
	edge[] (v2);
\node[roundnode] at (0,1) (v4) []{}
	edge[] (v2)
	edge[] (v3)
	edge[in=90, out=130,loop]()
	edge[in=140, out=180,loop]();
	\begin{scope}[xshift=1.5cm]
	\node[roundnode] at (0,0) (v1) []{};
	\node[roundnode] at (1,0) (v2) []{}
		edge[] (v1)
		edge[in=270, out=310,loop]()
		edge[in=320, out=360,loop]();
	\node[roundnode] at (1,1) (v3) []{}
		edge[] (v1)
		edge[in=0, out=40,loop]()
		edge[in=50, out=90,loop]();
	\node[roundnode] at (0,1) (v4) []{}
		edge[] (v1)
		edge[] (v2)
		edge[] (v3);
	\end{scope}
\end{scope}
\end{tikzpicture}
\end{center}
\caption{A graph $G$ on the left, and two ear decompositions of $G$ in the middle and on the right.}
\label{fig:ear_decomp} 
\end{figure}
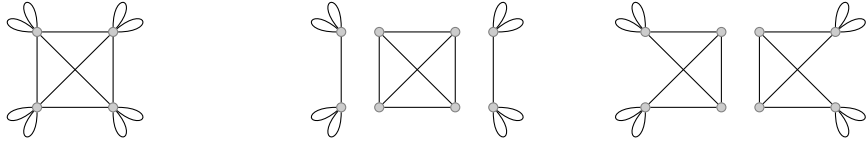
\begin{lem}\label{lem:rigideardecomp}
Let $G$ be an $\M_{lc}$-connected looped simple graph with at least one loop. Then $G$ has an ear decomposition into rigid $\M_{lc}$-circuits.
\end{lem}

\begin{proof}
Let $\ell$ be a loop of $G$. Since $G$ is $\M_{lc}$-connected there exists an $\M_{lc}$-circuit $H_1$ containing $\ell$. Then $H_1$ is a rigid $\M_{lc}$-circuit and, if $G=H_1$, then we are done. So suppose $G\neq H_1$. Extend $H_1$ to an ear decomposition $H_1,H_2,\dots,H_k$ of $G$ such that each $H_i$ is a rigid $\M_{lc}$-circuit for $1\leq i\leq k$ and $k$ is as large as possible.

%Suppose $C_1,C_2,\dots,C_m$ is not an ear decomposition of $\M_{lc}$.
Suppose $k<m$. Then we may 
choose an edge or loop $f$ in $H_{k+1}$ which does not belong to $\bigcup_{i=1}^k H_i$. 
%Take $e_1\in \{f_1,\ell_1\}$ of opposite type to $e_k$. 
Since $\bigcup_{i=1}^{k+1} H_i$ is $\M_{lc}$-connected, there exists an $\M_{lc}$-circuit $H_{k+1}'\subseteq \bigcup_{i=1}^{k+1} H_i$ such that $\ell,f$ are in $H_{k+1}'$. Then $H_{k+1}'$ is a rigid $\M_{lc}$-circuit and $H_1,H_2,\dots,H_{k},H_{k+1}'$ is a partial ear decomposition of $G$. Since every partial ear decomposition can be extended to a `full' ear decomposition, this contradicts the maximality of $k$.
\end{proof}

Lemma \ref{lem:rigideardecomp} and the fact that the union of two redundantly rigid graphs is redundantly rigid immediately give

\begin{cor}\label{cor:mlcredundant}
Let $G$ be an $\M_{lc}$-connected looped simple graph. Then $G$ is redundantly rigid if and only if $G$ has a loop.
\end{cor}

\begin{lem}\label{lem:lastear}
Let $G$ be an $\M_{lc}$-connected  looped simple graph that  contains a loop. Let $C_1,C_2,\dots,C_m$ be an ear decomposition of $\M_{lc}(G)$ where $m\geq 2$ and $H_i=G[C_i]$ is a rigid $\M_{lc}$-circuit for $1\leq i \leq m$. Let $Y=V(H_m) \sm \bigcup_{i=1}^{m-1}V(H_i)$ and $X=V(H_m) \sm Y$. Then
\begin{enumerate}
\item[(i)] $|\tilde C_m|=2|Y|+1$;
\item[(ii)] if $Y\neq \emptyset$ then every edge/loop $e\in \tilde C_m$ is incident to $Y$, $X$ is mixed critical in $H_m$ and $G[Y]$ is connected;
\item[(iii)] if $G$ is balanced, $Y\neq \emptyset$ and $\tilde C_m$ contains no loops then $Y$ has at least three neighbours in $X$.
\end{enumerate}
\end{lem}

\begin{proof}
(i) Let $G_j=\bigcup_{i=1}^j H_i$ and $D_j=\bigcup_{i=1}^j C_j$. Hence $E(G_j)\cup L(G_j)=D_j$. Lemma \ref{lem:ear}(i) implies that $G_{m-1}$ is $\M_{lc}$-connected. Corollary \ref{cor:mlcredundant} implies that $G_{m-1}$ and $G$ are rigid so $r(D_{m-1})=2|V\sm Y|$ and $r(E\cup L)=2|V|$. Hence by Lemma \ref{lem:ear}(iii) we have $|\tilde C_m|=r(E\cup L)-r(D_{m-1})+1=2|V|-2|V-Y|+1=2|Y|+1$.

(ii) Suppose $Y\neq \emptyset$. Let $k$ be the number of edges/loops in $E\cup L - D_{m-1}$ which have all endvertices in $V(G_{m-1})$. Since $H_m$ is a rigid $\M_{lc}$-circuit, part (i) implies that 
$$i_{H_m}(X)=|C_m|-|\tilde C_m|+k=2|X\cup Y|+1 - (2|Y|+1)+k=2|X|+k.$$
Since $H_m[X]$ is a proper subgraph of $H_m$ we must have $k=0$ and $X$ is mixed critical in $H_m$. 

Assume $G[Y]$ is disconnected. 
Let $Y_1,Y_2,\ldots,Y_k$ be the vertex sets of the connected components of $G[Y]$. Since $H_m$ is an $\M_{lc}$-circuit, $k\geq 2$ and $X$ is mixed critical, we have   $i_{H_m}(X\cup Y_i)-i_{H_m}(X)\leq 2|X\cup Y_i|-2|X|=2|Y_i|$. This
%Then $G[Y]$ consists of connected components $G[Y_2],G[Y_2],\dots,G[Y_k]$ for some $k\geq 2$ where $Y_1,Y_2,\dots,Y_k$ partitions $Y$. Since $H_m$ is an  $\M_{lc}$-circuit, $H_m[Y_i]$ is independent for all $1\leq i \leq k$. Each component of $Y$ satisfies $i_{H_m}(X\cup Y_i)-i_{H_m}(X)\leq 2|X\cup Y_i|-2|X|=2|Y_i|$, which 
implies that 
$$|\tilde C_m|=\sum_{i=1}^k(i_{H_m}(X\cup Y_i)-i_{H_m}(X))\leq \sum_{i=1}^k2|Y_i|=2|Y|,$$ 
contradicting part (i).

(iii) Let $X'$ be the set of vertices in $X$ which are adjacent to $Y$. Then $G- X'$ has a component with no loops. The fact that $G$ is balanced and $\M_{lc}$-connected now 
%(CHECK CASE WHEN $|X|\leq 1$ - would imply a circuit with exactly 2 loops which isn't possible.) 
implies that $|X'|\geq 3$.
\end{proof}

\begin{lem}\label{lem:feasible2}
Let $G=(V,E,L)$ be an $\mathcal{M}_{lc}$-connected looped simple graph which contains a loop.
%with $L\neq \emptyset$ and let 
Suppose that $G'$ is obtained from $G$ by
an edge or loop addition, or a $1$-extension. Then $G'$ is $\mathcal{M}_{lc}$-connected.
\end{lem}

\begin{proof}
Note that $G$ is redundantly rigid by Corollary \ref{cor:mlcredundant}.

First suppose that $G'$ is obtained from $G$ by adding a new edge or loop $f$. Since $G$ is rigid,
%and $L\neq \emptyset$, we see that $G$ is redundantly rigid by Corollary \ref{cor:mlcredundant}. This implies that 
there exists an $\M_{lc}$-circuit $C$ in $G'$ with $f\in C$. The $\M_{lc}$-connectivity of $G'$ now follows from the transitivity of the relation that defines $\M_{lc}$-connectivity.

Next suppose $G'$ is obtained from $G$ by a 1-extension operation that deletes an edge or loop $f$ and adds
a new vertex $v$ with three incident edges or loops  $f_1,f_2,f_3$. By 
%the 
transitivity, 
%of the relation that defines  $\M_{lc}$-connected components 
it will suffice to show that every $e\in (E\cup L)-f$ belongs to an  $\M_{lc}$-circuit of $G'$ containing $f_1,f_2$ and $f_3$.  Since $G$ is $\M_{lc}$-connected, there exists an  $\M_{lc}$-circuit $C$ in $G$ containing $e$ and $f$. Choose a base $B$ of  $\M_{lc}(G-f)$ with $C-f$ contained in $B$. Then $G[B]$ is rigid simce $G$ is rigid and $e\in C\subseteq B+f$. Hence $B'=B-e+f$ is another base of  $\M_{lc}(G)$ and $G[B']$ is rigid. Then $B''=B'-f+f_1+f_2+f_3$ is a base of  $\M_{lc}(G')$ and $G'[B'']$ is rigid, since $G'[B'']$ is obtained by performing a 1-extension on $G[B']$. Hence $B''+e$ contains a unique  $\M_{lc}$-circuit $C'$. Since $B''$ is independent we have $e\in C'$. If $f_1\notin C'$ then, since $v$ is incident with exactly three edges, we would have $C\subseteq B$, contradicting the fact that $B$ is independent. Hence $f_1$ is in $C'$ and, since $v$ is incident to exactly three edges, we also have $f_2,f_3\in C'$. 
%$z\notin \{u,w\}$. Let $f\in L$ be a loop if $e$ is a loop $f\in E$ be an edge incident with $z$ if $e=vz$.
%Then we also have $f\in E(G')\cup L(G')$. We will show that
%for all edges $g\in (E(G')\cup L(G'))-f$ there exists an $\M_{lc}$-circuit $C$ in $G'$ with $f,g\in C$.
%Then this will tell us $G'$ is $\M_{lc}$-connected by transitivity.
%
%If $g\in E\cup L$, then there exists an $\M_{lc}$-circuit $C'$ in $G$ with $f,g\in C'$ since $G$ is
%$\M_{lc}$-connected. If $uw\notin C'$, then we can set $C=C'$. Otherwise, Let $C$ be the 1-extension
%of $C'$ (the same 1-extension operation that produces $G'$ from $G$). Then $C$ is a subgraph of $G'$.
%TO FINISH...
\end{proof}

We next consider the inverse operations to edge/loop addition and 1-extension. To do this 
we extend our definition of admissibility from an $\M_{lc}$-circuit to an $\mathcal{M}_{lc}$-connected graph $G=(V,E,L)$.   We say that an edge or loop  $f\in E\cup L$ is \emph{admissible} if $G-f$ is 
$\mathcal{M}_{lc}$-connected, and a node $v\in V$ is \emph{admissible} if there exists a 1-reduction at $v$ which results in an $\mathcal{M}_{lc}$-connected graph. 

We will show that every balanced $\mathcal{M}_{lc}$-connected graph $G$ other than $K_1^{[3]}$  has an admissible edge or node.
The main idea is to find an admissible edge/loop or node in the last  $\M_{lc}$-circuit of an ear decomposition of $G$. We will need some additional terminology for nodes in rigid $\M_{lc}$-circuits to do this.

 Suppose $v$ is a node in a rigid $\M_{lc}$-circuit $G=(V,E,L)$ and $X$ is a critical set in $G$. We say that $X$ is \emph{node critical for $v$} 
if $N(v)-z\subseteq X\subseteq V\sm \{v,z\}$ for some $z\in N(v)$ with $d^\dag(z)\geq 4$. We will also refer to a critical set $X$ as being
\emph{node critical} when it is \emph{node critical} for some (unspecified) node. 
We say that $v$ 
%is a 
%{\em node} if $v\in V_3$, that 
%a node $v$
%, with $d(v)=3$,
 is a {\em leaf node} if $d_{G^\dag[V_3]}(v)\in \{0,1\}$, a {\em series node} if $d_{G^\dag[V_3]}(v)=2$ and a {\em branching node} if $d_{G^\dag[V_3]}(v)=3$.
Note that a 1-reduction
at a branching node cannot be admissible since the resulting graph will have  a vertex $v$ with $d^\dag(v)=2$.

\begin{thm}\label{thm:balancedrradmissiblenode}
Let $G=(V,E,L)$  be a balanced, $\mathcal{M}_{lc}$-connected, looped simple graph 
%with a loop, and 
distinct from $K_1^{[3]}$. Then $G$ has an admissible edge or loop, or an admissible node.
\end{thm}

\begin{proof}
We proceed by contradiction. Suppose that $G$ has no admissible edges or nodes. 
Since $G$ is balanced, $G$ has at least one loop.
Lemma \ref{lem:rigideardecomp} now implies that $G$ has an ear decomposition $C_1,C_2,\dots,C_m$ such that $H_i=G[C_i]$ is a rigid $\M_{lc}$-circuit for all $1\leq i \leq m$. Let $H_i=(V_i,E_i,L_i)$ for all $1\leq i \leq m$ and $H=\bigcup_{i=1}^{m-1}H_i$. Then $H$ is $\M_{lc}$-connected by Lemma \ref{lem:ear}(i). 
%and hence is connected and redundantly rigid by Corollary \ref{cor:mlcredundant}. 
Put $Y=V_m\sm \cup_{i=1}^{m-1}V_i$ and $X=V_m \sm Y$. Note that if $m=1$ then $Y=V$ and $X=\emptyset$. 
If $Y=\emptyset$ then $H=G-e$ for some $e\in E\cup L$ and $e$ is admissible in $G$ since $H$ is $\mathcal{M}_{lc}$-connected. Hence we may assume that $Y\neq \emptyset$. 
Lemma \ref {lem:lastear}(ii) now implies that $X$ is mixed critical in $H_m$.
Hence $Y$ contains at least one node of $H_m$ by Lemma \ref{lem:minus}. 

\begin{claim}\label{clm:node}
No node of $H_m$ in $Y$  is admissible in $H_m$, and hence every node of $H_m$ in $Y$ has at least two distinct neighbours in $G$. 
\end{claim} 

\begin{proof}
Suppose some node $v\in Y$  is admissible in $H_m$.
Let $C_m'$ be the edge set of the rigid $\M_{lc}$-circuit obtained from $H_m$ by performing an admissible 1-reduction at $v$. 
Then $v$ is a node in $G$ and the graph obtained from $G$ by performing the same 1-reduction at $v$ will be $\mathcal{M}_{lc}$-connected since it will have an ear decomposition $C_1,C_2,\ldots,C_{m-1},C_m'$. 
Hence $v$ will be admissible in $G$, contradicting the fact that $G$ has no admissible nodes. We can now use lemma \ref{lem:admcases}(c) to deduce that $v$ has at least two distinct neighbours.
\end{proof}

\begin{claim}\label{clm:nbs}
Suppose $v$ is a node of $H_m$ in $Y$. Then $v$  has at least two neighbours in $Y$.
\end{claim}

\begin{proof}
We split the proof into three cases.
\\\textbf{Case 1:} $N(v)\cap Y=\emptyset$.
%We first consider the case when no neighbour of $v$ belongs to $Y$.

Since $H_m[Y]$ is connected by Lemma \ref{lem:lastear}(ii), we have $Y=\{v\}$.  If $G[N(v)]$ is not complete then we may choose $p,q\in N(v)$ with $pq\not\in E$. Then the 1-reduction at $v$ which adds the edge $pq$ creates the graph $H+pq$. Lemma \ref{lem:feasible2} and the fact that $H$ is $\mathcal{M}_{lc}$-connected now imply that this 1-reduction is admissible in $G$. This contradicts the assumption that $G$ has no admissible nodes.
Hence $G[N(v)]$ is complete. Then $G-v=H$ by Lemma \ref{lem:lastear}(ii), so $G-v$ is $\mathcal{M}_{lc}$-connected. Since $G-e$ is a 1-extension of $H$ for any edge $e$ in $G[N(v)]$, it is also $\mathcal{M}_{lc}$-connected by Lemma \ref{lem:feasible2}. Hence $e$ is admissible in $G$. This contradicts the 
assumption  that $G$ has no admissible edges.
%Thus we may suppose $v$ has a neighbour in $Y$. 
\\\textbf{Case 2:} $N(v)=\{r,s,t\}$ with $\{r,t\}\subseteq X$, $s\in Y$.
%We next consider the case when  $N(v)=\{r,s,t\}$ with $\{r,t\}\subseteq X$.

Since $v$ is not admissible in $H_m$, there exist critical sets $R,T$ in $H_m$ such that $\{s,r\}\subseteq T\subseteq V_m\sm \{v,t\}$ and 
$\{s,t\}\subseteq R\subseteq V_m\sm \{v,r\}$ by Lemma \ref{lem:admcases}. We may suppose that $R,T$ have been chosen to be the minimal critical sets with these properties. 
If $R,T$ are both pure critical, then $(R,S,X)$ would be an unbalanced flower on $v$ in $H_m$, and 
we could use Lemma \ref{lem:flower2} to contradict the hypothesis that $G$ is balanced. Relabelling $R,T$ if necessary 
we may assume that $R$ is mixed critical, but not pure critical, in $H_m$. 
Lemma \ref{lem:flower3} now implies that $(X,R,T)$ is
a strong or weak flower on $v$ in $H_m$.   
Let $H_m'$ denote the graph obtained from $H_m$ by applying a 1-reduction at $v$ adding the edge $st$.
Then $H_m'$ contains a unique $\M_{lc}$-circuit $D$. The minimality of $R$ and the fact that $R$ is not pure critical now imply that $D=H_m[R]+st$ and $D$ is a rigid $\M_{lc}$-circuit. Let $C_m'=E(D)\cup L(D)$.

We will show that $v$ is admissible in $G$ by verifying that $G'=G-v+st$ is $\mathcal{M}_{lc}$-connected. Since $H$ is $\mathcal{M}_{lc}$-connected, it will suffice to show that $G'=H\cup D$ and that $D$ contains an edge or loop of $H$.
We first verifiy that $D$ contains an edge of $H$. Since each edge of $\tilde C_m$ is incident with $Y$ by Lemma \ref {lem:lastear}(ii), this is equivalent to showing that $D$ contains an edge of $H_m[X]$. This follows because $R$ and $X$ are mixed critical in $H_m$, so $R\cap X$ is mixed critical and nonempty, and all edges of $H_m[R\cap X]$ belong to $D$.
It remains to show that $G'=D\cup H$ i.e.\ $(\tilde C_m\sm \{vr,vs,vt\})+st\subseteq C_m'$. Lemma \ref{lem:flower1} implies that all edges of $H_m$ which are not incident with $v$ are induced by either $X$ or $R$, and
no edge of $\tilde C_m$ is induced by $X$ by Lemma \ref{lem:lastear}(ii). Since $C_m'=E(D)\cup L(D)$ and $D=H_m[R]+st$ we have $(\tilde C_m\sm \{vr,vs,vt\})+st\subseteq C_m'$.
Thus $G'$ is $\mathcal{M}_{lc}$-connected and $v$ is admissible in $G$.
\\\textbf{Case 3:} $N(v)=\{r,s\}$ with  $r\in X$ and $s\in Y$.
%Finally we consider the case when $N(v)=\{r,s\}$ with  $r\in X$ and $s\in Y$.

Since $v$ is not admissible in $H_m$, Lemma \ref{lem:admcases} implies that there exists a pure critical set $R$ in $H_m$ with $\{r,s\}\subseteq R\subseteq V(H_m)-v$ and a mixed critical set $S$ with $s\in S\subseteq V(H_m)\sm \{r,v\}$. We may assume that $R$ and $S$ have been chosen to be the minimal such sets.  

Suppose that $|R\cap X|\geq 2$. Then $X\cup R$ is mixed critical, $X\cap R$ is pure critical, $H_m[R]$ is simple and $H_m-v=H_m[X]\cup H_m[R]$ by Lemma \ref{lem:union}(c). Note that Lemma \ref{lem:union}(b) applied to $X$ and $S$ tells us that $rs\notin E$.
Let $H_m'=H_m-v+rs$. Then $H_m'$ contains a unique $\M_{lc}$-circuit $D$. The minimality of $R$  now implies that $D=H_m[R]+rs$ and $D$ is simple. Let 
%$C_m'=E(D)$ and 
$G'=G-v+rs$. We will show that $G'$ is $\M_{lc}$-connected.
Since $H_m-v=H_m[X]\cup H_m[R]$ and $X\subseteq V(H)$, all edges of $G'$ belong to  $H$ or $D$. Since $H$ and $D$ are both $\M_{lc}$-connected and have at least one edge in common (as $X\cap R$ is a critical set with at least two vertices in $H_m$), $G'$ is $\M_{lc}$-connected
%Corollary \ref{cor:mlcredundant} now implies that $G'$ is redundantly rigid 
and hence $v$ is admissible in $G$.
This contradicts the assumption that $G$ has no admissible nodes so we must have $R\cap X=\{r\}$.

Since $R$ is pure-critical, $R\cap X=\{r\}$ and $s\in R$, we can find a path $P$ in $H_m[R]$ from $r$ to $s$ which avoids $X-r$.
Since $X$ and $S$ are mixed critical in $H_m$ and $v\not\in X\cup S$, Lemma \ref{lem:union}(b) implies there are no edges in $H_m$ from $X\sm S$ to $S\sm X$. 
Since $r\in X\sm S$ and $s\in S\sm X$, this contradicts the existence of the path $P$.
\end{proof}

\begin{claim}\label{clm:notbal}
Let $v$ be a node of $H_m$ in $Y$ with three distinct neighbours. Then there is no unbalanced flower on $v$ in $H_m$.
\end{claim}

\begin{proof}
Let $N(v)=\{w,u,z\}$ and suppose there exists an unbalanced flower $(W,U,Z)$ on $v$ in $H_m$ with
$u,z\in W\subseteq V_m\sm \{v,w\}$, $w,z\in U\subseteq V_m\sm \{v,u\}$, and $w,u\in Z\subseteq V_m\sm \{v,z\}$.

First suppose that $N(v)\subseteq Y$. We may assume that $U,Z$ are pure critical and that $W$ is mixed critical.
Replacing $W$ by $W\cup X$ if necessary, we may assume that $X\subseteq W$. The properties of an unbalanced flower (Lemma \ref{lem:flower2}) now tell us that the component of $G-\{u,z\}$ which contains $\{v,w\}$ is loopless. This contradicts the hypothesis that $G$ is balanced.

Now suppose $N(v)\not\subseteq Y$.  Claim \ref{clm:nbs} implies that $v$ has at least two neighbours in $Y$. Hence we may assume that $u,z\in Y$ and $w\in X$. Since $X$ is mixed critical, each connected component of $H_m[X]$ contains a loop. Lemma \ref{lem:flower2} and the fact that $X\cap Z=\{w\}\subseteq  X\cap U$, now implies that $U$ is mixed critical, $W$ is pure critical, and no vertex of $(W\cup Z)\sm \{w,z\}$ is incident with a loop in $H_m$. The facts that $H_m[X]$ is rigid and no vertex of $X\cap (W-z)$ is incident with a loop now imply that $X\cap (W-z)=\emptyset$, and hence $X\subseteq U$. This contradicts the hypothesis that $G$ is balanced since the component of $G-\{w,z\}$ which contains $\{v,u\}$ will be loopless.  
\end{proof}

\begin{claim}\label{clm:nodecritical}
We can choose a node $v$ of $H_m$ in $Y$ and a critical set $X_v$ in $H_m$ such that $X_v$ is mixed node critical for $v$ and $X\subseteq X_v$.
\end{claim}

\begin{proof}
Let $F=H_m^\dag[V_3\cap Y]$.  Then $F$ is a forest by Lemma \ref{lem:forest} and we may choose a vertex $v$ of $F$ such that $d_F(v)\leq 1$. Since $v$ has at least two neighbours in $Y$ by Claim \ref{clm:nbs}, $v$ has a neighbour $z$ in $Y$ with $d^\dag(z)\geq 4$. 

We first consider the case when $N(v)=\{y,z\}$. Then $y\in Y$ by Claim \ref{clm:nbs}. Since $v$ is not admissible in $H_m$, Lemma \ref{lem:admcases} implies that there exists a mixed critical set $S$ in $H_m$ with
$y\in S\subseteq V_m\sm \{v,z\}$. Then $S\cup X$ is a mixed node critical set for $v$.

We next consider the case when $N(v)=\{w,u,z\}\subseteq Y$.  Since $d_F(v)\leq 1$ we may assume that $d^\dag(u)\geq 4$. Since $v$ is not admissible in $H_m$, Lemma \ref{lem:flower3} and Claim \ref{clm:notbal} imply that there exists a strong or weak flower $(W,U,Z)$ on $v$ in $H_m$ with
$u,z\in W\subseteq V_m\sm \{v,w\}$, $w,z\in U\subseteq V_m\sm \{v,u\}$, and $w,u\in Z\subseteq V_m\sm \{v,z\}$. Then either $Z$, or $U$, is mixed critical and hence either $Z\cup X$, or $U\cup X$, is a mixed node critical set for $v$. 
%Hence neither $U$ nor $Z$ is mixed critical, a contradiction. 
%Then $(W,U,Z)$ is an unbalanced flower contradicting Claim \ref{clm:notbal}.
%and $W$ is mixed critical by Lemma \ref{lem:flower2}. Replacing $W$ by $W\cup X$ if necessary, we may assume that $X\subseteq W$. The properties of an unbalanced flower now tell us that the component of $G-\{u,z\}$ which contains $\{v,w\}$ is loopless and contradicts the hypothesis that $G$ is balanced.

It remains to consider the case when $N(v)=\{w,u,z\}\not\subseteq Y$.  Since $v$ has at least two neighbours in $Y$ we may assume that $u\in Y$ and $w\in X$. Since $v$ is not admissible in $H_m$, Lemma \ref{lem:flower3} and Claim \ref{clm:notbal} imply that there exists a strong or weak flower $(W,U,Z)$ on $v$ in $H_m$ with
$u,z\in W\subseteq V_m\sm \{v,w\}$, $w,z\in U\subseteq V_m\sm \{v,u\}$, and $w,u\in Z\subseteq V_m\sm \{v,z\}$. 
%We may assume that $W,U,Z$ are minimal critical sets with these properties. 
If $Z$ is mixed critical or $|X\cap Z|\geq 2$ then $Z\cup X$ is a mixed node critical set for $v$. Hence we may assume that $Z$ is not mixed critical and $X\cap Z=\{w\}$.

Thus $(W,U,Z)$ is a weak flower. Lemma \ref{lem:flower1} and the fact that  $Z$ is not mixed critical now imply that  $W$ is mixed critical and $i_L(W^\dag)=0=d(W^\dag, Z^\dag)$. Furthermore, since $X\cap Z=\{w\}$, we have $X-w\subseteq Z^\dag$. This implies that $w$ is an isolated vertex of $H_m[X]$ and contradicts the fact that $X$ is mixed critical in $H_m$. 
%Hence $(W,U,Z)$ is not a weak flower, completing the proof of the claim.
%
%The fact that $Z$ is not mixed critical, now implies that  $(W,U,Z)$ is an unbalanced flower on $v$, contradicting Claim \ref{clm:notbal}.
%Since $X$ is mixed critical, each connected component of $H_m[X]$ contains a loop. Lemma \ref{lem:flower2} and the fact that $X\cap Z=\{w\}\subseteq  X\cap U$, now implies that $U$ is mixed critical, $W$ is pure critical, and no vertex of $(W\cup Z)\sm \{w,z\}$ is incident with a loop in $H_m$. The facts that $H_m[X]$ is rigid and no vertex of $X\cap (W-z)$ is incident with a loop now imply that $X\cap (W-z)=\emptyset$, and hence $X\subseteq U$. This contradicts the hypothesis that $G$ is balanced since the component of $G-\{w,z\}$ which contains $\{v,u\}$ will be loopless.  
\end{proof}

Choose a node $v$ of $H_m$ in $Y$ and a mixed node critical set $X_v$ for $v$ in $H_m$ which satisfy the conditions of Claim \ref{clm:nodecritical} and are such that $|X_v|$ is as large as possible over all such choices of $v$ and $X_v$. Let $Y_v=V_m\setminus X_v$. Since $X_v$ is mixed node critical for $v$, $|Y_v|\geq 2$. Let $F_v=H_m^\dag[V_3\cap Y_v]$. By Lemma \ref{lem:minus} we can choose a node $z\neq v$ in $Y_v$ such that $d_{F_v}(z)\leq 1$. By Claim \ref{clm:node}, $z$ has at least two distinct neighbours in $G$.

Suppose $z$ has exactly two neighbours, say $a,b$, in $G$. If $\{a,b\}\cap  X_v \neq \emptyset$ then the set $X_v'=X_v+z$ would contradict the maximality of $X_v$. Hence $\{a,b\}\subset Y_v$ and since $d_{F_v}(z)\leq 1$
%since $z$ is a leaf node in $G^\dag[V_3\cap Y_v]$ 
we may assume that $d^\dag(b)\geq 4$. Since $z$ is not admissible in $H_m$ we have $a\in X_z\subseteq V\setminus \{z,b\}$ for some mixed critical set $X_z$. Then 
%$X_z$ is mixed node-critical for $z$. The set 
$X_z'=X_z\cup X_v$ is  mixed node critical for $z$ and contradicts the maximality of $X_v$. Hence  $z$ has three distinct neighbours in $G$. 

The facts that $H_m$ is an $\M_{lc}$-circuit, $X_v$ is mixed critical and $v\not\in X_v+z$ imply that $z$ has at most two neighbours in $X_v$.  If $z$ had exactly two neighbours in $X_v$, then $X_v\cup \{z\}$ would be a mixed node critical set for $v$ and would contradict the maximality of $X_v$. Hence $z$ has at most one neighbour in $X_v$.
Consider the following two cases.

\textbf{Case 1:} $z$ is  a series node in $H_m$.

Let $N(z)=\{p,q,t\}$. 
Since $z$ is a series node in $H_m$ and $d_{F_v}(z)\leq 1$, $z$ has exactly one neighbour, say $t$, in $X_v$ and $t$ is a node in $H_m$.
Without loss of generality, we may assume $d^\dag(p)=3$ and $d^\dag(q)\geq 4$. Since $z$ is non-admissible, Lemma \ref{lem:admcases}(a) implies there exists a (pure or mixed) critical set $X_z$ with $\{t,p\}\subseteq X_z\subseteq V\sm \{z,q\}$. Since $H_m^\dag[V_3]$ is a forest by Lemma \ref{lem:forest}, we have $pt\notin E$ and hence $X_z\neq \{t,p\}$.
Since $X_v,X_z$ are critical, $t$ is a node, $t\in X_v\cap X_z$ and $z\not\in X_v\cup X_z$, all neighbours of $t$ other than $z$ belong to $X_v\cap X_z$. Lemma \ref{lem:union}(b) or (c) now implies that $X_v\cup X_z$ is mixed critical. Since 
$\{t,p\}\subseteq X_v\cup X_z\subseteq V\sm \{z,q\}$ and $d^\dag(q)\geq 4$, $X_v\cup X_z$ is mixed node critical for $z$. This contradicts the maximality of $X_v$.

\textbf{Case 2:} $z$ is a leaf node in $H_m$.
%If $|N(z)\cap X_v|\geq 2$ then $X_v+z$ would be the required mixed node critical set (for $v$). Hence we may suppose that $|N(z)\cap X_v|\leq 1$.
%
%We first consider the case when $z$ has at most two distinct neighbours in $G$. If $z$ is incident with two loops, or is incident with one loop and $|N(z)\cap X_v|= 1$, then $X_v+z$ would again be the required mixed node critical set (for $v$). Hence we may suppose that $z$ is incident with exactly one loop and $N(z)\cap X_v=\emptyset$.
%Let $N(z)=\{b,c\}$. Since $z$ is a leaf node we may suppose that $d^\dag(c)\geq 4$.
%Since $z$ is not admissible, Lemma \ref{lem:admcases}(c) implies there is a mixed node critical set $X_b$ containing $b$ but not $c$ or $z$. Lemma \ref{lem:union}(b) implies that $X_v\cup X_b$ is mixed node critical for $z$ and $|X_v\cup X_b| > |X_v|$.
%
%It remains to consider the case when $z$ has three distinct neighbours $z_1,z_2,z_3$.

Let $N(z)=\{z_1,z_2,z_3\}$. 
Since $z$ is non-admissible, Lemma \ref{lem:flower3} and Claim \ref{clm:notbal}  imply there is either a strong or weak flower $(Z_1,Z_2,Z_3)$ on $z$ in $H_m$ with $Z_i\subseteq V_m\setminus \{z,z_i\}$ for $1\leq i \leq 3$.
%By Claim \ref{clm:notbal} $(Z_1,Z_2,Z_3)$ is either a strong or weak flower.
We have $z_iz_j\not\in E$ for all $1\leq i<j\leq 3$ by Lemma \ref{lem:flower1}(b) (since $z_i\in Z_i^\dag$). 
Since $z$ is a leaf node in $H_m$, we may suppose that $d^\dag(z_1)\geq 4$ and $d^\dag(z_2)\geq 4$.
Since  neither $z_1$ nor $z_2$ are nodes in $H_m$, $Z_1$ and $Z_2$ are two node critical sets for $z$ with $Z_1\cup Z_2=V_m-z$ (by Lemma \ref{lem:flower1}(a)) and at least one of them, say $Z_1$, is mixed critical. 
Then $Z_1\cup X_v$ is mixed critical. If $z_1\not\in X_v$ then  $Z_1\cup X_v$ would be a mixed node critical set for $z$ which would be larger than $X_v$ since $z_2,z_3\in Z_1$ and $|N(z)\cap X_v||\leq 1$. This would contradict the maximality of $X_v$ so we must have $z_1\in X_v$. The fact that $|N(z)\cap X_v|\leq 1$ now implies that $z_2,z_3\not\in X_v$. It follows that, if either $Z_2$ is mixed critical or $|Z_2\cap X_v| \geq 2$, then $Z_2\cup X_v$ would be a mixed node critical set for $z$ which is larger than $X_v$. 
Hence $Z_2$ is pure critical and $Z_2\cap X_v =\{z_1\}$.

%If $|Z_1|=2$ then we would have $Z_1=\{z_2,z_3\}$ and $Z_1\cap Z_2=\{z_3\}$. Since $z_2z_3\not\in E$ and $Z_1\cup Z_2=V_m-z$, this would contradict the fact that 
%$z_2$ is not a node. Hence $|Z_1|\geq 3$. 
%
%If $|X_v|\leq 2$ and $z_1\notin X$ then $X\cup Z_1$ would be a mixed node critical set for $z$ with $|X\cup Z_1|>|X_v|$, and would contradict the maximality of $X_v$. Moreover if $z_i\in X$ then we can assume that $|X\cap Z_2|=1$ (otherwise $X\cup Z_2$ would be the required mixed node critical set for $z$, with $|Z_2|\geq 3$ similarly to above). Lemma \ref{lem:flower1}(a) implies that $X\cup Z_1=V_m-z$ and $Z_1\cup Z_2=V_m-z$ so we must have $(X-z_1)\subseteq Z_1$. Thus $V_m\sm Z_1=\{z,z_1\}$. Since $z_1$ is not a node this contradicts Lemma \ref{lem:minus}.
%Hence we can assume that $|X_v|\geq 3$.
%
%Thus $(W,U,Z)$ is a weak flower. 

We complete the proof by using a similar argument to the last paragraph of the proof of Claim \ref{clm:nodecritical}.
Lemma \ref{lem:flower1} and the fact that  $Z_2$ is not mixed critical imply that $(Z_1,Z_2,Z_3)$ is a weak flower, and $i_L(Z_2^\dag)=0=d(Z_1^\dag, Z_2^\dag)$. Furthermore, 
since $X_v\cap Z_2=\{z_1\}$, we have $X_v-z_1\subseteq Z_2^\dag$. This implies that $z_1$ is an isolated vertex of $H_m[X_v]$ and contradicts the fact that $X_v$ is mixed critical in $H_m$. 
\end{proof}

\subsection*{Recursive constructions}
 We close this section by using Theorem \ref{thm:balancedrradmissiblenode} to obtain a recursive construction for rigid  $\mathcal{M}_{lc}$-connected graphs.   It uses the special case of the 2-sum operation in which one side of the 2-sum is a copy of $K_4$. We will refer to this operation and its inverse as a {\em $K_4$-extension} and a {\em $K_4$-reduction}, respectively. 

We will need the following result on admissible reductions of loopless $\mathcal{M}_{lc}$-connected graphs and an extension of Lemma \ref{lem:unbalanced} to $\mathcal{M}_{lc}$-connected graphs.
 
\begin{thm}\label{thm:admissible_simple}
Let $G=(V,E)$ be an  $\mathcal{M}_{lc}$-connected simple graph which is distinct from $K_4$ and $uv,xy\in E$. \\
(a) If $G$ is an $\mathcal{M}_{lc}$-circuit then some vertex of $V\sm \{u,v,x\}$ is an admissible node in $G$.\\
(b) If $G$ is not an $\mathcal{M}_{lc}$-circuit then either some edge of $E\sm \{uv,xy\}$ is admissible in $G$, or some vertex of $V\sm \{u,v,x,y\}$ is an admissible node in $G$.
\end{thm}

\begin{proof}
Part (a) follows immediately from  \cite[Theorem 3.8]{BJ}. To see (b) we choose an $\mathcal{M}_{lc}$-circuit $H_1$ in $G$ containing $uv,xy$ and then extend $H_1$ to an ear decomposition $H_1,H_2,\ldots,H_m$ of $G$. Then \cite[Theorem 5.4]{J&J} implies that either some edge of $H_m$ distinct from $uv,xy$ is admissible in $G$ or some vertex of $H_m-\bigcup_{i=1}^{m-1}H_{m-1}$ is an admissible node of $G$.
\end{proof}

An {\em unbalanced $2$-separator} of a looped simple graph $G$ is a pair of vertices $\{u,v\}$ such that  $G-\{u,v\}$ has a component with no loops. Note that we allow this loopless component to be equal to $G-\{u,v\}$. An {\em unbalanced $2$-separation} of $G$ is an ordered pair of subgraphs $(G_1,G_2)$ such that $G=G_1\cup G_2$,
$|V(G_1)\cap V(G_2)|=2<|V(G_2)|$ and $E(G_1)\cap E(G_2)=\emptyset=L(G_2)$.

\begin{lem}\label{lem:feasible1}
Let $G=(V,E,L)$ be a looped simple graph 
with $L\neq \emptyset$, 
and $(G_1,G_2)$ be an unbalanced $2$-separator of $G$ such that
$V(G_1)\cap V(G_2)=\{u,v\}$  and $uv\not\in E$.
Then the following statements are equivalent:\\
(a) $G+uv$ is $\mathcal{M}_{lc}$-connected;\\
(b) $G_1+uv$ and $G_2+uv$ are both $\mathcal{M}_{lc}$-connected;\\
(c) $G$ is $\mathcal{M}_{lc}$-connected.
\end{lem}

\begin{proof}
Let $G_i+uv=(V_i,E_i,L_i)$ for $i=1,2$. 
By symmetry we may suppose $L_1\neq \emptyset =L_2$. Choose $f_1\in L_1$ and $g_2\in E_2-uv$.

\smallskip\noindent
(a)$\Rightarrow$(b).
Suppose that $G+uv$ is $\M_{lc}$-connected.  Then there exists an $\M_{lc}$-circuit $C$ in $G+uv$ containing $f_1,g_2$.   Lemma \ref{lem:unbalanced} now implies that $C$ is the 2-sum of two $\M_{lc}$-circuits $C_1,C_2$ with $f_1,uv\in C_1\subseteq G_1+uv$ and $g_2,uv\in C_2\subseteq G_2+uv$. Transitivity now implies that $G_2+uv$ is $\M_{lc}$-connected and that all loops of $G_1+uv$ belong the the same $\M_{lc}$-connected component $H_1$ of $G_1+uv$ as  $uv$. To complete the proof we choose an edge $g_1\in E_1$ and show $g_1$ is in $H_1$. Since $G+uv$ is $\M_{lc}$-connected, there exists an $\M_{lc}$-circuit $C'$ in $G+uv$ containing $f_1,g_1$. If $C'\subseteq G_1+uv$ then we are done. On the other hand, if $C'\not\subseteq G_1+uv$ then Lemma \ref{lem:unbalanced} implies that $(C'\cap G_1)+uv$ is an $\M_{lc}$-circuit  in $G_1+uv$ containing $f_1,g_1,uv$.

\smallskip\noindent
(b)$\Rightarrow$(c).
Suppose $G_1+uv$ and $G_2+uv$ are both $\mathcal{M}_{lc}$-connected. Then there exists an $\M_{lc}$-circuit $C_i$ in $G_i+uv$ such that $C_1$ contains $f_1,uv$ and $C_2$ contains $g_2,uv$. By  
Lemma \ref{lem:unbalanced}, $C=(C_1-uv)\cup (C_2-uv)$ is an $\M_{lc}$-circuit in $G$. Transitivity now implies that all loops of $G_1$ and all edges of $G_2$ belong the the same $\M_{lc}$-connected component $H$ of $G$. To complete the proof we choose an edge $g_1\in E_1-uv$ and show $g_1$ is in $H_1$. Since $G_1+uv$ is $\M_{lc}$-connected, there exists an $\M_{lc}$-circuit $C_1'$ in $G+uv$ containing $f_1,g_1$. If $C_1'\subseteq G$ then we are done. On the other hand, if $C_1'\not\subseteq G_1+uv$ then $uv\in C_1'$ and Lemma \ref{lem:unbalanced} implies that $(C_1'-uv)\cup (C_2-uv)$ is an $\M_{lc}$-circuit  in $G$ containing $f_1,g_1$.

\smallskip\noindent
(c)$\Rightarrow$(a).
This follows immediately from Lemma \ref{lem:feasible2}. 
\end{proof}

We can now give our recursive construction.

\begin{thm}\label{thm:rec_unbalanced}
A looped simple graph  is rigid and $\mathcal{M}_{lc}$-connected if and only if it can be obtained from $K_1^{[3]}$ by recursively applying the operations of  $1$-extension, $K_4$-extension and adding a new edge or loop.
\end{thm}

\begin{proof}
Sufficiency follows from Lemma \ref{lem:feasible2} and the fact that $K_1^{[3]}$ is $\mathcal{M}_{lc}$-connected. To prove necessity it will suffice to show that every rigid $\mathcal{M}_{lc}$-connected graph $G$ with at least two vertices can be reduced to a smaller rigid $\mathcal{M}_{lc}$-connected graph by applying either a 1-reduction, a $K_4$-reduction or an edge deletion. If $G$ is balanced then Theorem \ref{thm:balancedrradmissiblenode} implies there is an edge, loop or node such that the corresponding  edge/loop deletion or 1-reduction gives a rigid $\mathcal{M}_{lc}$-connected graph.

Hence we may assume that $G$ is unbalanced. By Lemma \ref{lem:feasible1}, $G$ is the 2-sum of an $\M_{lc}$-connected graph $G_1$ with at least one loop and a $\M_{lc}$-connected graph $G_2$ with no loops along an edge $uv$. We may assume that $G_2$ is 3-connected by choosing a 2-sum such that $G_2$ is as small as possible. If $G_2=K_4$, then $G$ can be reduced to $G_1$ by applying the $K_4$-reduction operation.

Hence we may assume that $G_2\neq K_4$.  Then Theorem \ref{thm:admissible_simple} implies that $G_2$ contains either an admissible edge distinct from $uv$, or an admissible node $x$  distinct from $u,v$. The graph $G'$ obtained from $G$ by performing the same reduction operation  will then be $\mathcal{M}_{lc}$-connected since it is 
the 2-sum of $G_1$ and  $G_2'$ along $uv$. 
\end{proof}

Figure \ref{fig:rec_unbalanced} provides an illustration of Theorem \ref{thm:rec_unbalanced}. In this figure a thick edge or
loop means the next step in the construction will be a 1-extension which deletes that edge or loop, and a dashed edge
or loop means the next step in the construction will add that edge or loop.
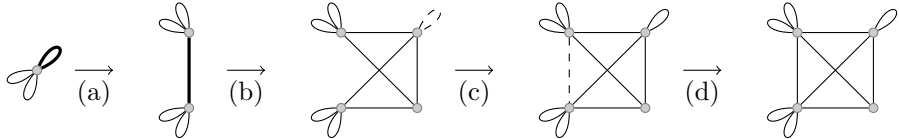
\begin{figure}[h]
\begin{center}
\begin{tikzpicture}[font=\small]
\begin{scope}[yshift=0.5cm]
\node[roundnode] at (0,0) (v1) []{}
	edge[in=180,out=220,loop]()
	edge[in=230, out=270,loop]()
	edge[in=25,out=65,loop,very thick]();
\draw[->] (0.5,0)--(1,0)  node[pos=0.5, below] {(a)};

\end{scope}
\begin{scope}[xshift=2cm]
\node[roundnode] at (0,0) (v1) []{}
	edge[in=180,out=220,loop]()
	edge[in=230, out=270,loop]();
\node[roundnode] at (0,1) (v4) []{}
	edge[very thick] (v1)
	edge[in=90, out=130,loop]()
	edge[in=140, out=180,loop]();
\draw[->] (0.5,0.5)--(1,0.5) node[pos=0.5, below] {(b)};
\end{scope}
\begin{scope}[xshift=4cm]
\node[roundnode] at (0,0) (v1) []{}
	edge[in=180,out=220,loop]()
	edge[in=230, out=270,loop]();
\node[roundnode] at (1,0) (v2) []{}
	edge[] (v1);
\node[roundnode] at (1,1) (v3) []{}
	edge[] (v1)
	edge[] (v2)
	edge[in=25, out=65,loop, dashed]();
\node[roundnode] at (0,1) (v4) []{}
	edge[] (v2)
	edge[] (v3)
	edge[in=90, out=130,loop]()
	edge[in=140, out=180,loop]();
\draw[->] (1.5,0.5)--(2,0.5) node[pos=0.5, below] {(c)};
\end{scope}
\begin{scope}[xshift=7cm]
\node[roundnode] at (0,0) (v1) []{}
	edge[in=180,out=220,loop]()
	edge[in=230, out=270,loop]();
\node[roundnode] at (1,0) (v2) []{}
	edge[] (v1);
\node[roundnode] at (1,1) (v3) []{}
	edge[] (v1)
	edge[] (v2)
	edge[in=25, out=65,loop]();
\node[roundnode] at (0,1) (v4) []{}
	edge[] (v2)
	edge[] (v3)
	edge[dashed] (v1)
	edge[in=90, out=130,loop]()
	edge[in=140, out=180,loop]();
\draw[->] (1.5,0.5)--(2,0.5) node[pos=0.5, below] {(d)};
\end{scope}
\begin{scope}[xshift=10cm]
\node[roundnode] at (0,0) (v1) []{}
	edge[in=180,out=220,loop]()
	edge[in=230, out=270,loop]();
\node[roundnode] at (1,0) (v2) []{}
	edge[] (v1);
\node[roundnode] at (1,1) (v3) []{}
	edge[] (v1)
	edge[] (v2)
	edge[in=25, out=65,loop]();
\node[roundnode] at (0,1) (v4) []{}
	edge[] (v2)
	edge[] (v3)
	edge[] (v1)
	edge[in=90, out=130,loop]()
	edge[in=140, out=180,loop]();
\end{scope}
\end{tikzpicture}
\end{center}
\caption{A construction of the rigid $\M_{lc}$-connected graph drawn on the far right from a copy of $K_1^{[3]}$, by (a) 1-extension on a loop, (b) $K_4$-extension, (c) loop addition and (d) edge addition.}
\label{fig:rec_unbalanced}
\end{figure}

Theorem \ref{thm:rec_unbalanced} implies in particular that all rigid $\mathcal{M}_{lc}$-circuits can be constructed from $K_1^{[3]}$ by applying the operations of  $1$-extension, $K_4$-extension and adding a new edge or loop. Since every rigid $\mathcal{M}_{lc}$-circuit $G=(V,E,L)$ satisfies $|E|+|L|=2|V|$ we can never use the edge/loop addition operation when constructing a rigid circuit from $K_1^{[3]}$. This gives the following result.

\begin{thm}
A looped simple graph is a rigid $\M_{lc}$-circuit if and only if it can be obtained from $K_1^{[3]}$ by recursively applying the $1$-extension and $K_4$-extension operations.
\end{thm}

\section{Balanced $\mathcal{M}_{lc}$-connected graphs}
\label{sec:feasible}

We next consider reductions that preserve balance as well as $\mathcal{M}_{lc}$-connectivity. 
%To this end let us say that a graph $G$ is a \emph{brick} if it is balanced and $\mathcal{M}_{lc}$-connected.
%, connected and redundantly rigid. 
%We will need some new terminology and several lemmas.
We will need two more structural results in unbalanced separations.

\begin{lem}\label{lem:2setdeg4}
Let $G$ be $\M_{lc}$-connected and let $\{u,v\}$ be an unbalanced $2$-separator in $G$. Then $u$ and $v$ each have at least four incident edges or loops in $G$. Furthermore, if $G'$ is obtained from $G$ by performing an edge/loop deletion or $1$-reduction, then 
$\{u,v\}$ is an unbalanced $2$-separator in $G'$.
\end{lem}

\begin{proof} Let $(G_1,G_2)$ be an unbalanced $2$-separation in $G$ with $V(G_1)\cap V(G_2)=\{u,v\}$. Then $G_1+uv$ and $G_2+uv$ are $\M_{lc}$-connected by Lemma \ref{lem:feasible1}. Hence $u$ and $v$ are each incident with at least three edges or loops in $G_i+uv$ for $i=1,2$.  This implies that $u$ and $v$ each have at least four incident edges or loops in $G$. In particular, neither $u$ nor $v$ is a node of $G$. It is now straightforward to check that $\{u,v\}$ is an unbalanced $2$-separator in $G'$.
\end{proof}

Our next lemma shows that a rigid looped simple graph cannot contain two `crossing' unbalanced 2-separations.

\begin{lem}\label{lem:crossing}
Supose that $G=(V,E,L)$ is rigid looped simple graph and $(G_1,G_2), (G_1',G_2')$ are two unbalanced $2$-separations in $G$ with $G_i=(V_i,E_i,L_i)$ and $G_i'=(V_i',E_i',L_i')$ for $i=1,2$, $L_2=\emptyset=L_2'$, $V_1\cap V_2=\{u,v\}$, $V_1'\cap V_2'=\{u',v'\}$ and $\{u,v\}\cap \{u',v'\}=\emptyset$. Then $\{u',v'\}\subseteq V_i$ for some $i\in \{1,2\}$.
\end{lem}

\begin{proof} Suppose for a contradiction that $u'\in V_1$ and $v'\in V_2$. If 
$\{u,v\}\subseteq V_i'$ for some $i\in \{1,2\}$ then either $G-u'$ or $G-v'$ would have a loopless component.
This would contradict the hypothesis that $G$ is rigid and hence we may assume that $u\in V_1'$ and $v\in V_2'$.

Let $H_1=G_1[V_1\cap V_1']$, $H_2=G_2[V_1'\cap V_2]$, $H_3=G_2[V_2\cap V_2']$,
$H_4=G_2'[V_1\cap V_2']$, and put $n_i=|V(H_i)|$ for $1\leq i\leq 4$, see Figure \ref{fig:crossing_separators}.
Then $H_1,H_2,H_3,H_4$ cover $E\cup L$ and  $H_2,H_3,H_4$ are loopless. This gives
$$r(G)\leq \sum_{i=1}^4 r(H_i)\leq 2n_1+\sum_{i=2}^4 (2n_i-3)=2|V|-1$$ 
and again contradicts the hypothesis that $G$ is rigid.
\end{proof}

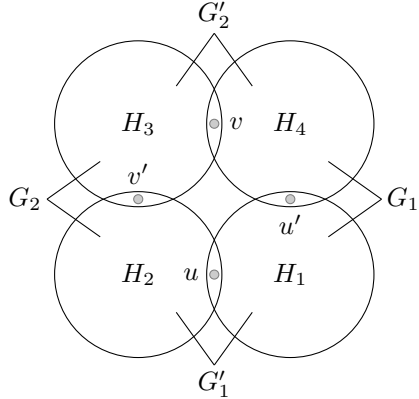
\begin{figure}[h]
\begin{center}
\begin{tikzpicture}[font=\small]
\draw (0,0) ellipse (1.1cm and 1.1cm);
\node[] at (0,0) (H2) [label=center:$H_2$] {};
\draw (0,2) ellipse (1.1cm and 1.1cm);
\node[] at (0,2) (H2) [label=center:$H_3$] {};
\draw (2,2) ellipse (1.1cm and 1.1cm);
\node[] at (2,2) (H2) [label=center:$H_4$] {};
\draw (2,0) ellipse (1.1cm and 1.1cm);
\node[] at (2,0) (H2) [label=center:$H_1$] {};

\node[roundnode] at (1,0) (u) [label=left:$u$]{};
\node[roundnode] at (1,2) (v) [label=right:$v$]{};
\node[roundnode] at (2,1) (u') [label=below:$u'$]{};
\node[roundnode] at (0,1) (v') [label=above:$v'$]{};

\node[] at (3,1) (V1) [label=right:$G_1$]{};
\draw (3.2,1)--(2.5,1.5);
\draw (3.2,1)--(2.5,0.5);

\node[] at (-1,1) (V2) [label=left:$G_2$]{};
\draw (-1.2,1)--(-0.5,1.5);
\draw (-1.2,1)--(-0.5,0.5);

\node[] at (1,-1) (V11) [label=below:$G_1'$]{};
\draw (1,-1.2)--(1.5,-0.5);
\draw (1,-1.2)--(0.5,-0.5);

\node[] at (1,3) (V22) [label=above:$G_2'$]{};
\draw (1,3.2)--(1.5,2.5);
\draw (1,3.2)--(0.5,2.5);

\end{tikzpicture}
\end{center}
\caption{The subgraphs $H_1,H_2,H_3,H_4$ in the proof of Lemma \ref{lem:crossing}.}
\label{fig:crossing_separators}
\end{figure}

An edge or loop  $f$ in a balanced $\mathcal{M}_{lc}$-connected graph $G$ is \emph{feasible} if $G-f$ is balanced and $\mathcal{M}_{lc}$-connected, and a node $v$ is \emph{feasible} if there exists a 1-reduction at $v$ which results in a balanced $\mathcal{M}_{lc}$-connected graph.
Figure \ref{fig:non-adm_adm-non-fea_fea} illustrates the difference between admissibility and feasibility.
On the far left, the vertex $x$ is not admissible in $H$. Each 1-reduction at $x$ that adds a loop creates a vertex of degree 2, and the 1-reduction at $x$
which adds an edge results in a graph with only 3 loops, none of which is contained in an $\M_{lc}$-circuit.
The only admisible 1-reduction at the vertex $y$ in the graph $G$ is the one that adds the edge $y_1y_2$. However,
the graph $G-y+y_1y_2$ is not balanced since $(G-y+y_1y_2)-\{u,v\}$ has no loops. Therefore $y$ is admisible but not feasible in $G$. The 1-reduction at $y_1$
which adds the edge $vy$ results in a balanced rigid $\M_{lc}$-circuit, so the vertex $y_1$ is feasible in $G$.

\begin{figure}[h]
\begin{center}
\begin{tikzpicture}[font=\small]
\node[roundnode] at (0,0) (v1) []{}
	edge[in=205,out=245,loop]();
\node[roundnode] at (1,0) (v2) []{}
	edge[] (v1)
	edge[in=270, out=310,loop]()
	edge[in=320, out=360,loop]();
\node[roundnode] at (1,1) (v3) []{}
	edge[] (v1)
	edge[] (v2);
\node[roundnode] at (0,1) (v4) []{}
	edge[] (v1)
	edge[] (v2);
\node[roundnode] at (0.5,1.5) (x) [label=right:$x$]{}
	edge[] (v3)
	edge[] (v4)
	edge[in=70,out=110,loop]();
\node[] at (0.5,-0.4) (H) [label=below:$H$]{};
\begin{scope}[xshift=3cm]
\node[roundnode] at (0,0) (v1) [label=above:$v$]{}
	edge[in=180,out=220,loop]()
	edge[in=230, out=270,loop]();
\node[roundnode] at (0,1) (v2) [label=below:$u$]{}
	edge[in=90, out=130,loop]()
	edge[in=140, out=180,loop]();
\node[roundnode] at (1,0) (u1) [label=below:$y_1$]{}
	edge[] (v1)
	edge[] (v2);
\node[roundnode] at (1,1) (u2) [label=above:$y_2$]{}
	edge[] (v1)
	edge[] (v2);
\node[roundnode] at (1.5,0.5) (y) [label=below:$y$]{}
	edge[] (u1)
	edge[] (u2)
	edge[in=20, out=340,loop]();
\node[] at (0.5,-0.4) (G) [label=below:$G$]{};
\end{scope}
\begin{scope}[xshift=6cm]
\node[roundnode] at (0,0) (v1) [label=above:$v$]{}
	edge[in=180,out=220,loop]()
	edge[in=230, out=270,loop]();
\node[roundnode] at (0,1) (v2) [label=below:$u$]{}
	edge[in=90, out=130,loop]()
	edge[in=140, out=180,loop]();
\node[roundnode] at (1,0) (u1) [label=below:$y_1$]{}
	edge[] (v1)
	edge[] (v2);
\node[roundnode] at (1,1) (u2) [label=above:$y_2$]{}
	edge[] (v1)
	edge[] (v2)
	edge[] (u1);
\node[] at (0.5,-0.4) (Gy) [label=below:$G-y+y_1y_2$]{};
\end{scope}
\begin{scope}[xshift=9cm]
\node[roundnode] at (0,0) (v1) [label=above:$v$]{}
	edge[in=180,out=220,loop]()
	edge[in=230, out=270,loop]();
\node[roundnode] at (0,1) (v2) [label=below:$u$]{}
	edge[in=90, out=130,loop]()
	edge[in=140, out=180,loop]();
\node[roundnode] at (1,1) (u2) [label=above:$y_2$]{}
	edge[] (v1)
	edge[] (v2);
\node[roundnode] at (1.5,0.5) (y) [label=below:$y$]{}
	edge[] (u2)
	edge[] (v1)
	edge[in=20, out=340,loop]();
\node[] at (0.5,-0.4) (Gy1) [label=below:$G-y_1+vy$]{};
\end{scope}
\end{tikzpicture}
\end{center}
\caption{The vertex $x$ is non-admissible in $H$, the vertex $y$ is admissible but non-feasible in $G$, and the vertex $y_1$ is feasible in $G$.}
\label{fig:non-adm_adm-non-fea_fea}
\end{figure}
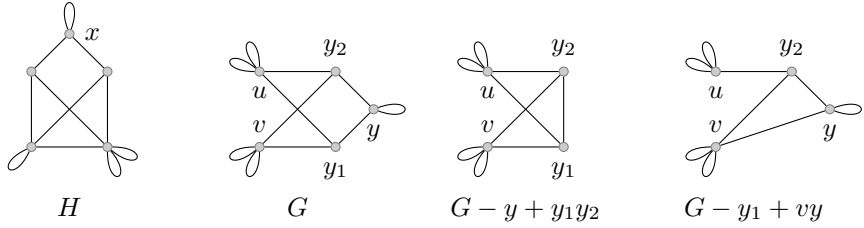
\begin{thm}\label{thm:feasible}
Let $G=(V,E,L)$ be a balanced, $\mathcal{M}_{lc}$-connected looped simple graph distinct from $K_1^{[3]}$. Then some edge, loop, or node of $G$ is feasible.
\end{thm}

\begin{proof}
%Let $G=(V,E,L)$ be balanced, $\mathcal{M}_{lc}$-connected, with $|V|\geq 2$ and 
Suppose, for a contradiction,  that all possible edge/loop deletions and node 1-reductions of $G$ fail to be either $\mathcal{M}_{lc}$-connected or balanced. Theorem \ref{thm:balancedrradmissiblenode} implies that $G$ contains either an admissible node $w$ or an admissible edge or loop $f$. By assumption neither $w$ nor $f$ is feasible. Let $G'$ be the result of deleting $f$ or performing an admissible 1-reduction at $w$. Then $G'$ is $\mathcal{M}_{lc}$-connected and not balanced. Hence $G'$ contains an unbalanced 2-separation $H_1,H_2$ where $H_1$ is loopless and  $V(H_1)\cap V(H_2)=\{u,v\}$.
 We may suppose that 
%$H_1-\{u,v\}$ is a loopless component of $G'-\{u,v\}$ and that 
the pair $(r,\{u,v\})$, where $r\in \{w,f\}$, has been chosen so that 
$X=V(H_1)\sm\{u,v\}$ is as small as possible.
%$H_1$ has as few vertices as possible. 

Let $H_1^+=H_1+uv$ and $H_2^+=H_2+uv$. 
%$H=H_1+uv$ and $L=H_2+uv$. 
\begin{claim}\label{claim:-2}
$H_1^+$ and $H_2^+$ are $\mathcal{M}_{lc}$-connected and $H_2^+$ is redundantly rigid.
\end{claim}

\begin{proof}[Proof of Claim]
Corollary \ref{cor:mlcredundant} implies that $H_1^+$ and $H_2^+$ are $\mathcal{M}_{lc}$-connected. Corollary \ref{cor:mlcredundant} now implies that $H_2^+$ is redundantly rigid.
\end{proof}

%It follows from Lemma \ref{lem:feasible3} that $H$ and $L$ are $\mathcal{M}_{lc}$-connected graphs. Also 
Note that the minimality of $X$ implies that $H_1^+$ is 3-connected. %Lemma \ref{lem:Mconnected} implies that 
%$H$ is redundantly rigid as a bar-joint framework by \cite{J&J}.

\begin{claim}\label{claim:-1}
$uv\notin E$. 
\end{claim}

\begin{proof}[Proof of Claim]
Suppose $uv\in E$. Since $G'$ is $\mathcal{M}_{lc}$-connected and $\{u,v\}$ is an unbalanced  2-separator in $G'$, Lemma \ref{lem:feasible1} implies that $G'-uv$ is $\mathcal{M}_{lc}$-connected. Since $G-uv$ is obtained from $G'-uv$ by an edge addition or a 1-extension, $G-uv$ is $\mathcal{M}_{lc}$-connected by Lemma \ref{lem:feasible2}. It remains to show that $G-uv$ is balanced. 

Suppose $\{u',v'\}$ is an unbalanced 2-separator in $G-uv$. Since $G$ is balanced, $u$ and $v$ belong to different components of $(G-uv)-\{u',v'\}$, at least one of which is loopless. Furthermore, $\{u',v'\}$ is an unbalanced 2-separator in $G'-uv$ by Lemma \ref{lem:2setdeg4}.
This implies that $G'-uv$ does not contain three internally disjoint $uv$-paths. Since $H_1^+$ is 3-connected, there are two internally disjoint $uv$-paths in $H_1^+-uv$.
Hence $u',v'\in X$ and $u$ and $v$ belong to different components $J_u$ and $J_v$ of $H_2^+-uv$,
see Figure \ref{fig:uv_paths} for an illustration. Note that $J_u$ and $J_v$ both contain loops since $G'$ is rigid.
This contradicts the fact that either $u$ or $v$ belongs to a loopless component of $(G-uv)-\{u',v'\}$
since $G-uv$ is obtained from $G'-uv$ by applying a 1-extension or edge/loop addition.
% Then   $G'-\{u',v'\}-uv$ has exactly two components and each component contains at least one loop since it contains either $J_u$ or $J_v$. This contradicts the assumption that $\{u',v'\}$ is an unbalanced 2-separator of $G-uv$.
\end{proof}

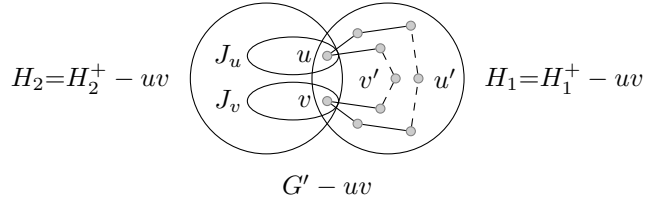
\begin{figure}[h]
\begin{center}
\begin{tikzpicture}[font=\small]
\draw (0,0) circle (1cm);
\node[] at (-1,0) (H2) [label=left:$H_2\text{=}H_2^+-uv$]{};
\draw (1.6,0) circle (1cm);
\node[] at (2.6,0) (H1) [label=right:$H_1\text{=}H_1^+-uv$]{};
\node[roundnode] at (0.8,0.3) (u) [label=left:$u$]{};
\draw (0.35,0.3) ellipse (0.6cm and 0.25cm);
\node[] at (-0.5,0.3) (Ju) [label=center:$J_u$]{};
\node[roundnode] at (0.8,-0.3) (v) [label=left:$v$]{};
\node[] at (-0.5,-0.3) (Jv) [label=center:$J_v$]{};
\draw (0.35,-0.3) ellipse (0.6cm and 0.25cm);

\node[roundnode] at (1.2,0.6) (a1) []{}
	edge[] (u);
\node[roundnode] at (1.9,0.7) (a2) []{}
	edge[] (a1);
\node[roundnode] at (2,0) (u') [label=right:$u'$]{}
	edge[dashed] (a2);

\node[roundnode] at (1.9,-0.7) (a3) []{}
	edge[dashed] (u');
\node[roundnode] at (1.2,-0.6) (a4) []{}
	edge[] (a3)
	edge[] (v);

\node[roundnode] at (1.5,0.4) (b1) []{}
	edge[] (u);
\node[roundnode] at (1.7,0) (v') [label=left:$v'$]{}
	edge[dashed] (b1);
\node[roundnode] at (1.5,-0.4) (b2) []{}
	edge[dashed] (v')
	edge[] (v);
\node[] at (0.8,-1) (G') [label=below:$G'-uv$]{};
\end{tikzpicture}
\end{center}
\caption{An illustration of the structure of $G'-uv$ in the proof of Claim \ref{claim:-1}. The 3-connectivity of $H_1^+$ implies there are two disjoint $uv$-paths within $H_1=H_1^+-uv$,
and hence removing $u'$ and $v'$ must destroy these paths.}
\label{fig:uv_paths}
\end{figure}

%Note that $H_1^+$ is  $\mathcal{M}_{lc}$-connected. 
Our strategy in the remainder of the proof is to show that some edge or node of $G$ in $H_1$ is feasible in $G$. We have to be careful when considering the edges and nodes of $H_1$ since not all of them are edges or nodes in $G$. In addition a vertex which is a node in both $H_1$ and $G$ may be incident with different edges in each graph. We use the following notation to handle this.  We first put $E^\dag(H_1)=E(H_1)\cap E(G)$. 
If $r=w$, we let $V^\dag(H_1)=X-N_G(w)$ and,
if $r=f$ and $f=yz$, then we let $V^\dag(H_1)=X-\{y,z\}$. 
If the reduction operation which converts $G$ to $G'$ adds an edge $e$ between two vertices of $H_1$ we put $\theta=e$. Otherwise there is a unique vertex $x$ of $X$ which is incident/adjacent to $r$ and we put $\theta =x$.

%Let $\theta=xy$ if $G'=G-w+xy$ and $xy\in E(H)$, let $\theta=z$ if $G'=G-w+xy$ and $xy\notin E(H)$ and let $\theta$ be the unique vertex of $X$ which is incident to $f$ in $G$ if $G'=G-f$.

\begin{claim}\label{claim:f1}
Suppose that $G-e$ is $\mathcal{M}_{lc}$-connected for some $e\in E^\dag(H_1)$. Then $H_1-\{u,v,e\}$ is connected.
\end{claim}

\begin{proof}[Proof of Claim]
Suppose $H_1-\{u,v,e\}$ has two components $J_1,J_2$. Choose $i\in \{1,2\}$ such that $\theta \notin V(J_i)\cup E(J_i)$. Then $\{u,v\}$ is an unbalanced 2-separation of $G-e$ with $J_i$ as a component and $V(J_i)$ is properly contained in $X$. This contradicts the minimality of $X$.
\end{proof}

\begin{claim}\label{claim:f2}
$G-e$ is not $\mathcal{M}_{lc}$-connected for all $e\in E^\dag(H_1)$.
\end{claim}

\begin{proof}[Proof of Claim]
Suppose that $G-e$ is $\mathcal{M}_{lc}$-connected for some edge $e=ab\in E^\dag(H_1)$. Then $G-e$ is not balanced so there exists an unbalanced 2-separator $T$ in $G-e$. Since $G$ is balanced, $a$ and $b$ are in different components of $(G-e)-T$. 
%Note that each vertex in $T$ has at least four incident edges in $G$ by Corollary \ref{cor:2setdeg4} so $T\subseteq V(G')$.

%Let $S=\{u,v\}$. 
We can apply Lemma \ref{lem:2setdeg4} to $G-e$ and $G'$ respectively to deduce that  $T$ and $S=\{u,v\}$ are unbalanced 2-separators in $G''=G'-e$.
Since $G'$ is $\mathcal{M}_{lc}$-connected and contains a loop, it is redundantly rigid by Corollary \ref{cor:mlcredundant}. Hence $G''=G'-e$ is rigid. 
Lemma \ref{lem:crossing} now implies that $T\subseteq V(H_i)$ for some $i\in\{1,2\}$.
 Since $H_1-\{u,v,e\}$ is connected by Claim \ref{claim:f1}, $G''[X]$ is a component of $G''-S$. Since $a,b\in X\cup S$ and $T$ separates $a,b$ in $G-e$, we have $T\cap X\neq \emptyset$. Hence $T\subseteq V(H_1)$.

Let $X_T$ be the vertex set of  a loopless component of $(G-e)-T$. Since $G$ is balanced, each component of $H_2-S$ contains a loop and hence $X_T\cap (V(H_2-S)=\emptyset$. This implies that $X_T\cup T \subseteq X\cup S\cup \{r\}$ so we may contradict the minimality of $X$ by showing that $r$ and some vertex of $X\cup S$ does not belong to $X_T\cup T$. 
%We use the vertex/edge $r$ to show this. 
We know that $r$ is adjacent/incident to some vertex $x\in X$. If $r$ is also adjacent/incident to some vertex of $V(H_2)\sm S$ then the facts that $r$ does not belong to $G'-e$ and $T$ is an unbalanced 2-separator of $G-e$ gives $r,x\notin X_T\cup T$. Similarly, if $r$ is only adjacent/incident to vertices of $H_1$, then either $r$ is a loop or $r$ is a node incident with a loop and we again have $r,x\notin X_T\cup T$.
%Suppose $H_2^+\sm S$ is connected. 
%Then some component $J'$ of $G''\sm T=(G'-e)\sm T$ contains $V\sm (X\cup S)$. Let $J$ be the component of $G-e\sm T$ which contains $V\sm (X\cup S)$. Then $V\sm X \subset (V(J)\cup N_{G-e}(J))$. Moreover, if $G'$ is a 1-reduction of $G$ at $w$ then the neighbours of $w$ in $X$ are contained in $V(J)\cup N_{G-e}(J)$ and, if $G'=G-f$ then the endvertex of $f$ in $X$ is contained in $V(J)\cup N_{G-e}(J)$. In both cases, this implies that the vertex set of the component of $(G-e)\sm T$ distinct from $J$ is a proper subset of $X$, contradicting the minimality of $X$.
\end{proof}

\begin{claim}\label{claim:f3}
$H_1^+-e$ is not $\mathcal{M}_{lc}$-connected for all $e\in E^\dag(H_1)$.
\end{claim}

\begin{proof}[Proof of Claim]
Suppose $H_1^+-e$ is $\mathcal{M}_{lc}$-connected. 
Then $G'-e$ is the 2-sum of $H_1^+-e$ and $H_2^+$, and $G'-e$ is $\mathcal{M}_{lc}$-connected by Lemma \ref{lem:feasible1}. Lemma \ref{lem:feasible2} now implies that $G-e$ is $\mathcal{M}_{lc}$-connected,
%. Since $G$ is balanced, $G-e$ is nearly balanced and hence Lemma \ref{lem:Mconnected_v2} implies $G-e$ is redundantly rigid 
contradicting Claim \ref{claim:f2}.
\end{proof}

\begin{claim}\label{claim:f4}
Suppose $p\in V^\dag(H_1)$ is a node of $H_1^+$, $N(p)=\{q,s,t\}$ and $G-p+st$ is $\mathcal{M}_{lc}$-connected. Then 
%$p\notin \{u,v\}$ and 
$(H_1-p+st)-\{u,v\}$ is connected.
\end{claim}

\begin{proof}[Proof of Claim]
%Firstly, since $u$ and $v$ have at least four incident edges in $G'$ by Corollary \ref{cor:2setdeg4}, $p\notin \{u,v\}$. 
Suppose $(H_1-p+st)-\{u,v\}$ has two components $J_1,J_2$. 
%Then there exist subgraphs $J_1,J_2$ of $H_1-\{u,v\}$ such that $H_1-\{u,v\}=J_1\cup J_2$,  $V(J_1)\cap V(J_2)=p$, $s,t\in V(J_1)$ and $q\in V(J_2)$. 
Choose $i\in \{1,2\}$ such that $\theta \notin V(J_i)\cup E(J_i)$.
Then $\{u,v\}$ is an unbalanced 2-separation in $G-p+st$ and $J_i$ is a loopless component of $(G-p+st)-\{u,v\}$. This contradicts the minimality of $X$.
\end{proof}

\begin{claim}\label{claim:f5}
Every node $p$ of $H_1^+$ in $V^\dag(H_1)$ is non-admissible in $G$.
\end{claim}

\begin{proof}[Proof of Claim]
Suppose $p\in V^\dag(H_1)$ is a node of $H_1^+$ with $N(p)=\{q,s,t\}$ and $G-p+st$ is $\mathcal{M}_{lc}$-connected. Then $G-p+st$ is not balanced, so has an unbalanced 2-separator $T$. Since $G$ is balanced, $st$ and $q$ are in different components of $(G-p+st)-T$. 
%If $G'=G-w+xy$ then $w\notin T$ (as every vertex in $T$ has at least 4 incident edges in $G-p+st$ by Corollary \ref{cor:2setdeg4}, and hence in $G$, and $w$ is a node in $G$). 

We can apply Lemma \ref{lem:2setdeg4} to $G-p+st$ and $G'$ respectively to deduce that  $T$ and $S=\{u,v\}$ are unbalanced 2-separators in $G''=G'-p+st$.
Since $G'$ is $\mathcal{M}_{lc}$-connected and contains loops it is redundantly rigid by Corollary \ref{cor:mlcredundant}. Hence 
$G'-pq$ is rigid. Since $G'-p$ is obtained from $G'-pq$ by deleting a vertex with two incident edges, it is rigid. Since $G''$ is obtained from $G'-p$ by an edge addition, it is also rigid.
Lemma \ref{lem:crossing} now implies that $T\subseteq V(H_i)$ for some $i\in\{1,2\}$.
 Since $H_1-p+st-\{u,v\}$ is connected by Claim \ref{claim:f1}, $G''[X-p]$ is a component of $G''-S$. Since $q,s,t\in X\cup S$ and $T$ separates $st$ from $q$ in $G''$, we have $T\cap X\neq \emptyset$. Hence $T\subseteq V(H_1)$.

%Since $T$ and $S$ are unbalanced 2-separators in $G''$, $S$ and $T$ do not cross by Lemma \ref{lem:crossing}. Claim \ref{claim:f4} implies $H_1^+-p+st-\{u,v\}$ is connected and hence $G''[X-p]$ is a component of $G''-S$. 

Let $X_T$ be the vertex set of a loopless component of $(G-p+st)-T$. Since $G$ is balanced, each component of $H_2-S$ contains a loop and hence $X_T\cap (V(H_2)-S)=\emptyset$. This implies that $X_T\cup T \subseteq (X\cup S)+r-p$ so we may contradict the minimality of $X$ by showing that $r$ does not belong to $X_T\cup T$. This is trivially true if $r=f$, so we may assume $r=w$.
%We use the vertex/edge $r$ to show this. 
We know that $w$ is adjacent to some vertex $x\in X$. If $w$ is also adjacent to some vertex of $H_2- S$ then the facts that $w$ does not belong to $G'-p+st$ and $T$ is an unbalanced 2-separator of $G-p+st$ give $w\notin X_T\cup T$. Similarly, if $w$ is only adjacent to vertices of $H_1$, then $w$ is a node incident with a loop and we again have $w\notin X_T\cup T$.
%
%Hence some component $J'$ of $G''\sm T$ contains $V\sm (X\cup S)$. Let $J$ be the component of $G-p+st\sm T$ which contains $V\sm (X\cup S)$. Thus $V\sm X\subset V(J)\cup N_{G-p+st}(J)$. Moreover, if $G'$ is formed by a 1-reduction at $w$ then the nieghbours of $w$ in $X$ are also contained in $V(J)\cup N_{G-p+st}(J)$, and if $G'=G-f$ then the endvertex of $f$ in $X$ is contained in $V(J)\cup N_{G-p+st}(J)$. This implies that the vertex set of the component of $G-p+st\sm T$ which is distinct from $J$ is a proper subset of $X$, contradicting the maximality of $X$.
\end{proof}

\begin{claim}\label{claim:f6}
Every node $p$ of $H_1^+$ in $V^\dag(H_1)$ is non-admissible in $H_1^+$.
\end{claim}

\begin{proof}[Proof of Claim]
Suppose $p$ is a node with $N(p)=\{q,s,t\}$ and $H_1^+-p+st$ is $\mathcal{M}_{lc}$-connected. 
%Since $H_1^+$ is 3-connected, $H_1^+-e$ is nearly 3-connected so \cite{J&J} implies that $H_1^+-e$ is $\mathcal{M}_{lc}$-connected.
Then $G'-p+st$ is the 2-sum of $H_1^+-p+st$ and $H_2^+$ and $G'-p+st$ is $\mathcal{M}_{lc}$-connected by Lemma \ref{lem:feasible1}. It follows from Lemma \ref{lem:feasible2} that $G-p+st$ is $\mathcal{M}_{lc}$-connected, contradicting Claim \ref{claim:f5}.
\end{proof}

\begin{claim}\label{claim:f7}
$H_1^+$ is an $\mathcal{M}_{lc}$-circuit.
\end{claim}

\begin{proof}[Proof of Claim]
This follows immediately from Theorem \ref{thm:admissible_simple}(b), Claims \ref{claim:-2}, \ref{claim:f3} and \ref{claim:f6}, and the definition of $V^\dag(H_1)$.
%
%Suppose $H_1^+$ is not an $\mathcal{M}_{lc}$-circuit. Since $H_1^+$ is $\mathcal{M}_{lc}$-connected, there exists an $\mathcal{M}_{lc}$-circuit $H_1^+_1$ in $H_1^+$ which contains $uv$ and $\theta$. By Lemma \ref{lem:ear}(ii) we may extend $C_1=E(H_1)$ to an ear decomposition $C_1,C_2,\dots,C_t$ of $\mathcal{M}_{lc}(H)$. Since $E(H)-E^\dag(H)\subseteq \{uv,\theta\}$, Claim \ref{claim:f3} implies that $H_1^+-e$ is not $\mathcal{M}_{lc}$-connected for all but at most two edges of $H_1^+$.
%% since $E(H)\sm E^\dag(H)\subseteq \{uv,xy\}$. 
%Since $H_1^+$ is 3-connected, it follows from \cite[Theorem 5.4]{J&J} that there is an admissible node $p$ of $H_1^+$ in $V^\dag(H)$, contradicting Claim \ref{claim:f6}.
\end{proof}

\begin{claim}\label{claim:f8}
$H_1^+$ is isomorphic to $K_4$.
\end{claim}

\begin{proof}[Proof of Claim]
Suppose $H_1^+$ is not isomorphic to $K_4$. By Claim \ref{claim:f6}, no node of $H_1^+$ in $V^\dag(H_1)$ is admissible in $H_1^+$. Theorem \ref{thm:admissible_simple}(a), Claim \ref{claim:f7}  and the definition of $V^\dag(H_1)$ now imply that $G'=G-w+xy$, for some $x,y\in V(H_1)$ and $u,v,x,y$ are the only admissible nodes in $H_1^+$. We shall show that $x$ is a feasible node in $G$.

Since $x$ is an admissible node of $H_1^+$, $H_1^+-x+st$ is $\mathcal{M}_{lc}$-connected for some $s,t\in N_{H_1^+}(x)$. Let $N_{H_1^+}(x)=\{q,s,t\}$. Since $xy$ is an edge of $H_1^+$ and $y$ is a node of $H_1^+$, we must have $y\in \{s,t\}$. Without loss of generality suppose $y=t$. 
Since $G'-x+sy$ is the 2-sum of $H_1^+-x+sy$ and $H_2^+$, and $H_2^+$ is $\mathcal{M}_{lc}$-connected by Claim \ref{claim:-2}, Lemma \ref{lem:feasible1} implies that $G'-x+sy$ is $\mathcal{M}_{lc}$-connected. Since $G-x+sw$ is a 1-extension of $G'-x+sy$,
%since the property of being $\mathcal{M}_{lc}$-connected is preserved by 1-extension (
Lemma \ref{lem:feasible2} now implies  that $G-x+sw$ is $\mathcal{M}_{lc}$-connected.

Suppose $H_1^+-x+sy-\{u,v\}$ is disconnected. Since $H_1^+$ is 3-connected, $H_1^+-x+sy-\{u,v\}$ has two components $J_1,J_2$ with  $sy\in E(J_1)$ and $q\in V(J_2)$. Then $V(J_2)$ is properly contained in $X$.
Since $J_2$ is a loopless component of $(G-x+sw)-\{u,v\}$, this contradicts the minimality of $X$. Thus $H_1^+-x+sy-\{u,v\}$ is connected. 

Since $x$ is not feasible in $G$, $G-x+sw$ is not balanced. Let $T$ be an unbalanced 2-separator in $G-x+sw$. Since $G$ is balanced, $T$ separates $sw$ and $q$. 
We can apply Lemma \ref{lem:2setdeg4} to $G-x+sw$ and $G'$ respectively to deduce that  $T$ and $S=\{u,v\}$ are unbalanced 2-separators in $G''=G'-x+sy$. Since $G'$ is $\mathcal{M}_{lc}$-connected and contains a loop, Corollary \ref{cor:mlcredundant} implies that $G'$ is redundantly rigid, and hence 
$G'-xq$ is rigid. Since $G'-x$ is obtained from $G'-xq$ by deleting a vertex with two incident edges, it is rigid. Since $G''=G'-x+sy$ is obtained from $G'-x$ by an edge addition, it is also rigid. 
Lemma \ref{lem:crossing} now implies that $T\subseteq V(H_i)$ for some $i\in\{1,2\}$.
Since $(H_1^+-x+sy)-\{u,v\}$ is connected, $G''[X-x]$ is a component of $G''-S$. Since $q,s,y\in X\cup S$, and $T$ separates $sy$ and $q$ in $G''$, we have $T\cap X\neq \emptyset$. Hence $T\subseteq V(H_1)$.

Let $X_T$ be the vertex set of the component of $(G-x+sw)-T$ which is loopless. Since $G$ is balanced, each component of $H_2^+-S$ contains a loop and hence $X_T\cap (V(H_2)-S)=\emptyset$. This implies that $X_T\cup T \subseteq (X\cup S)+w-x$, so we may contradict the minimality of $X$ by showing that $w$ does not belong to $X_T\cup T$.  Note that $w\notin T$ by Lemma \ref{lem:2setdeg4}.
If $w$ is adjacent to some vertex $z\in V(H_2)\sm S$ then $z\notin X_T\cup T$. Since $zw\in E(G-x+sw)$ and $T$ is an unbalanced 2-separator, $w\notin X_T$.
%the facts that $w$ does not belong to $G'-x+sy$ and $T$ is an unbalanced 2-separator of $G-x+sw$ gives $w\notin V(X_T)$. 
Similarly, if $w$ is only adjacent to vertices of $H_1^+$, then $w$ is a node incident with a loop and we again have $w\notin X_T$.
%Hence some component $J'$ of $G''-T=G'-x+sy-T$ contains $V-(X\cup S)$. Let $J$ be the component of $G-x+sy-T$ which contains $V-X-S$. Then $V-X\subset V(J)\cup N_{G-x+sy}(J)$. Moreover, $w$ and $y$ are also contained in $V(J)\cup N_{G-x+sy}(J)$. This implies that the vertex set of the component of $G-x+sy-T$ which is distinct from $J$ is a proper subset of $X$. This contradicts the minimality of $X$.
\end{proof}

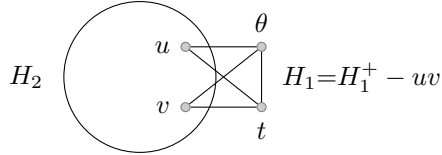
\begin{figure}[h]
\begin{center}
\begin{tikzpicture}[font=\small]
\draw (0,0) circle (1cm);
\node[] at (-1,0) (H2) [label=left:$H_2$]{};
\node[roundnode] at (0.6,0.4)(u) [label=left:$u$]{};
\node[roundnode] at (0.6,-0.4)(v) [label=left:$v$]{};

\node[roundnode] at (1.6,-0.4) (a) [label=below:$t$]{}
	edge[] (u)
	edge[] (v);
\node[roundnode] at (1.6,0.4) (b) [label=above:$\theta$]{}
	edge[] (u)
	edge[] (v)
	edge[] (a);
\node[] at (1.6,0) (H1) [label=right:$H_1\text{=}H_1^+-uv$]{};
\end{tikzpicture}
\end{center}
\caption{The structure of $G'$ in the proof of Claim \ref{claim:f9}.}
\label{fig:H1+_iso_K4}
\end{figure}

\begin{claim}\label{claim:f9}
$G'=G-w+xy$ for some $x,y\in V(H_1)$ and hence $\theta=xy\in E(H_1)$.
\end{claim}

\begin{proof}[Proof of Claim]
Suppose that the claim is false. Then $\theta$ is a vertex in $X$, $V(H_1)=\{u,v,\theta,t\}$ and $t$ is a node of $G$, see Figure \ref{fig:H1+_iso_K4}. We will show that $G-t+uv$ is balanced and $\mathcal{M}_{lc}$-connected. Note that $uv\notin E$ by Claim \ref{claim:-1}. Note also that $G-t+uv$ can be obtained from $H_2^+$ by either a sequence of one 1-extension and one edge addition (in the case that $G'=G-f$) or two 1-extensions 
%and one edge addition 
(in the case when $G'=G-w+z_1z_2$ for some $z_1,z_2\notin X$). Since $H_2^+$ is $\mathcal{M}_{lc}$-connected by Claim \ref{claim:-2}, it follows from Lemma \ref{lem:feasible2} that $G-t+uv$ is $\mathcal{M}_{lc}$-connected. Since $\theta$ is adjacent to $u$ and $v$ and $G$ is balanced there is no unbalanced 2-separation separating $\theta$ from $uv$ in $G-t+uv$. Thus $G-t+uv$ is balanced.
\end{proof}

\begin{claim}\label{claim:f10}
$X\neq \{x,y\}$.
\end{claim}

\begin{proof}[Proof of Claim]
Suppose that $X=\{x,y\}$. Then $x,y$ are nodes of $G$. We shall show that $G-x+wv$ is $\mathcal{M}_{lc}$-connected and balanced. Note that $wv\notin E$ since, if $w$ has a neighbour $z$ distinct from $x,y$, then $z\in V(H_2)\sm S$. Note further that $G-x+wv$ can be obtained from $H_2^+$ by a sequence of two 1-extensions. Since $H_2^+$ is $\mathcal{M}_{lc}$-connected by Claim \ref{claim:-2}, Lemma \ref{lem:feasible2} implies that $G-x+wv$ is $\mathcal{M}_{lc}$-connected. Suppose that $G-x+wv$ is not balanced. Since $G$ is balanced, there is an unbalanced 2-separator $T$ in $G-x+wv$ that separates $u$ and $wv$. Since $u,w$ and $v$ are all neighbours of $y$ in $G-x+wv$, we must have $y\in T$. Since  $y$ is a node in $G-x+wv$, this contradicts  Lemma \ref{lem:2setdeg4}. Thus $G-x+wv$ is balanced.
\end{proof}

We can now complete the proof of the theorem. Claims \ref{claim:f9} and \ref{claim:f10} allow us to assume, after relabelling if necessary, that $X=\{x,t\}$ and $v=y$. Thus $x$ is a node of $G$. We will show that $G-x+wt$ is $\mathcal{M}_{lc}$-connected and balanced. Note that $wt\notin E$ since the neighbour of $w$ distinct from $x,y$ belongs to $V-(X\cup S)$. Note further that $G-x+wt$ can be obtained from $H_2^+$ by a sequence of two 1-extensions. Since $H_2^+$ is $\mathcal{M}_{lc}$-connected by Claim \ref{claim:-2}, it follows from Lemma \ref{lem:feasible2} that $G-x+wt$ is $\mathcal{M}_{lc}$-connected. Suppose that $G-x+wt$ is not balanced. Since $G$ is balanced there is an unbalanced 2-separator $T$ in $G-x+wt$ separating $u$ and $wt$. 
Since $ut$ is an edge of $G-x+wt$ we must have $t\in T$. 
Since  $t$ is a node in $G-x+wt$, this contradicts  Lemma \ref{lem:2setdeg4}. Thus $G-x+wv$ is balanced.
\end{proof}

 Lemma \ref{lem:feasible2} and Lemma \ref{lem:2setdeg4} imply that the operations of edge/loop addition and 1-extension preserve the properties of being $\mathcal{M}_{lc}$-connected and balanced.  Combined with Theorem \ref{thm:feasible}, this immediately 
gives the following recursive construction. 

\begin{thm}\label{thm:balancedrecurse}
A looped simple graph is balanced and $\M_{lc}$-connected if and only if it can be obtained from $K_1^{[3]}$ by recursively applying the operations of performing a $1$-extension and adding a new edge or loop.
\end{thm}

Consider the balanced, $\M_{lc}$-connected graph $G$ drawn on the far right in Figure \ref{fig:rec_balanced}. We
gave a construction of $G$ from $K_1^{[3]}$ in Figure \ref{fig:rec_unbalanced}. However, the second step in this construction, where we use
$K_4$-extension, resulted in an unbalanced graph. In Figure \ref{fig:rec_balanced}, we show that
we can obtain $G$ from $K_1^{[3]}$ by using only 1-extensions and edge or loop additions.
\begin{figure}[h]
\begin{center}
\begin{tikzpicture}[font=\small]
\begin{scope}[yshift=0.5cm]
\node[roundnode] at (0,0) (v1) []{}
	edge[in=180,out=220,loop]()
	edge[in=230, out=270,loop]()
	edge[in=25,out=65,loop,very thick]();
\draw[->] (0.5,0)--(1,0);
\end{scope}
\begin{scope}[xshift=2cm]
\node[roundnode] at (0,0) (v1) []{}
	edge[in=180,out=220,loop]()
	edge[in=230, out=270,loop]();
\node[roundnode] at (0,1) (v4) []{}
	edge[very thick] (v1)
	edge[in=90, out=130,loop]()
	edge[in=140, out=180,loop]();
\draw[->] (0.5,0.5)--(1,0.5);
\end{scope}
\begin{scope}[xshift=4cm]
\node[roundnode] at (0,0) (v1) []{}
	edge[in=180,out=220,loop]()
	edge[in=230, out=270,loop]();
\node[roundnode] at (1,1) (v3) []{}
	edge[] (v1)
	edge[in=25, out=65,loop]();
\node[roundnode] at (0,1) (v4) []{}
	edge[very thick] (v3)
	edge[in=90, out=130,loop]()
	edge[in=140, out=180,loop]();
\draw[->] (1.5,0.5)--(2,0.5);
\end{scope}
\begin{scope}[xshift=7cm]
\node[roundnode] at (0,0) (v1) []{}
	edge[in=180,out=220,loop]()
	edge[in=230, out=270,loop]();
\node[roundnode] at (1,0) (v2) []{}
	edge[] (v1);
\node[roundnode] at (1,1) (v3) []{}
	edge[] (v1)
	edge[] (v2)
	edge[in=25, out=65,loop]();
\node[roundnode] at (0,1) (v4) []{}
	edge[] (v2)
	edge[dashed] (v1)
	edge[in=90, out=130,loop]()
	edge[in=140, out=180,loop]();
\draw[->] (1.5,0.5)--(2,0.5);
\end{scope}
\begin{scope}[xshift=10cm]
\node[roundnode] at (0,0) (v1) []{}
	edge[in=180,out=220,loop]()
	edge[in=230, out=270,loop]();
\node[roundnode] at (1,0) (v2) []{}
	edge[] (v1);
\node[roundnode] at (1,1) (v3) []{}
	edge[] (v1)
	edge[] (v2)
	edge[in=25, out=65,loop]();
\node[roundnode] at (0,1) (v4) []{}
	edge[] (v2)
	edge[] (v1)
	edge[dashed] (v3)
	edge[in=90, out=130,loop]()
	edge[in=140, out=180,loop]();
\draw[->] (1.5,0.5)--(2,0.5);
\end{scope}
\begin{scope}[xshift=13cm]
\node[roundnode] at (0,0) (v1) []{}
	edge[in=180,out=220,loop]()
	edge[in=230, out=270,loop]();
\node[roundnode] at (1,0) (v2) []{}
	edge[] (v1);
\node[roundnode] at (1,1) (v3) []{}
	edge[] (v1)
	edge[] (v2)
	edge[in=25, out=65,loop]();
\node[roundnode] at (0,1) (v4) []{}
	edge[] (v2)
	edge[] (v3)
	edge[] (v1)
	edge[in=90, out=130,loop]()
	edge[in=140, out=180,loop]();
\end{scope}
\end{tikzpicture}
\end{center}
\caption{An illustration of the recursive construction given in Theorem \ref{thm:balancedrecurse}.}
\label{fig:rec_balanced}
\end{figure}
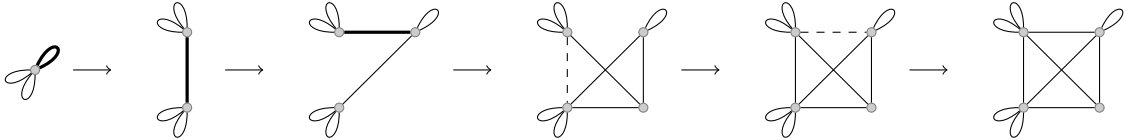

\section{Global rigidity} 
\label{sec:globalthm}

We will use our recursive construction to characterise generic global rigidity. We first need a lemma which shows that the properties of redundant rigidity and $\M_{lc}$-connectedness are equivalent for connected balanced graphs.

\begin{lem}\label{lem:Mconnected}
Let $G$ be a balanced looped simple graph. Then $G$ is $\M_{lc}$-connected if and only if $G$ is connected and redundantly rigid.
\end{lem}

\begin{proof}
Necessity follows from Corollary \ref{cor:mlcredundant} and the assumption that $G$ is $\M_{lc}$-connected.

To prove sufficiency, we suppose, for a contradiction that $G$ is connected and redundantly rigid but not $\M_{lc}$-connected. Let $H_1,H_2,\dots,H_m$ be the $\M_{lc}$-components of $G$. Let $V_i=V(H_i)$, $X_i=V_i \sm \bigcup_{j\neq i}V_j$ and $Y_i=V_i \sm X_i$. Since $G$ is connected, $|Y_i|\geq 1$, and since $G$ is balanced, $|Y_i|\geq 3$ when $H_i$ is loopless.

We may assume that $H_1,H_2,\dots,H_s$ are loopless and $H_{s+1},H_{s+2},\dots, H_m$ are not. Then 
\begin{eqnarray*}
r(G)&=&\sum_{i=1}^s (2|V_i|-3) +\sum_{i=s+1}^m 2|V_i| \\
&=& \sum_{i=1}^s(2|X_i|+2|Y_i|-3)+\sum_{i=s+1}^m(2|X_i|+2|Y_i|)\\
&\geq & \sum_{i=1}^m (2|X_i|+|Y_i|)+m-s,
\end{eqnarray*}
where the final inequality follows from the fact that $|Y_i|\geq 3$ for $1\leq i \leq s$ and $|Y_i|\geq 1$ for $s+1\leq i\leq m$. Since the $X_i$ are disjoint we have $\sum_{i=1}^m|X_i|=|\bigcup_{i=1}^mX_i|$. Also
%, since $G$ is connected, 
each element of $Y_i$ is contained in at least one other $Y_j$ with $j\neq i$. Hence we have $\sum_{i=1}^m|Y_i|\geq 2|\bigcup_{i=1}^mY_i|$. In addition the hypothesis that $G$ is balanced implies that at least one $H_i$ contains a loop so $m>s$. Hence
\begin{eqnarray*}
r(G)\geq 2(|\bigcup_{i=1}^mX_i|+|\bigcup_{i=1}^mY_i|)+m-s >2|V(G)|.
\end{eqnarray*}
%since at least one $H_i$ contains a loop.
This contradicts the fact that $r(G)\leq 2|V(G)|$.
\end{proof}

We may use Lemma \ref{lem:Mconnected} to restate Theorem \ref{thm:balancedrecurse} as

\begin{thm}\label{thm:fullrecurse}
A looped simple graph is balanced, connected and redundantly rigid if and only if it can be obtained from $K_1^{[3]}$ by recursively applying the operations of performing a $1$-extension and adding a new edge or loop.
\end{thm}

We can now characterise global rigidity for 2-dimensional generic linearly constrained frameworks.

\begin{thm}\label{thm:global_char_con}
Suppose $G$ is a connected looped simple graph with at least two vertices and $(G,p,q)$ is a generic realisation of $G$ as a linearly constrained framework in $\R^2$. Then the following statements are equivalent:\\
(a) $(G,p,q)$ is globally rigid;\\
(b) $G$ is balanced and redundantly rigid;\\
(c) $(G,p,q)$ has a full rank equilibrium stress.
\end{thm}

\begin{proof}
The implications (a) $\Rightarrow$ (b) and (c) $\Rightarrow$ (a) follow from  Theorems \ref{thm:glob_nec} and \ref{thm:stress}, respectively. It remains to prove that (b) $\Rightarrow$ (c). 
%The assumption that $G$ is redundantly rigid tells us that $G$ has at least four loops.
%To prove sufficiency we suppose that $G$ is
%balanced and each connected component of $G$ is either a single vertex with at least two loops or has at least two vertices and is redundantly rigid. We may assume further that $G$ is connected (since $(G,p,q)$ is globally rigid if and only if $(H,p|_H,q|_H)$ is globally rigid for each component $H$ of $G$) and that $G$ is not equal to $K_1^{[2]}$ . 
%
We use induction on the number of vertices of $G$ to show that $(G,p,q)$ has a full rank equilibrium stress whenever $G$ is connected, balanced and redundantly rigid. It is straighforward to check that 
every generic realisation of the 
smallest redundantly rigid looped simple graph $K_1^{[3]}$ 
has a full rank equilibrum stress (given by a $1\times 1$ stress matrix of rank zero).  The induction step now follows by using Lemma \ref{lem:stress_gen}, and Theorems \ref{thm:stress_ext} and \ref{thm:fullrecurse}. 
%Hence we see that $(G,p,q)$ has a full rank equilibrium stress. 
%
%We may now deduce that $(G,p,q)$ is globally rigid by applying Theorem \ref{thm:stress}.
%(Note that $G$ has at least two loops since it is rigid.)
\end{proof}

We can use Theorem \ref{thm:global_char_con} and the fact that a linearly constrained framework is globally rigid if and only if each of its connected components is globally rigid to deduce:
 
\begin{thm}\label{thm:global_char}
Suppose $(G,p,q)$ is a generic linearly constrained framework in
$\R^2$. Then $(G,p,q)$ is globally rigid if and only if $G$ is
balanced and each connected component of $G$ is either a single vertex with two loops or is redundantly rigid.
%either a single vertex with two loops or has at least two vertices and is redundantly rigid.
\end{thm}

%\begin{proof}
%Necessity follows from Theorem \ref{thm:glob_nec}. To prove sufficiency we suppose that $G$ is
%balanced and each connected component of $G$ is either a single vertex with at least two loops or has at least two vertices and is redundantly rigid. We may assume further that $G$ is connected (since $(G,p,q)$ is globally rigid if and only if $(H,p|_H,q|_H)$ is globally rigid for each component $H$ of $G$) and that $G$ is not equal to $K_1^{[2]}$ . 
%
%We use induction on the number of vertices of $G$ to show that $(G,p,q)$ has a full rank equilibrium stress whenever $G$ is connected, balanced and redundantly rigid. It is straighforward to check that the smallest redundantly rigid looped simple graph $K_1^{[3]}$ has a full rank equilibrum stress.  The induction step now follows by using Lemma \ref{lem:stress_gen}, and Theorems \ref{thm:stress_ext} and \ref{thm:fullrecurse}. 
%%Hence we see that $(G,p,q)$ has a full rank equilibrium stress. 
%
%We may now deduce that $(G,p,q)$ is globally rigid by applying Theorem \ref{thm:stress}.
%(Note that $G$ has at least two loops since it is rigid.)
%\end{proof}

Theorem \ref{thm:ST} implies that redundant rigidity can be checked efficiently by graph orientation or pebble game type algorithms \cite{B+J,L&S}. Since we can also check the property of being balanced in polynomial time, Theorem \ref{thm:global_char} gives rise to an efficient algorithm to decide whether a given looped simple graph is generically globally rigid in $\mathbb{R}^2$.

%\section{Concluding remark}
We conclude this section by mentioning a possible direction for future research.
As mentioned in the introduction, \cite{CGJN} gives a characterisation
of generic rigidity for linearly constrained frameworks in $\mathbb{R}^d$ when
each vertex is constrained to lie in an affine subspace of sufficiently small dimension compared to $d$. It would be interesting to obtain an analogous characterisation for generic global rigidity.

\section{The number of equivalent realisations}\label{sec:count}

We will extend  Theorem \ref{thm:global_char} by determining the number of distinct frameworks which are equivalent to a given generic linearly constrained framework $(G,p,q)$ when $G=(V,E,L)$ is a rigid $\M_{lc}$-connected looped simple graph. For $u,v\in V$, let $b(u,v)$ be the number of loopless connected components of $G-\{u,v\}$ and put $b(G)=\sum_{u,v\in V}b(u,v)$.  

\begin{thm}\label{thm:num}  Suppose $(G,p,q)$ is a generic linearly constrained framework in $\R^2$ and that $G$ is rigid and $\M_{lc}$-connected.
Then there are exactly $2^{b(G)}$ distinct frameworks which are equivalent to $(G,p,q)$.
\end{thm}

Our proof of Theorem \ref{thm:num} is similar to the proof of an analogous result for bar-joint frameworks \cite[Theorem 8.2]{JJS}. We will indicate below how the latter can be adapted to prove Theorem \ref{thm:num}.

We first need a result on generic points in $\R^n$. We will denote the algebraic closure of a field $\bK$ by $\overline \bK$.

\begin{lem}\label{genpoint1}
Let $f:\R^n\to \R^n$ by $f({\bf x})=
(f_1({\bf x}),f_2({\bf x}),\ldots,f_n({\bf x}))$, where $f_i({\bf x})$
is a polynomial with  coefficients in some extension field $\bK$ of $\bQ$ for all $1\leq i\leq n$.
Suppose that $\max_{{\bf x}\in \R^n}\{\mbox{rank }df|_{\bf x}\}=n$.
If either ${\bf x}$ or $f({\bf x})$ is generic over  $\bK$, then ${\bf x}$ and $f({\bf x})$ are both generic over  $\bK$ and $\overline{\bK({\bf x})}=\overline{\bK({f(\bf x}))}$.
\end{lem}
The proof of Lemma \ref{genpoint1} is the same as that for  \cite[Lemmas  3.1, 3.2]{JJS}. The only difference being that we work with polynomials with coefficients in $\bK$ rather than $\bZ$.

Let $G=(V,E,L)$ be a looped simple graph with  $E=\{e_1,e_2,\ldots,e_m\}$, $F=\{\ell_1,\ell_2,\ldots,\ell_s\}$ and $q:L\to \R^2$. The {\em rigidity map} $f_{G,q}:\R^{2|V|}\to \R^{|E\cup L|}$ is defined by putting 
$$f_{G,q}(p)=(f_{e_1}(p),\ldots,f_{e_m}(p),f_{\ell_1}(p)\ldots, 
f_{\ell_s}(p))$$
where $f_{e_i}(p)=\|p(u)-p(v)\|^2$ when $e_i=uv$ and $f_{\ell_j}(p)=q(\ell_j)\cdot p(v)$ when $\ell_j$ is incident to $v$.

\begin{lem}\label{genpoint2}
Let $(G,p,q)$ be a generic, rigid, linearly constrained framework in $\R^2$ and suppose that $(G,p',q)$ is an equivalent framework. Then $\overline{\bQ(p,q)}=\overline{\bQ(f_{H,q}(p))}=\overline{\bQ(p',q)}$.
\end{lem}
\begin{proof} This follows from Lemma \ref{genpoint1} by choosing a minimally rigid spanning subgraph $H$ of $G$ and then putting
$f=f_{H,q}$ and $\bK=\bQ(q)$.
\end{proof}

Given a  linearly constrained framework $(G,p,q)$ in $\R^2$, we say that two vertices $u,v$ of $G$ are {\em globally linked in $(G,p,q)$} if 
$\|p(u)-p(v)\|=\|p'(u)-p'(v)\|$ whenever $(G,p',q)$ is equivalent to $(G,p,q)$.

\begin{lem}\label{globlinked1}
Let $(G,p,q)$ be a generic linearly constrained framework in $\R^2$ and $v$ be a node of $G$ such that $N_G(v)=\{u,w,x\}$ and $G-v$ is rigid. Then  $u,w$ are globally linked in $(G,p,q)$.
\end{lem}

The proof of Lemma \ref{globlinked1} is the same as that of \cite[Lemma 4.1]{JJS}. The only difference being that we use Lemma \ref{genpoint2} instead of \cite[Lemmas 3.3, 3.4]{JJS}.

\begin{lem}\label{lem:globlinked2}
Suppose $G$ is a rigid, $\M_{lc}$-connected looped simple graph and $\{u,v\}$ is an unbalanced $2$-separation in $G$. Then $u,v$ are globally linked in every generic realisation of $G$ as a linearly constrained framework  in $\R^2$.
\end{lem}
\begin{proof}
We use induction on $|E(G)|$. Choose an unbalanced 2-separation $(H_1,H_2)$ of $G$ such that $H_2$ is simple, $u,v\in V(H_1)$ and $|V(H_2)|$ is as small as possible. Let $V(H_1)\cap V(H_2)=\{u',v'\}$. Then $H_2+u'v'$ is $3$-connected.  In addition  $H_1+u'v'$ is rigid, and $H_1+u'v'$ and $H_2+u'v'$ are both $\M_{lc}$-connected by Lemma \ref{lem:feasible1}. 

Suppose $H_2+u'v'\neq K_4$. By Theorem \ref{thm:admissible_simple}, $H_2+u'v'$ has an admissible edge $e$ distinct from $u'v'$ or an admissible node $w$ 
distinct from both $u'$ and $v'$.  Let $H_2'=H_2-e$ in the former case
and otherwise let $H_2'=H_2-w+xy$ be be obtained by performing an admissible 1-reduction at $w$. Then $G'=(H_1\cup H_2')-u'v'$ is rigid
and $\M_{lc}$-connected by Lemma \ref{lem:feasible1}. By induction, $u,v$ are globally linked in every generic realisation of $G'$. This immediately implies that  $u,v$ are globally linked in every generic realisation of $G$ if $G=G'+f$. Hence we may suppose that $G=G'-w+xy$. Since $G$ is redundantly rigid and $w$ is a node of $G$, $G-w$ is rigid. We can now use Lemma \ref{globlinked1} to deduce that $x, y$ are globally linked in $G$. The fact that $u,v$ are globally linked in every generic realisation of $G'$ now implies that $u,v$ are globally linked in every generic realisation of $G$.

It remains to consider the case when  $H_2+u'v'= K_4$. Choose $w\in V(H_2) \sm \{u',v'\}$. Since $G$ is redundantly rigid and $w$ is a node of $G$, $G-w$ is rigid. Lemma \ref{globlinked1} now implies  that $u',v'$ are globally linked in every generic realisation of $G$. If $\{u,v\}=\{u',v'\}$ then we are done so we may assume this is not the case. Then $\{u,v\}$ is an unbalanced 2-seperation of $H_1+u'v'$. Since $H_1+u'v'$ is rigid and $\M_{lc}$-connected by Lemma \ref{lem:feasible1}, we may use induction to deduce that $u,v$ are globally linked in every generic realisation of $H_1+u'v'$. The fact that
$u',v'$ are globally linked in every generic realisation of $G$, now implies that $u,v$ are globally linked in every generic realisation of $G$.
\end{proof}

\begin{proof}[Proof of Theorem \ref{thm:num}]

We use induction on $b(G)$. If $b(G)=0$ then the result follows from Lemma \ref{lem:Mconnected} and Theorem \ref{thm:global_char_con}. Hence we may suppose that $b(G)\geq 1$. Let $(H_1,H_2)$ be an unbalanced 2-separation in $G$ where $H_2$ is loopless and $|V(H_2)|$ is as small as possible. Let $V(H_1)\cap V(H_2)=\{u,v\}$. Then $H_2+uv$ is 3-connected and we have 
$b(G)=b(H_1+uv)+1$ by Lemma \ref{lem:crossing}. In addition, $H_1+uv$ is rigid,  
and $H_1+u'v'$ and $H_2+u'v'$ are both $\M_{lc}$-connected by Lemma \ref{lem:feasible1}. Since $(H_2+uv,p|_{H_2})$ is globally rigid as a bar-joint framework by Theorem \ref{thm:bar-joint} and $u,v$ are globally linked in $(G,p,q)$ by Lemma \ref{lem:globlinked2}, the number of linearly constrained  frameworks which are equivalent to $(G,p,q)$ is exactly twice the number of linearly constrained  frameworks which are equivalent to $(H_1+uv,p|_{H_1},q|_{H_1})$ (each equivalent framework to $(H_1+uv,p|_{H_1},q|_{H_1})$ gives rise to two equivalent frameworks to $(G,p,q)$ which are related by reflecting $H_2$ in the line through $u,v$). We can now use induction to deduce that the number of linearly constrained  frameworks which are equivalent to $(G,p,q)$ is $2\times 2^{b(H_1+uv)}=2^{b(G)}$.
\end{proof}

We close by noting that the problem of counting the number of non-congruent frameworks which are equivalent to a given generic bar-joint framework in $\R^2$ can be converted to that of  counting the number of distinct frameworks which are equivalent to a related linearly constrained  framework in $\R^2$. Given a simple graph $G$ we construct a looped simple graph $G^*$ by choosing an edge $uv$ of $G$ and adding two loops at both $u$ and $v$. It is not difficult to see that the number of distinct linearly constrained frameworks which are equivalent to a generic rigid $(G^*,p,q)$ is exactly twice the number of non-congruent bar-joint framework frameworks which are equivalent to $(G,p)$. In particular, we can use this construction to deduce Theorem \ref{thm:bar-joint} from Theorem \ref{thm:num}.

\subsection*{Acknowledgements.} The second and third authors would like to thank the London Mathematical Society for providing partial financial support for this research through a scheme 4 grant.

\end{document}